\documentclass[12pt]{amsart}
\usepackage{amsmath,amsthm,amsfonts,amssymb,mathrsfs}
\date{\today}

\usepackage{color}
\input xymatrix
\xyoption{all}

\usepackage{hyperref}

\newtheorem{theorem}{Theorem}[section]

\newtheorem{proposition}[theorem]{Proposition}
\newtheorem{corollary}[theorem]{Corollary}
\newtheorem{lemma}[theorem]{Lemma}
\theoremstyle{definition}
\newtheorem{example}[theorem]{Example}
\newtheorem{remark}[theorem]{Remark}

\newtheorem{definition}[theorem]{Definition}

\begin{document}

\title[On locally compact semitopological $0$-bisimple inverse $\omega$-semigroups]{On locally compact semitopological $0$-bisimple inverse $\omega$-semigroups}

\author{Oleg~Gutik}
\address{Faculty of Mathematics, National University of Lviv,
Universytetska 1, Lviv, 79000, Ukraine}
\email{o\_gutik@lnu.edu.ua,
ovgutik@yahoo.com}



\begin{abstract}
We describe the structure of Hausdorff locally compact semitopological $0$-bisimple inverse $\omega$-semigroups with compact maximal subgroups. In particular, we show that a Hausdorff locally compact semitopological $0$-bisimple inverse $\omega$-semigroup with a compact maximal subgroup is either compact or it is a topological sum of its $\mathscr{H}$-classes. We describe the structure of Hausdorff locally compact semitopological $0$-bisimple inverse $\omega$-semigroups with a monothetic maximal subgroups. We show the following dichotomy: a $T_1$  locally compact semitopological Reilly semigroup $\left(\textbf{B}(\mathbb{Z}_{+},\theta)^0,\tau\right)$ over the additive group of integers $\mathbb{Z}_{+}$, with adjoined zero and with a non-an\-ni\-hi\-la\-ting homomorphism $\theta$: the semigroup $\left(\textbf{B}(\mathbb{Z}_{+},\theta)^0,\tau\right)$ is either compact or discrete. At the end we discuss on the remainder under the closure of the discrete Reilly semigroup $\textbf{B}(\mathbb{Z}_{+},\theta)$ in a semitopological semigroup.
\end{abstract}

\maketitle

\section{Introduction and preliminaries}\label{section-1}

Further we shall follow the terminology of \cite{Carruth-Hildebrant-Koch-1983-1986, Clifford-Preston-1961-1967, Engelking-1989, Ruppert-1984}. By $\mathbb{N}$ (resp. $\mathbb{N}_0$) we denote the set of all positive (resp., non-negative) integers and by $\mathbb{Z}_{+}$ we denote the additive group of integers. We shall say that a set $A$ intersects almost all elements of a infinite family $\mathscr{K}=\{K_i\}_{i\in\Omega}$ of non-empty sets if $A\cap K_i=\varnothing$ only for finitely many elements of $\mathscr{K}$. All topological spaces, considered in this paper, are Hausdorff, if the otherwise is not stated explicitely.  If $A$ is a subset of a topological space $X$  then by $\operatorname{cl}_X(A)$ and $\operatorname{int}_X(A)$ we denote the closure and interior of $A$ in $X$, respectively.

A subset $A$ of a topological space $X$ is called \emph{regular open} if $A=\operatorname{int}_X(\operatorname{cl}_X(A))$.

We recall that a topological space $X$ is called
\begin{itemize}
  \item \emph{semiregular} if $X$ has a base consisting of regular open subsets;
  \item \emph{compact} if each open cover of $X$ has a finite subcover;
  \item \emph{locally compact} if each point of $X$ has an open neighbourhood with the compact closure;
  \item \emph{\v{C}ech-complete} if $X$ is Tychonoff and there exists a compactification $cX$ of $X$ such that the remainder $cX\setminus c(X)$ is an $F_\sigma$-set in $cX$.
\end{itemize}
It is well-known (see \cite[Section~3.9]{Engelking-1989}) and easy to show that every locally compact space is semiregular and \v{C}ech-complete.

Given a semigroup $S$, we shall
denote the set of idempotents of $S$ by $E(S)$.
A semigroup $S$ with the adjoined  zero  will be denoted by
$S^0$  (cf.
\cite{Clifford-Preston-1961-1967}).

A semigroup $S$ is called \emph{inverse} if for every $x\in S$ there exists a unique $x^{-1}\in S$ such that $xyx=x$ and $yxy=y$. In the following this element $y$ we shall denote by $x^{-1}$ and call \emph{inverse of} $x$. A map $\operatorname{\textsf{inv}}\colon S\to S$ which poses every $s\in S$ its inverse is called \emph{inversion}.

The subset of idempotents of a semigroup $S$ we shall denote by $E(S)$. If $S$ is an inverse semigroup then $E(S)$ is closed under multiplication and we shall refer to $E(S)$  as a \emph{band} (or the \emph{band of} $S$). If the band $E(S)$ is non-empty, then the semigroup operation on $S$ determines the following {\em natural partial order}  $\leqslant$ on $E(S)$: $e\leqslant f$ if and only if $ef=fe=e$.  A \emph{semilattice} is a commutative semigroup of idempotents. A semilattice $E$ is called {\em linearly ordered} or a \emph{chain} if its natural order is a linear order. A chain $L$ is called an \emph{$\omega$-chain} if $L$ is isomorphic to $\{0,-1,-2,-3,\ldots\}$ with the semigroup operation $x\cdot y=\min\{x,y\}$, thus the natural partial order on $\{0,-1,-2,-3,\ldots\}$ coincides with the usual order~$\leqslant$.

A \emph{semitopological} (\emph{topological}) \emph{semigroup} is a topological space with a separately continuous (jointly continuous) semigroup operation. An inverse topological semigroup with continuous inversion is called a \emph{topological inverse semigroup}. Also, a \emph{semitopological} (\emph{topological}) \emph{group} is a topological space with a separately continuous (jointly continuous) group operation (and inversion).

A topology $\tau$ on a semigroup $S$ is called:
\begin{itemize}
  \item \emph{semigroup} if $(S,\tau)$ is a topological semigroup;
  \item \emph{shift-continuous} if $(S,\tau)$ is a semitopological semigroup; \; and
  \item \emph{inverse semigroup} if $(S,\tau)$ is a topological inverse semigroup.
\end{itemize}

The \emph{bicyclic semigroup} ${\mathscr{C}}(p,q)$ is the semigroup with
the identity $1$ generated by two elements $p$ and $q$ subject
only to the condition $pq=1$.
The bicyclic monoid is a combinatorial bisimple $F$-inverse semigroup and it plays the
important role in the algebraic theory of semigroups and in the
theory of topological semigroups. For example well-known
Andersen's Theorem~\cite{Andersen-1952} states that\break a~($0$-)simp\-le
semigroup with an idempotent is completely ($0$-)simple if and only if it does not
contain the bicyclic semigroup. Eberhart and Selden showed that the bicyclic semigroup admits only the discrete semigroup topology and if a topological semigroup $S$ contains it as a dense subsemigroup then ${\mathscr{C}}(p,q)$ is an open subset of $S$~\cite{Eberhart-Selden-1969}. Bertman and  West in \cite{Bertman-West-1976} extended these results for semitopological semigroups. In \cite{Bardyla-2016, Bardyla-Gutik-2016, Chuchman-Gutik-2010, Chuchman-Gutik-2011, Fihel-Gutik-2011, {Gutik-2015}, Gutik-Maksymyk-2016??, Gutik-Pozdnyakova-2014, Gutik-Repovs-2011, Gutik-Repovs-2012, Mesyan-Mitchell-Morayne-Peresse-2016} these topologizability results were extended to generalizations of the bicyclic semigroup. Stable and $\Gamma$-compact topological semigroups do not contain the bicyclic semigroup~\cite{Anderson-Hunter-Koch-1965, Hildebrant-Koch-1988}. The problem of an embedding of the bicyclic monoid into compact-like topological semigroups was discussed in \cite{Banakh-Dimitrova-Gutik-2009, Banakh-Dimitrova-Gutik-2010, Gutik-Repovs-2007}.

In \cite{Eberhart-Selden-1969} Eberhart and Selden proved that if the bicyclic monoid ${\mathscr{C}}(p,q)$ is a dense subsemigroup of a topological monoid $S$ and $I=S\setminus{\mathscr{C}}(p,q)\neq\varnothing$ then $I$ is a two-sided ideal of the semigroup $S$. Also, there they described the closure of the bicyclic monoid ${\mathscr{C}}(p,q)$ in a locally compact topological inverse semigroup. The closure of the bicyclic monoid in a countably compact (pseudocompact) topological semigroups was studied in \cite{Banakh-Dimitrova-Gutik-2010}.

In the paper \cite{Gutik-2015} was shown the following dichotomy: a locally compact  semitopological bicyclic semigroup with an  adjoined zero is either compact or discrete.  Also there was shown a counterpart of this result for a locally compact semitopological bicyclic semigroup with an adjoined compact ideal. On the other hand, there was constructed a \v{C}ech-complete first countable topological inverse bicyclic semigroup with adjoined non-isolated zero which is not compact. The above dichotomy was extended by Bardyla in \cite{Bardyla-2016} to locally compact $\lambda$-polycyclic semitopological monoids and to locally compact semitopological interassociates of the bicyclic monoid \cite{Gutik-Maksymyk-2016}.

Let $S$ be a monoid and $\theta\colon S\to H_S(1)$ a homomorphism from $S$ into the group of units $H_S(1)$ of $S$. The set $\mathbb{N}_0\times S\times \mathbb{N}_0$ with the semigroup operation
\begin{equation*}
  (i, s, j)\cdot(m,t,n)=\left(i+m-\min\{j,m\},\theta^{m-\min\{j,m\}}(s)\theta^{j-\min\{j,m\}}(t),j+n-\min\{j,m\}\right),
\end{equation*}
where  $i,j,m,n\in \mathbb{N}_0$, $s,t\in S$ and $\theta^0$ is an identity map on $S$, is called the \emph{Bruck--Reilly extension} of the monoid $S$ and it will be denoted by $\textbf{BR}(S,\theta)$. We observe that if $S$ is a trivial monoid (i.e., $S$ is a singleton set), then $\textbf{BR}(S,\theta)$ is isomorphic to the bicyclic semigroup  ${\mathscr{C}}(p,q)$ and in case when $\theta\colon S\to H_S(1)$ is the annihilating homomorphism (i.e., $\theta(s)=1_S$ for all $s\in S$, where $1_S$ is the identity of the monoid $S$), then $\textbf{BR}(S,\theta)$ is called the \emph{Bruck semigroup} over the monoid $S$ \cite{Bruck-1958}.

For arbitrary $i,j\in \mathbb{N}_0$ and non-empty subset $A$ of the semigroup $S$ we define subsets $A_{i,j}$ and $A_{i,j}^0$ of $\textbf{BR}(S,\theta)$ as follows: $A_{i,j}=\left\{(i, s, j)\colon s\in A\right\}$ and $A_{i,j}^0=A_{i,j}\cup\left\{(i+1,1_S,j+1)\right\}$.

A regular semigroup $S$ is \emph{$\omega$-semigroup} if $E(S)$ is isomorphic to the $\omega$-chain and $S$ is \emph{bisimple} if $S$ constitutes a single $\mathscr{D}$-class.

The construction of the Bruck semigroup over a monoid was used in \cite{Bruck-1958} for the proof of the statement that every semigroup embeds into a simple monoid. Also Reilly and Warne prove that every bisimple regular $\omega$-semigroup $S$ is isomorphic to the Bruck--Reilly extension of some group \cite{Reilly-1966, Warne-1966}, i.e., for every bisimple regular $\omega$-semigroup $S$ there exists a group $G$ and a homomorphism $\theta\colon G\to G$ such that $S$ is isomorphic to the Bruck--Reilly extension $\textbf{BR}(G,\theta)$. A bisimple regular $\omega$-semigroup is called Reilly semigroup \cite{Petrich-1984}. Later similar as in the Petrich book \cite{Petrich-1984}, the Reilly semigroup $\textbf{BR}(G,\theta)$ we shall denote by $\textbf{B}(G,\theta)$.

Gutik constructed a topological counterpart of the Bruck construction in \cite{Gutik-1994}. Also in \cite{Gutik-1994, Gutik-1997} he studied the problem of a topologization of the Bruck semigroup over a topological semigroup. A.~Selden described the structure of locally compact bisimple regular topological $\omega$-semigroups and studied the closures of such semigroups in locally compact topological semigroups in \cite{Selden-1975, Selden-1976, Selden-1977}.
In the paper \cite{Gutik-Pavlyk-2009} topologizations of the Bruck--Reilly extension $\textbf{BR}(S,\theta)$ of a semitopological semigroup $S$ which turns $\textbf{BR}(S,\theta)$ into a semitopological semigroup are studied.

An inverse semigroup $S$ is called a \emph{$0$-bisimple $\omega$-semigroup} if $S$ has two $\mathscr{D}$-classes: $S\setminus \{0\}$ and $\{0\}$, and  $E(S)\setminus \{0\}$ is order isomorphic to the $\omega$-chain. The results of the Lallement and Petrich paper \cite{Lallement-Petrich-1969} imply that every $0$-bisimple inverse $\omega$-semigroup is isomorphic to a Reilly semigroup with adjoined zero $\textbf{B}(G,\theta)^0$.

In this paper we describe the structure of locally compact semitopological $0$-bisimple inverse $\omega$-semigroups with a compact maximal subgroup. In particular, we show that a locally compact semitopological $0$-bisimple inverse $\omega$-semigroup with a compact maximal subgroup is either compact or topologically isomorphic to the topological sum of its $\mathscr{H}$-classes. We describe the structure of locally compact semitopological $0$-bisimple inverse $\omega$-semigroups with a monothetic maximal subgroups. In particular we prove the dichotomy that a locally compact $T_1$-semitopological Reilly semigroup $\left(\textbf{B}(\mathbb{Z}_{+},\theta)^0,\tau\right)$ over the additive group of integers $\mathbb{Z}_{+}$, with adjoined zero and with a non-an\-ni\-hi\-la\-ting homomorphism $\theta$: the semigroup $\left(\textbf{B}(\mathbb{Z}_{+},\theta)^0,\tau\right)$ is either compact or discrete. At the end we discuss on the remainder under the closure of the discrete Reilly semigroup $\textbf{B}(\mathbb{Z}_{+},\theta)$ in a semitopological semigroup.

\section{On a topological Bruck--Reilly extension of a semitopological monoid}\label{section-2}

Later we need the following proposition  which is a simple generalization of Lemma~1.2 from \cite{Munn-Reilly-1966} and follows from Theorem~XI.1.1 of \cite{Petrich-1984}.

\begin{proposition}\label{proposition-2.1}
Let $S$ be an arbitrary monoid and  $\theta\colon S\to H_S(1)$ a homomorphism from $S$ into the group of units $H_S(1)$ of $S$. Then a map \emph{$\eta\colon \textbf{BR}(S,\theta)\to \mathscr{C}(p,q)$}, $(i, s, j)\mapsto q^ip^j$ is a homomorphism and hence the relation $\eta^\natural$ on \emph{$\textbf{BR}(S,\theta)$} defined in the following way
\begin{equation*}
  (i, s, j)\eta^\natural(m,t,n) \qquad \hbox{if and only if} \qquad i=m \quad \hbox{and} \quad j=n
\end{equation*}
is a congruence.
\end{proposition}

\begin{proposition}\label{proposition-2.2}
Let \emph{$\tau_{\textbf{BR}}$} be a topology on \emph{$\textbf{BR}(S,\theta)$} such that the set $S_{i,j}$ is open in \emph{$\left(\textbf{BR}(S,\theta),\tau_{\textbf{BR}}\right)$} for all $i,j\in \mathbb{N}_0$. Then $\eta^\natural$ is a closed congruence on \emph{$\left(\textbf{BR}(S,\theta),\tau_{\textbf{BR}}\right)$}.
\end{proposition}

\begin{proof}
Fix an arbitrary non-$\eta^\natural$-equivalent elements $(i, s, j),(m,t,n)\in\textbf{BR}(S,\theta)$. Then the definition of the relation $\eta^\natural$ implies that $S_{i,j}\cap S_{m,n}=\varnothing$. Since $S_{i,j}$ and $S_{m,n}$ are open disjoint neighbourhoods of the points $(i, s, j)$ and $(m,t,n)$ in the topological space $\left(\textbf{BR}(S,\theta),\tau_{\textbf{BR}}\right)$, Proposition~\ref{proposition-2.1} implies that $\eta^\natural$ is a closed congruence on \emph{$\left(\textbf{BR}(S,\theta),\tau_{\textbf{BR}}\right)$}.
\end{proof}

\begin{definition}\label{definition-2.3}
Let $\mathcal{C}$ be a class of semitopological semigroups and $(S,\tau_S)\in\mathcal{C}$. If $\tau_{\textbf{BR}}$ is a topology on $\textbf{BR}(S,\theta)$ such that $\left(\textbf{BR}(S,\theta),\tau_{\textbf{BR}}\right)\in\mathcal{C}$ and for some $i\in \mathbb{N}_0$ the subsemigroup $S_{i,i}$ with the topology restricted from $\left(\textbf{BR}(S,\theta),\tau_{\textbf{BR}}\right)$ is topologically isomorphic to $(S,\tau_S)$ under the map $\xi_i\colon S_{i,i}\to S$, $(i,s,i)\mapsto s$, then $\left(\textbf{BR}(S,\theta),\tau_{\textbf{BR}}\right)$ is called a \emph{topological Bruck--Reilly extension} of $(S,\tau_S)$ in the class $\mathcal{C}$.
\end{definition}

For all non-negative integers $i,j,m,n$ and an arbitrary Bruck--Reilly extension $\textbf{BR}(S,\theta)$ we define a map $\phi^{i,j}_{m,n}\colon \textbf{BR}(S,\theta)\to \textbf{BR}(S,\theta)$, $x\mapsto (m,1_S,i)\cdot x \cdot (j,1_S,n)$.

\begin{proposition}\label{proposition-2.4}
Let \emph{$\left(\textbf{BR}(S,\theta),\tau_{\textbf{BR}}\right)$} be a topological Bruck--Reilly extension of a semitopological semigroup $(S,\tau_S)$ in the class of semitopological semigroups. Then the following assertions hold:
\begin{itemize}
  \item[$(i)$] for all non-negative integers $i,j,m,n$ the restrictions $\phi^{i,j}_{m,n}|_{S_{i,j}}\colon S_{i,j}\to S_{m,n}$  and $\phi^{m,n}_{i,j}|_{S_{m,n}}\colon S_{m,n}\to S_{i,j}$  are mutually invertible homeomorphisms between subspaces $S_{i,j}$ and \emph{$S_{m,n}$ of $\left(\textbf{BR}(S,\theta),\tau_{\textbf{BR}}\right)$};

  \item[$(ii)$] for all non-negative integers $i,m$ the restrictions $\phi^{i,i}_{m,m}|_{S_{i,i}}\colon S_{i,i}\to S_{m,m}$ and $\phi^{m,m}_{i,i}|_{S_{m,m}}\colon S_{m,m}\to S_{i,i}$ are mutually invertible topological isomorphisms between semitopological subsemigroups $S_{i,i}$ and $S_{m,m}$ of \emph{$\left(\textbf{BR}(S,\theta),\tau_{\textbf{BR}}\right)$};
  \item[$(iii)$] the homomorphism $\theta\colon S\to H_S(1)$ is a continuous map;
  \item[$(iv)$] for all non-negative integers $i,j,m,n$ the restrictions $\phi^{i,j}_{m,n}|_{S_{i,j}^0}\colon S_{i,j}^0\to S_{m,n}^0$  and $\phi^{m,n}_{i,j}|_{S_{m,n}^0}\colon S_{m,n}^0\to S_{i,j}^0$  are mutually invertible homeomorphisms between subspaces $S_{i,j}^0$ and $S_{m,n}^0$ of \emph{$\left(\textbf{BR}(S,\theta),\tau_{\textbf{BR}}\right)$};

  \item[$(v)$] if $\theta\colon S\to H_S(1)$ is an annihilating homomorphism then for all non-negative integers $i,m$ the restrictions $\phi^{i,i}_{m,m}|_{S_{i,i}^0}\colon S_{i,i}^0\to S_{m,m}^0$ and $\phi^{m,m}_{i,i}|_{S_{m,m}^0}\colon S_{m,m}^0\to S_{i,i}^0$ are mutually invertible topological isomorphisms between semitopological subsemigroups $S_{i,i}^0$ and $S_{m,m}^0$ of \emph{$\left(\textbf{BR}(S,\theta),\tau_{\textbf{BR}}\right)$};
  \item[$(vi)$] for every positive integer $i$ the left and right shifts \emph{$\rho_{(i,1_S,i)}\colon \textbf{BR}(S,\theta)\to \textbf{BR}(S,\theta)$,  $x\mapsto x\cdot (i,1_S,i)$} and \emph{$\lambda_{(i,1_S,i)}\colon \textbf{BR}(S,\theta)\to \textbf{BR}(S,\theta)$, $x\mapsto(i,1_S,i)\cdot x$} are retractions, and hence \emph{$(i,1_S,i)\textbf{BR}(S,\theta)$}, \emph{$\textbf{BR}(S,\theta)(i,1_S,i)$} are closed subsets, and hence $S_{0,0}$ is an open subset of \emph{$\left(\textbf{BR}(S,\theta),\tau_{\textbf{BR}}\right)$};

  \item[$(vii)$] for every positive integer $i$ the sets $S_{0,i}$ and $S_{i,0}$ are open in \emph{$\left(\textbf{BR}(S,\theta),\tau_{\textbf{BR}}\right)$};

  \item[$(viii)$] for any non-negative integer $k$ and for every positive integer $i$ the sets $S_{0,k}\cup S_{1,k+1}\cup\cdots\cup S_{i,k+i}$ and $S_{k,0}\cup S_{k+1,1}\cup\cdots\cup S_{k+i,i}$ are open in \emph{$\left(\textbf{BR}(S,\theta),\tau_{\textbf{BR}}\right)$}.
\end{itemize}
\end{proposition}

\begin{proof}
Assertions $(i)$ and $(ii)$ follow from the definition of the semigroup $\textbf{BR}(S,\theta)$ and the separate continuity of the semigroup operation in $\left(\textbf{BR}(S,\theta),\tau_{\textbf{BR}}\right)$.

$(iii)$ Since $\left(i,s,i\right)\cdot\left(i+1,1_S,i+1\right)=\left(i+1,\theta(s),i+1\right)$ for any $s\in S$, assertions $(i)$ and $(ii)$, and the separate continuity of the semigroup operation of $\left(\textbf{BR}(S,\theta),\tau_{\textbf{BR}}\right)$ imply that the homomorphism $\theta\colon S\to H_S(1)$ is a continuous map, because the following diagram
\begin{equation*}
\xymatrix{
S_{i,i}\ar[rrr]^{\rho_{(i+1,1_S,i+1)}} &&& S_{i+1,i+1}\ar[dd]^{\xi_{i+1}}\\
&&&\\
S\ar[rrr]_{\theta}\ar[uu]^{\xi_i^{-1}} &&& S
}
\end{equation*}
commutes, where $\rho_{(i+1,1_S,i+1)}$ is the right shift on the semigroup $\textbf{BR}(S,\theta)$ by the element $(i+1,1_S,i+1)$..

$(iv)$ Since $\left(n,1_S,i\right)(i+1,1_S, j+1)(j,1_S,m)=(n+1,1_S,m+1)$,
assertion $(iv)$ follows from the definition of the semigroup $\textbf{BR}(S,\theta)$, the separate continuity of the semigroup operation in $\left(\textbf{BR}(S,\theta),\tau_{\textbf{BR}}\right)$ and $(iii)$.

$(v)$ If $\theta\colon S\to H_S(1)$ is the annihilating homomorphism then it is obvious that $S_{i,i}^0$ is a subsemigroup of $\textbf{BR}(S,\theta)$ for any non-negative integer $i$ such that $(i+1,1_S,i+1)$ is zero of $S_{i,i}^0$, and next we apply item $(iv)$.

$(vi)$ The semigroup operation of $\textbf{BR}(S,\theta)$ implies that
\begin{equation*}
  (i,1_S,i)\cdot(m,s,n)=\left(m, s,n\right) \quad \hbox{and} \quad (m,s,n)\cdot(i,1_S,i)=\left(m, s,n\right), \qquad \hbox{for all integers} \quad m,n\geqslant i,
\end{equation*}
and
\begin{equation*}
\begin{split}
  (i,1_S,i)\cdot(m,s,n)= & \;(i,\theta^{i-m}(s),i+n-m)\neq\left(m, s,n\right) \quad \hbox{and} \\
    (m,s,n)\cdot(i,1_S,i)= &\;(i-n+m,\theta^{i-n}(s),i)\neq\left(m, s,n\right), \qquad \hbox{for all non-negative integers} \quad m,n< i.
\end{split}
\end{equation*}
Then the above arguments, \cite[Exercise.~1.5.C]{Engelking-1989}, Hausdorffness of $\left(\textbf{BR}(S,\theta),\tau_{\textbf{BR}}\right)$ and the separate continuity of the semigroup operation of $\left(\textbf{BR}(S,\theta),\tau_{\textbf{BR}}\right)$ imply our assertion.

$(vii)$ Item $(vi)$ implies that $S_{0,0}\cup S_{0,1}\cup\cdots\cup S_{0,i}$ and $S_{0,0}\cup S_{1,0}\cup\cdots\cup S_{i,0}$ are open subsets of $\left(\textbf{BR}(S,\theta),\tau_{\textbf{BR}}\right)$. Then  $S_{0,i}=(S_{0,0}\cup S_{0,1}\cup\cdots\cup S_{0,i})\cap \rho_{(i,1_S,0)}^{-1}(S_{0,0})$ is an open subset of $\left(\textbf{BR}(S,\theta),\tau_{\textbf{BR}}\right)$, as well. Similar arguments show that $S_{i,0}$ is an open subset of $\left(\textbf{BR}(S,\theta),\tau_{\textbf{BR}}\right)$, too.

$(viii)$ Fix an arbitrary non-negative integer $k$, an arbitrary positive integer $i$ and any element $(k+i,s,i)\in S_{k+i,i}$. Then  $(k+i,s,i)\cdot(i,1_S,0)=(k+i,s,0)$ and hence the separate continuity of the semigroup operation in $\left(\textbf{BR}(S,\theta),\tau_{\textbf{BR}}\right)$ and assertion  $(vii)$ imply that there exists an open neighbourhood $V$ of the point $(k+i,s,i)$ in $\left(\textbf{BR}(S,\theta),\tau_{\textbf{BR}}\right)$ such that $V\cdot (i,1_S,0)\subseteq S_{k+i,0}$. Without loss of generality we may assume that
\begin{equation*}
V\subseteq \textbf{BR}(S,\theta)\setminus\left((k+i+1,1_S,k+i+1)\textbf{BR}(S,\theta)\cup \textbf{BR}(S,\theta)(i+1,1_S,i+1)\right).
\end{equation*}
Fix an arbitrary $(m,t,n)\in V$. Then we have that $m\leqslant k+i$ and $n\leqslant i$. This implies that
\begin{equation*}
  (m,t,n)\cdot(i,1_S,0)=\left(m+i-n,\theta^{i-n}(t),0\right)\in S_{k+i,0},
\end{equation*}
and hence $m=k+n$. This implies that $S_{k,0}\cup S_{k+1,1}\cup\cdots\cup S_{k+i,i}=\rho_{(i,1_S,0)}(S_{k+i,0})$ is an open subset of $\left(\textbf{BR}(S,\theta),\tau_{\textbf{BR}}\right)$, because all shifts are continuous in $\left(\textbf{BR}(S,\theta),\tau_{\textbf{BR}}\right)$. The proof of openness of the set $S_{0,k}\cup S_{1,k+1}\cup\cdots\cup S_{i,k+i}$ in $\left(\textbf{BR}(S,\theta),\tau_{\textbf{BR}}\right)$ is similar.
\end{proof}

Proposition~\ref{proposition-2.4} and Theorem~3.3.8 of \cite{Engelking-1989} imply the following

\begin{corollary}\label{corollary-2.5}
Every semitopological bisimple inverse $\omega$-semigroup $S$ is topologically isomorphic to a topological Bruck--Reilly extension \emph{$\left(\textbf{B}(G,\theta),\tau_{\textbf{B}}\right)$} of a semitopological group $(G,\tau_G)$ in the class of semitopological semigroups. Moreover, if $S$ is locally compact then $(G,\tau_G)$ is a locally compact topological group.
\end{corollary}

We observe that the continuity of the group operation and inversion in a locally compact semitopological group $(G,\tau_G)$ follows from well-known Ellis' Theorem (see \cite[Theorem~2]{Ellis-1957}).

The following corollary describes the structure of $0$-bisimple inverse $\omega$-semigroups and it follows from Theorem~4.2 of \cite{Munn-1970} (also see corresponding statements in \cite{Lallement-Petrich-1969} and \cite{McAlister-1974}).

\begin{corollary}\label{corollary-2.6}
Every $0$-bisimple inverse $\omega$-semigroup $S$ is isomorphic to a Reilly semigroup with adjoined zero \emph{$\textbf{B}(G,\theta)^0=\textbf{B}(G,\theta)\sqcup\{\textsf{0}\}$}.
\end{corollary}

Later in this paper by $\textsf{0}$ we shall denote zero of a Reilly semigroup with adjoined zero $\textbf{B}(G,\theta)^0$.

By Theorem~3.3.8 from \cite{Engelking-1989} every open subspace of a locally compact space is locally compact and hence Corollaries~\ref{corollary-2.5} and~\ref{corollary-2.6} imply the following theorem.

\begin{theorem}\label{theorem-2.7}
Every semitopological $0$-bisimple inverse $\omega$-semigroup $S$ is topologically isomorphic to a topological Bruck--Reilly extension \emph{$\left(\textbf{B}(G,\theta)^0,\tau_{\textbf{B}}^0\right)$} of a semitopological group $(G,\tau_G)$ with adjoined zero (not necessarily as an isolated point) in the class of semitopological semigroups. Moreover, if $S$ is locally compact then $(G,\tau_G)$ is a locally compact topological group.
\end{theorem}

\section{On locally compact semitopological $0$-bisimple inverse $\omega$-semigroups with compact maximal subgroups}\label{section-3}

In this section we describe the structure of locally compact semitopological $0$-bisimple inverse $\omega$-semigroups with compact maximal subgroups.

\begin{lemma}\label{lemma-3.1}
Let $S$ be a semitopological $0$-bisimple inverse $\omega$-semigroup with compact maximal subgroups. Then every non-zero $\mathscr{H}$-class of $S$ is an open-and-closed subset of $S$.
\end{lemma}

\begin{proof}
By Theorem~\ref{theorem-2.7} there exist a compact semitopological group (which by the Ellis Theorem is a topological group \cite{Ellis-1957}) $(G,\tau_G)$ and a continuous homomorphism $\theta\colon G\to G$ such that the semitopological semigroup $S$ is topologically isomorphic to a topological Bruck--Reilly extension $\left(\textbf{B}(G,\theta)^0,\tau_{\textbf{B}}^0\right)$ of $(G,\tau_G)$ with an adjoined zero. It is obvious that every non-zero $\mathscr{H}$-class of the semigroup $\textbf{B}(G,\theta)^0$ coincides with $G_{i,j}$ for some $i,j\in \mathbb{N}_0$. Then items $(ii)$ and $(vi)$ of Proposition~\ref{proposition-2.4} imply the assertion of the lemma.
\end{proof}

For an arbitrary group $G$ and a homomorphism $\theta\colon G\to G$ we define a map $\eta\colon \textbf{BR}(G,\theta)^0\to \mathscr{C}^0={\mathscr{C}}(p,q)\cup\{0\}$ by the formulae $\eta(i, g, j)=q^ip^j$ and $\eta(\textsf{0})=0$, for $g\in G$ and $i,j\in \mathbb{N}_0$. Then using Proposition~\ref{proposition-2.1} we can show that the map $\eta$ is a homomorphism and hence the relation $\eta^\natural$ on $\textbf{BR}(G,\theta)$ defined in the following way
\begin{equation}\label{eq-3.1}
  (i, s, j)\eta^\natural(m,t,n) \quad \hbox{if and only if} \quad i=m \quad \hbox{and} \quad j=n, \qquad \hbox{and}\qquad \textsf{0}\eta^\natural \textsf{0}
\end{equation}
is a congruence on the semigroup $\textbf{BR}(G,\theta)^0$.

\begin{lemma}\label{lemma-3.2}
Let \emph{$\left(\textbf{B}(G,\theta)^0,\tau_{\textbf{B}}^0\right)$} be a semitopological semigroup with a compact group of units. Then $\eta^\natural$ is a closed congruence on \emph{$\left(\textbf{B}(G,\theta)^0,\tau_{\textbf{B}}^0\right)$}.
\end{lemma}

\begin{proof}
Fix arbitrary non-$\eta^\natural$-equivalent non-zero elements $(i, s, j)$ and $(m,t,n)$ of the Reilly semigroup with adjoined zero $\textbf{B}(G,\theta)^0$. Then $G_{i,j}$ and $G_{m,n}$ are open-and-closed neighbourhoods of the points $(i, s, j)$ and $(m,t,n)$ in the space $\left(\textbf{B}(G,\theta)^0,\tau_{\textbf{B}}^0\right)$, respectively, such that $\eta^\natural\cap\left(G_{i,j}\times G_{m,n}\right)=\varnothing$. Also, the above arguments imply that $G_{i,j}\times\left(\textbf{B}(G,\theta)^0\setminus G_{i,j}\right)$ is an open-and-closed neighbourhood of the ordered pair $\left((i, s, j),\textsf{0}\right)$ in $\textbf{B}(G,\theta)^0\times \textbf{B}(G,\theta)^0$ with the product topology which does not intersect the congruence $\eta^\natural$ of semigroup $\textbf{B}(G,\theta)^0$. Hence we get that $\eta^\natural$ is a closed congruence on the semitopological semigroup $\left(\textbf{B}(G,\theta)^0,\tau_{\textbf{B}}^0\right)$.
\end{proof}

Later we shall need the following notions. A continuous map $f\colon X\to Y$ from a topological space $X$ into a topological space $Y$ is called:
\begin{itemize}
  \item[$\bullet$] \emph{quotient} if the set $f^{-1}(U)$ is open in $X$ if and only if $U$ is open in $Y$ (see \cite{Moore-1925} and \cite[Section~2.4]{Engelking-1989});
  \item[$\bullet$] \emph{hereditarily quotient} (or \emph{pseudoopen}) if for every $B\subset Y$ the restriction $f|_{B}\colon f^{-1}(B)\to B$ of $f$ is a quotient map (see \cite{McDougle-1958, McDougle-1959, Arkhangelskii-1963} and \cite[Section~2.4]{Engelking-1989});
  \item[$\bullet$] \emph{open} if $f(U)$ is open in $Y$ for every open subset $U$ in $X$;
  \item[$\bullet$] \emph{closed} if $f(F)$ is closed in $Y$ for every closed subset $F$ in $X$;
  \item[$\bullet$] \emph{perfect} if $X$ is Hausdorff, $f$ is a closed map and all fibers $f^{-1}(y)$ are compact subsets of $X$ \cite{Vainstein-1947}.
\end{itemize}
Every perfect map is closed, every closed map   and every hereditarily quotient map are quotient \cite{Engelking-1989}. Moreover a continuous map $f\colon X\to Y$ from a topological space $X$ onto a topological space $Y$ is hereditarily quotient if and only if for every $y\in Y$ and every open subset $U$ in $X$ which contains $f^{-1}(y)$ we have that $y\in\operatorname{int}_Y(f(U))$ (see \cite[2.4.F]{Engelking-1989}).

\begin{lemma}\label{lemma-3.4}
Let \emph{$\left(\textbf{B}(G,\theta)^0,\tau_{\textbf{B}}^0\right)$} be a semitopological semigroup with a compact group of units. Then the quotient natural homomorphism $\eta\colon \textbf{B}(G,\theta)^0\to \mathscr{C}^0$ is an open map.
\end{lemma}

\begin{proof}
If $U$ is an open subset of $\left(\textbf{B}(G,\theta)^0,\tau_{\textbf{B}}^0\right)$ such that $U\not\ni \textsf{0}$ then $\eta(U)$ is an open subset of $\mathscr{C}^0$, because $\mathscr{C}(p,q)$ is a discrete open subset of the space $\mathscr{C}^0$ (see \cite[Proposition~1]{Bertman-West-1976}).

Suppose $U\ni \textsf{0}$ is an open subset of $\left(\textbf{B}(G,\theta)^0,\tau_{\textbf{B}}^0\right)$. Put $U^*=\eta^{-1}\left(\eta(U)\right)$. Then $U^*=\eta^{-1}\left(\eta(U^*)\right)$. Since $\eta\colon \textbf{B}(G,\theta)^0\to \mathscr{C}^0$ is a natural homomorphism,
\begin{equation*}
  U^*=\bigcup\left\{G_{i,j}\colon G_{i,j}\cap U\neq\varnothing\right\}\cup \{\textsf{0}\}.
\end{equation*}
Since every $0$-bisimple inverse $\omega$-semigroup is isomorphic to a Reilly semigroup with adjoined zero, the last equality and Lemma~\ref{lemma-3.1} imply that $U^*$ is an open subset of the space $\left(\textbf{B}(G,\theta)^0,\tau_{\textbf{B}}^0\right)$, and since $\eta$ is a quotient map and $U^*=\eta^{-1}\left(\eta(U^*)\right)$, we conclude that $\eta(U)$ is an open subset of the space $\mathscr{C}^0$.
\end{proof}

The following simple example from the paper \cite{Gutik-2015} shows that  the semigroup $\mathscr{C}^0$ admits a topology $\tau_{\operatorname{\textsf{Ac}}}$ making it a compact semitopological semigroup.

\begin{example}[\cite{Gutik-2015}]\label{example-3.5}
On the semigroup $\mathscr{C}^0$ we define a topology $\tau_{\operatorname{\textsf{Ac}}}$ in the following way:
\begin{itemize}
  \item[$(i)$] every element of the bicyclic monoid ${\mathscr{C}}(p,q)$ is an isolated point in the space $(\mathscr{C}^0,\tau_{\operatorname{\textsf{Ac}}})$;
  \item[$(ii)$] the family $\mathscr{B}(0)=\left\{U\subseteq \mathscr{C}^0\colon U\ni 0 \hbox{~and~} {\mathscr{C}}(p,q)\setminus U \hbox{~is finite}\right\}$ is a base of the topology $\tau_{\operatorname{\textsf{Ac}}}$ at zero $0\in\mathscr{C}^0$,
\end{itemize}
i.e., $\tau_{\operatorname{\textsf{Ac}}}$ is the topology of the Alexandroff one-point compactification of the discrete space ${\mathscr{C}}(p,q)$ with the remainder $\{0\}$. The semigroup operation in $(\mathscr{C}^0,\tau_{\operatorname{\textsf{Ac}}})$ is separately continuous, because all elements of the bicyclic semigroup ${\mathscr{C}}(p,q)$ are isolated points in the space $(\mathscr{C}^0,\tau_{\operatorname{\textsf{Ac}}})$ and left
and right translations are finite-to-one functions in ${\mathscr{C}}(p,q)$ (see \cite[Lemma I.1]{Eberhart-Selden-1969}).
\end{example}

\begin{lemma}\label{lemma-3.6}
Let \emph{$\left(\textbf{B}(G,\theta)^0,\tau_{\textbf{B}}^0\right)$} be a locally compact semitopological semigroup with a compact group of units. Then the quotient semigroup \emph{$\textbf{B}(G,\theta)^0/\eta^\natural$} with the quotient topology is topologically isomorphic to the semigroup $\mathscr{C}^0$ with either the topology $\tau_{\operatorname{\textsf{Ac}}}$ or the discrete topology.
\end{lemma}

\begin{proof}
By Lemma~\ref{lemma-3.1} every non-zero $\mathscr{H}$-class of the semigroup $\textbf{B}(G,\theta)^0$ is an open-and-closed subset of $\left(\textbf{B}(G,\theta)^0,\tau_{\textbf{B}}^0\right)$ and hence the quotient space $\left(\textbf{B}(G,\theta)^0,\tau_{\textbf{B}}^0\right)/\eta^\natural$ is Hausdorff. Lemma~\ref{lemma-3.4} implies that $\eta$ is an open map. By Theorem~3.3.15 from \cite{Engelking-1989}, $\textbf{B}(G,\theta)^0/\eta^\natural$ with the quotient topology is a locally compact space. Then Theorem~1 of \cite{Gutik-2015} implies the statement of our lemma.
\end{proof}

Lemma~\ref{lemma-3.6} implies the following proposition.

\begin{proposition}\label{proposition-3.7}
Let \emph{$\left(\textbf{B}(G,\theta)^0,\tau_{\textbf{B}}^0\right)$} be a locally compact semitopological semigroup with a compact group of units and an isolated zero \emph{$\textsf{0}$}. Then the topological space \emph{$\left(\textbf{B}(G,\theta)^0,\tau_{\textbf{B}}^0\right)$} is homeomorphic to topological sum of its $\mathscr{H}$-classes with induced from \emph{$\left(\textbf{B}(G,\theta)^0,\tau_{\textbf{B}}^0\right)$} topologies, i.e.,
\emph{\begin{equation*}
\left(\textbf{B}(G,\theta)^0,\tau_{\textbf{B}}^0\right)=\bigoplus\left\{G_{i,j}\colon i,j\in \mathbb{N}_0\right\}\oplus\{\textsf{0}\}.
\end{equation*}}
\end{proposition}

\begin{lemma}\label{lemma-3.8}
Let \emph{$\left(\textbf{B}(G,\theta)^0,\tau_{\textbf{B}}^0\right)$} be a locally compact semitopological semigroup with compact maximal subgroups and a non-isolated zero \emph{$\textsf{0}$}. Then for every open neighbourhood \emph{$U_\textsf{0}$} of zero in \emph{$\left(\textbf{B}(G,\theta)^0,\tau_{\textbf{B}}^0\right)$} there are only  finitely many non-zero $\mathscr{H}$-classes $G_{i,j}$, $i,j\in \mathbb{N}_0$, non-intersecting with  \emph{$U_\textsf{0}$}.
\end{lemma}

\begin{proof}
Suppose to the contrary that there exists an open neighbourhood $U_\textsf{0}$ of zero in \emph{$\left(\textbf{B}(G,\theta)^0,\tau_{\textbf{B}}^0\right)$} such that $G_{i,j}\cap U_\textsf{0}=\varnothing$ for infinitely many non-zero $\mathscr{H}$-classes $G_{i,j}$, $i,j\in \mathbb{N}_0$. Then by Lemma~\ref{lemma-3.4} the quotient natural homomorphism $\eta\colon \textbf{B}(G,\theta)^0\to \mathscr{C}^0$ is an open map, and hence the quotient semigroup \emph{$\textbf{B}(G,\theta)^0/\eta^\natural$} with the quotient topology is neither compact nor discrete, which contradicts Lemma~\ref{lemma-3.6}.
\end{proof}

For each subset $A$ of $\left(\textbf{B}(G,\theta)^0,\tau_{\textbf{B}}^0\right)$ and each $i,j\in \mathbb{N}_0$ put $\left[A\right]_{i,j}=A\cap G_{i,j}$.

\begin{lemma}\label{lemma-3.9}
Let \emph{$\left(\textbf{B}(G,\theta)^0,\tau_{\textbf{B}}^0\right)$} be a locally compact semitopological semigroup with a compact group of units and  non-isolated zero \emph{$\textsf{0}$}. Then for every open neighbourhood \emph{$U_\textsf{0}$} of zero in \emph{$\left(\textbf{B}(G,\theta)^0,\tau_{\textbf{B}}^0\right)$} and every non-negative integer $i_0$ the sets
\begin{equation*}
  \left\{j\in \mathbb{N}_0\colon G_{i_0,j}\nsubseteq U_\emph{\textsf{0}}\right\} \qquad \hbox{and} \qquad \left\{j\in \mathbb{N}_0\colon G_{j,i_0}\nsubseteq U_\emph{\textsf{0}}\right\}
\end{equation*}
are finite.
\end{lemma}

\begin{proof}
Suppose to the contrary that there exist a open neighbourhood $U_\textsf{0}$ of zero in $\left(\textbf{B}(G,\theta)^0,\tau_{\textbf{B}}^0\right)$ and a non-negative integer $i_0$ such that the set $\left\{j\in \mathbb{N}_0\colon G_{i_0,j}\nsubseteq U_\textsf{0}\right\}$ is infinite. Since the topological space  $\left(\textbf{B}(G,\theta)^0,\tau_{\textbf{B}}^0\right)$ is locally compact, we can take a regular open neighbourhood $U_\textsf{0}$ of the zero
with compact closure.

We consider the following two cases:
\begin{itemize}
  \item[$(1)$] there exists a non-negative integer $j_0$ such that $\left[U_\textsf{0}\right]_{i_0,j}\neq G_{i_0,j}$ for all $j\geqslant j_0$;
  \item[$(2)$] for every positive integer $k$ there exists a positive integer $n>k$ such that $\left[U_\textsf{0}\right]_{i_0,n}= G_{i_0,n}$.
\end{itemize}

Suppose case $(1)$ holds. Since every maximal subgroup of $\left(\textbf{B}(G,\theta)^0,\tau_{\textbf{B}}^0\right)$ is compact, by  Proposition~\ref{proposition-2.4}$(ii)$  without loss of generality we may assume that $j_0=0$. By Lemma~\ref{lemma-3.1} every $\mathscr{H}$-class $G_{i,j}$ is an open subset of $\left(\textbf{B}(G,\theta)^0,\tau_{\textbf{B}}^0\right)$ and hence the set
\begin{equation*}
  \mathfrak{H}_{i_0}^0(U_\textsf{0})=\{\textsf{0}\}\cup \bigcup_{j\in \mathbb{N}_0}\left[\operatorname{cl}_{\textbf{B}(G,\theta)^0}(U_\textsf{0})\right]_{i_0,j}
\end{equation*}
is compact. By Lemma~\ref{lemma-3.1} the family $\mathscr{U}_{i_0}=\left\{\left\{U_\textsf{0}\right\},\left\{G_{i_0,j}\colon j\in \mathbb{N}_0\right\}\right\}$ is an open cover of the compactum $\mathfrak{H}_{i_0}^0(U_\textsf{0})$ and hence there exists a positive integer $j_1$ such that $\left[U_\textsf{0}\right]_{i_0,n}=\left[\operatorname{cl}_{\textbf{B}(G,\theta)^0}(U_\textsf{0})\right]_{i_0,n}$ for all $n\geqslant j_1$. Since $\left(\textbf{B}(G,\theta)^0,\tau_{\textbf{B}}^0\right)$ is a semitopological semigroup, a set $V_\textsf{0}=\rho_{(1,1_G,0)}^{-1}( U_\textsf{0})$ is  an open neighbourhood of zero. By Lemma~\ref{lemma-3.1} the family $\mathscr{V}_{i_0}=\left\{\left\{V_\textsf{0}\right\},\left\{G_{i_0,j}\colon j\in \mathbb{N}_0\right\}\right\}$ is an open cover of the compactum $\mathfrak{H}_{i_0}^0(U_\textsf{0})$ and hence there exists a positive integer $j_2\geqslant j_1$ such that $\left[V_\textsf{0}\right]_{i_0,n}=\left[U_\textsf{0}\right]_{i_0,n}=\left[\operatorname{cl}_{\textbf{B}(G,\theta)^0}(U_\textsf{0})\right]_{i_0,n}$, for all integers $n\geqslant j_2$. Indeed, since $(i_0,g,j)\cdot (1,1_G,0)=(i_0,g,j-1)$ for all positive integers $j$ and any $g\in G$, for all integers $n\geqslant j_2$, we have that
\begin{equation*}
  \left[V_\textsf{0}\right]_{i_0,n}=\left[U_\textsf{0}\right]_{i_0,n}= \left[\operatorname{cl}_{\textbf{B}(G,\theta)^0}(U_\textsf{0})\right]_{i_0,n}.
\end{equation*}
We put
\begin{equation*}
  \widetilde{U}_\textsf{0}=U_\textsf{0}\setminus\left(G_{i_0,0}\cup\ldots\cup G_{i_0,j_2-1}\right).
\end{equation*}
By Lemma~\ref{lemma-3.1}, $\widetilde{U}_\textsf{0}$ is an open neighbourhood of zero in $\left(\textbf{B}(G,\theta)^0,\tau_{\textbf{B}}^0\right)$ such that $\left[\widetilde{U}_\textsf{0}\right]_{i_0,n}=\left[U_\textsf{0}\right]_{i_0,n}= \left[\operatorname{cl}_{\textbf{B}(G,\theta)^0}(U_\textsf{0})\right]_{i_0,n}$, for all integers $n\geqslant j_2$. Since $\left(\textbf{B}(G,\theta)^0,\tau_{\textbf{B}}^0\right)$  is locally compact, without loss of generality we may assume that the neighbourhood $U_\textsf{0}$ is a regular open set. This implies that $\widetilde{U}_\textsf{0}$ is a regular open set too. So there exist distinct $g,h\in G$ such that $(i_0,g,n)\notin \left[U_\textsf{0}\right]_{i_0,n}$ and $(i_0,h,n)\in \left[U_\textsf{0}\right]_{i_0,n}$, for all integers $n\geqslant j_2$. But $(i_0,gh^{-1},i_0)\cdot(i_0,h,n)=(i_0,g,n)$  for all integers $n\geqslant j_2$. Let $W_\textsf{0}=\lambda_{(i_0,gh^{-1},i_0)}^{-1}\big(\widetilde{U}_\textsf{0}\big)$. Hence we have that
\begin{equation*}
\big[\widetilde{U}_\textsf{0}\big]_{i_0,n}\setminus \left[W_\textsf{0}\right]_{i_0,n}\neq\varnothing  \qquad \hbox{and} \qquad \left[W_\textsf{0}\right]_{i_0,n}\setminus\big[\widetilde{U}_\textsf{0}\big]_{i_0,n}\neq\varnothing,
\end{equation*}
for all integers $n\geqslant j_2$. Then by Lemma~\ref{lemma-3.1} the family $\mathscr{W}_{i_0}=\left\{\left\{W_\textsf{0}\right\},\left\{G_{i_0,j}\colon j\in \mathbb{N}_0\right\}\right\}$ is an open cover of $\mathfrak{H}_{i_0}^0(U_\textsf{0})$ which has not a finite subcover. This contradicts the compactness of $\mathfrak{H}_{i_0}^0(U_\textsf{0})$, and hence the set $\left\{j\in \mathbb{N}_0\colon G_{i_0,j}\nsubseteq U_\textsf{0}\right\}$ is finite.

Suppose case $(2)$ holds. Then there are infinitely many non-negative integers $j$ such that
$\left[U_\textsf{0}\right]_{i_0,j}=G_{i_0,j}$ but  $\left[U_\textsf{0}\right]_{i_0,j-1}\neq G_{i_0,j-1}$.
Since every maximal subgroup of $\left(\textbf{B}(G,\theta)^0,\tau_{\textbf{B}}^0\right)$ is compact, Lemma~\ref{lemma-3.1} implies that every $\mathscr{H}$-class $G_{i,j}$ is an open subset of $\left(\textbf{B}(G,\theta)^0,\tau_{\textbf{B}}^0\right)$ and hence the set
\begin{equation*}
  \mathfrak{H}_{i_0}^0(U_\textsf{0})=\{\textsf{0}\}\cup \bigcup_{j\in \mathbb{N}_0} \left[\operatorname{cl}_{\textbf{B}(G,\theta)^0}(U_\textsf{0})\right]_{i_0,j}
\end{equation*}
is compact. Let $V_\textsf{0}=\rho_{(1,1_G,0)}^{-1}(U_\textsf{0})$. By Lemma~\ref{lemma-3.1} the family $\mathscr{V}_{i_0}=\left\{\left\{V_\textsf{0}\right\},\left\{G_{i_0,j}\colon j\in \mathbb{N}_0\right\}\right\}$ is an open cover of the compactum $\mathfrak{H}_{i_0}^0(U_\textsf{0})$. Then the continuity of the right shift $\rho_{(1,1_G,0)}$ and the equality $(i_0,g,j)\cdot (1,1_G,0)=(i_0,g,j-1)$ imply that $\left[V_\textsf{0}\right]_{i_0,j}\neq G_{i_0,j}$ for infinitely many non-negative integers $j$. Also, the equality $(i_0,g,j)\cdot (1,1_G,0)=(i_0,g,j-1)$ and assumption of case $(2)$ imply that $\left[U_\textsf{0}\right]_{i_0,j}\setminus \left[V_\textsf{0}\right]_{i_0,j}\neq\varnothing$ for infinitely many non-negative integers $j$. Hence, the open cover $\mathscr{V}_{i_0}$ of $\mathfrak{H}_{i_0}^0(U_\textsf{0})$ does not contain a finite subcover, which contradicts the compactness of $\mathfrak{H}_{i_0}^0(U_\textsf{0})$, and hence the set $\left\{j\in \mathbb{N}_0\colon G_{i_0,j}\nsubseteq U_\textsf{0}\right\}$ is finite.

The proof of the statement that the set $\left\{j\in \mathbb{N}_0\colon G_{j,i_0}\nsubseteq U_\textsf{0}\right\}$ is finite, is similar.
\end{proof}

\begin{lemma}\label{lemma-3.10}
Let $\left(\textbf{B}(G,\theta)^0,\tau_{\textbf{B}}^0\right)$ be a locally compact semitopological semigroup with a compact group of units and non-isolated zero \emph{$\textsf{0}$}. Then for every open neighbourhood \emph{$U_\textsf{0}$} of zero in \emph{$\left(\textbf{B}(G,\theta)^0,\tau_{\textbf{B}}^0\right)$} the set
\emph{$A_{U_\textsf{0}}=\left\{(i,j)\in \mathbb{N}_0\times \mathbb{N}_0\colon G_{i,j}\nsubseteq U_\textsf{0}\right\}$} is finite.
\end{lemma}

\begin{proof}
Suppose to the contrary that there exists an open neighbourhood $U_\textsf{0}$ of zero in \emph{$\left(\textbf{B}(G,\theta)^0,\tau_{\textbf{B}}^0\right)$} such that the set $A_{U_\textsf{0}}$ is infinite. Since the topological space  $\left(\textbf{B}(G,\theta)^0,\tau_{\textbf{B}}^0\right)$ is locally compact without loss of generality we may assume that the closure $\operatorname{cl}_{\textbf{B}(G,\theta)^0}(U_\textsf{0})$ of $U_\textsf{0}$ is compact and the neighbourhood $U_\textsf{0}$ is regular open. By Lemma~\ref{lemma-3.9}, for every positive integer $k$ there exists $(i,j)\in A_{U(0)}$ such that $i>k$ and $j>k$.

By induction we define an infinite sequence $\left\{\left(i_n,j_n\right)\right\}_{n\in\mathbb{N}}$ of elements of the set $A_{U_\textsf{0}}$ in the following way. By the assumption, there exists the smallest non-negative integer  $i_1$ such that $G_{i_1,j}\nsubseteq U_\textsf{0}$, $j\in \mathbb{N}_0$. By Lemma~\ref{lemma-3.9} there exits $j_1=\max\left\{j\in \mathbb{N}_0\colon G_{i_1,j}\nsubseteq U_\textsf{0}\right\}$.

At $k+1$-th step of induction we define pair $\left(i_{k+1},j_{k+1}\right)\in A_{U_\textsf{0}}$ as follows. Let $i_{k+1}$ be the smallest non-negative integer $>i_{k}$ such that $G_{i_{k+1},j}\nsubseteq U_\textsf{0}$, $j\in \mathbb{N}_0$. By Lemma~\ref{lemma-3.9} there exits $j_{k+1}=\max\left\{j\in \mathbb{N}_0\colon G_{i_{k+1},j}\nsubseteq U_\textsf{0}\right\}$.
Our assumption and Lemma~\ref{lemma-3.9} imply that so ordered pair $(i_{k+1},j_{k+1})$ there exists in $A_{U_\textsf{0}}$.

Now, by the separate continuity of the semigroup operation in $\left(\textbf{B}(G,\theta)^0,\tau_{\textbf{B}}^0\right)$ there exists an open neighbourhood $V_\textsf{0}\subseteq U_\textsf{0}$ of zero in $\left(\textbf{B}(G,\theta)^0,\tau_{\textbf{B}}^0\right)$ such that $V_\textsf{0}\cdot (1,1_G,0)\subseteq U_\textsf{0}$. Then the construction of the sequence $\left\{\left(i_n,j_n\right)\right\}_{n\in\mathbb{N}}$ implies that
\begin{equation*}
  \left[V_\textsf{0}\right]_{i_n,j_n}\subseteq \left[U_\textsf{0}\right]_{i_n,j_n}\neq G_{i_n,j_n} \qquad \hbox{and} \qquad \left[U_\textsf{0}\right]_{i_n,j_n+1}= G_{i_n,j_n+1},
\end{equation*}
for each $\left(i_n,j_n\right)$.
By Lemma~\ref{lemma-3.1} the family $\mathscr{V}=\left\{\left\{V_\textsf{0}\right\},\left\{G_{i,j}\colon i,j\in \mathbb{N}_0\right\}\right\}$ is an open cover of the compact set  $\operatorname{cl}_{\textbf{B}(G,\theta)^0}(U_\textsf{0})$. Then the continuity of the right shift $\rho_{(1,1_G,0)}$ implies that $\left[V_\textsf{0}\right]_{i_n,j_n+1}\neq G_{i_n,j_n+1}$ for infinitely many pairs $\left(i_n,j_n+1\right)$. This implies that $\left[U_\textsf{0}\right]_{i_n,j_n+1}\setminus \left[V_\textsf{0}\right]_{i_n,j_n+1}\neq\varnothing$ for infinitely many pairs $\left(i_n,j_n+1\right)$.
Then above arguments imply that the cover $\mathscr{V}$ has not a finite subcover, which contradicts the compactness of $\operatorname{cl}_{\textbf{B}(G,\theta)^0}(U_\textsf{0})$.
\end{proof}

\begin{example}\label{example-3.11}
Let $(G,\tau_G)$ be a semitopological group, $\theta\colon G\to G$ be a continuous homomorphism and $\mathscr{B}_G(g)$ be a base of the topology $\tau_G$ at a point $g\in G$.

On the semigroup $\textbf{B}(G,\theta)^0$ we define a topology $\tau^{\oplus}_{\textbf{B}}$ in the following way:
\begin{itemize}
  \item[$(i)$] at any non-zero element $(i,g,j)\in G_{i,j}$ of the semigroup $\textbf{B}(G,\theta)^0$ the family
  \begin{equation*}
    \mathscr{B}^{\oplus}_{\textbf{B}}(i,g,j)=\left\{U_{i,j}\colon U\in \mathscr{B}_G(g)\right\}
  \end{equation*}
  is a base of the topology $\tau^{\oplus}_{\textbf{B}}$ at the point $(i,g,j)\in\textbf{B}(G,\theta)^0$;
  \item[$(ii)$] zero $\textsf{0}\in\textbf{B}(G,\theta)^0$ is an isolated point in $\left(\textbf{B}(G,\theta)^0,\tau^{\oplus}_{\textbf{B}}\right)$.
\end{itemize}
Simple verifications show that the semigroup operation in $\left(\textbf{B}(G,\theta)^0,\tau^{\oplus}_{\textbf{B}}\right)$ is separately continuous.
\end{example}

In Example~\ref{example-3.12}, we  extend the construction proposed in Example~\ref{example-3.5} onto compact semitopological $0$-bisimple inverse $\omega$-semigroup with a compact maximal subgroup.

\begin{example}\label{example-3.12}
Let $(G,\tau_G)$ be a compact Hausdorff topological group, $\theta\colon G\to G$ be a continuous homomorphism and $\mathscr{B}_G(g)$ be a base of the topology $\tau_G$ at a point $g\in G$.

On the semigroup $\textbf{B}(G,\theta)^0$ we define a topology $\tau^{\operatorname{\textsf{Ac}}}_{\textbf{B}}$ by pointing it base $\mathscr{B}^{\operatorname{\textsf{Ac}}}_{\textbf{B}}(x)$ at each point $x$ of $\textbf{B}(G,\theta)^0$, namely
\begin{itemize}
  \item[$(i)$] at any non-zero element $(i,g,j)\in G_{i,j}$ of the semigroup $\textbf{B}(G,\theta)^0$
  \begin{equation*}
    \mathscr{B}^{\operatorname{\textsf{Ac}}}_{\textbf{B}}(i,g,j)=\left\{U_{i,j}\colon U\in \mathscr{B}_G(g)\right\};
  \end{equation*}

  \item[$(ii)$] $\mathscr{B}(\textsf{0})=\left\{U_{(i_1,j_1),\ldots,(i_k,j_k)}\colon (i_1,j_1),\ldots,(i_k,j_k)\in \mathbb{N}_0\times \mathbb{N}_0\right\}$, where
   \begin{equation*}
     U_{(i_1,j_1),\ldots,(i_k,j_k)}=\textbf{B}(G,\theta)^0\setminus\left\{G_{i_1,j_1}\cup\cdots\cup G_{i_k,j_k}\right\}.
   \end{equation*}
\end{itemize}
I.e., $\tau^{\operatorname{\textsf{Ac}}}_{\textbf{B}}$ is the topology of the Alexandroff one-point compactification of the locally compact space $\bigoplus\left\{G_{i,j}\colon i,j\in \mathbb{N}_0\right\}$ (where for any $i,j\in \mathbb{N}_0$ the space $G_{i,j}$ is homeomorphic to the compact group $(G,\tau_G)$ by the map $(i,g,j)\mapsto g$),  with the remainder $\{\textsf{0}\}$. Simple routine verifications show that the semigroup operation in $\left(\textbf{B}(G,\theta)^0,\tau^{\operatorname{\textsf{Ac}}}_{\textbf{B}}\right)$ is separately continuous.
\end{example}

Proposition~\ref{proposition-3.7} and Lemma~\ref{lemma-3.9} imply the following dichotomy for locally compact semitopological $0$-bisimple inverse $\omega$-semigroups with a compact maximal subgroup:

\begin{theorem}\label{theorem-3.13}
Let $S$ be a locally compact semitopological $0$-bisimple inverse $\omega$-semigroup with a compact maximal subgroups $G$ distinct from zero. Then $S$ is topologically isomorphic either to $\left(\textbf{B}(G,\theta)^0,\tau^{\oplus}_{\textbf{B}}\right)$ or to $\left(\textbf{B}(G,\theta)^0,\tau^{\operatorname{\textsf{Ac}}}_{\textbf{B}}\right)$, for some continuous homomorphism $\theta\colon G\to G$.
\end{theorem}

Since the bicyclic monoid $\mathscr{C}(p,q)$ does not embed into any compact topological semigroup \cite{Anderson-Hunter-Koch-1965}, Theorem~\ref{theorem-3.13} implies the following corollary.

\begin{corollary}\label{corollary-3.14}
If $S$ be a locally compact topological $0$-bisimple inverse $\omega$-semigroup with a compact maximal subgroup $G$ distinct from zero, then $S$ is topologically isomorphic to $\left(\textbf{B}(G,\theta)^0,\tau^{\oplus}_{\textbf{B}}\right)$, for some continuous homomorphism $\theta\colon G\to G$.
\end{corollary}

Later we shall need the following trivial lemma, which follows from the separate continuity of the semigroup operation in semitopological semigroups.

\begin{lemma}\label{lemma-3.15}
Let $S$ be a semitopological semigroup and $I$ be a compact ideal in $S$. Then the Rees-quotient semigroup $S/I$ with the quotient topology is a semitopological semigroup.
\end{lemma}

\begin{theorem}\label{theorem-3.16}
Let \emph{$(\textbf{B}(G,\theta)^I,\tau)$} be a locally compact semitopological bisimple inverse $\omega$-semigroup with a compact group of units $G$ and an adjoined compact ideal $I$, i.e., \emph{$\textbf{B}(G,\theta)^I=\textbf{B}(G,\theta)\sqcup I$}. Then either \emph{$(\textbf{B}(G,\theta)^I,\tau)$} is a compact semitopological semigroup or the ideal $I$ is open.
\end{theorem}

\begin{proof}
Suppose that the ideal $I$ is not open. By Lemma~\ref{lemma-3.1} the Rees-quotient semigroup $\textbf{B}(G,\theta)^I/I$ with the quotient topology $\tau_{\operatorname{\textsf{q}}}$ is a semitopological semigroup. Let $\pi\colon \textbf{B}(G,\theta)^I\to \textbf{B}(G,\theta)^I/I$ be the natural homomorphism which is a quotient map. It is obvious that the Rees-quotient semigroup $\textbf{B}(G,\theta)^I/I$ is isomorphic to the Reilly semigroup $\textbf{B}(G,\theta)^0$ and the image $\pi(I)$ is zero $\textsf{0}$ of the semigroup $\textbf{B}(G,\theta)^0$.

Now we shall show that the natural homomorphism $\pi\colon \textbf{B}(G,\theta)^I\to \textbf{B}(G,\theta)^I/I$ is a hereditarily quotient map.

We shall show that for every open neighbourhood $U(I)$ of the ideal $I$ in the space $(\textbf{B}(G,\theta)^I,\tau)$ the image  $\pi(U(I))$ is an open neighbourhood of the zero $\textsf{0}$ in the space $\left(\textbf{B}(G,\theta)^I/I,\tau_{\operatorname{\textsf{q}}}\right)$. Indeed, $\textbf{B}(G,\theta)^I\setminus U(I)$ is a closed subset of $(\textbf{B}(G,\theta)^I,\tau)$. Also, since the restriction $\pi|_{\textbf{B}(G,\theta)}\colon \textbf{B}(G,\theta)\to \pi(\textbf{B}(G,\theta))$ of the natural homomorphism $\pi\colon \textbf{B}(G,\theta)^I\to \textbf{B}(G,\theta)^I/I$ is one-to-one,  $\pi(\textbf{B}(G,\theta)^I\setminus U(I))$ is a closed subset of $\left(\textbf{B}(G,\theta)^I/I,\tau_{\operatorname{\textsf{q}}}\right)$. So $\pi(U(I))$ is an open neighbourhood of the zero $\textsf{0}$ of the semigroup $\left(\textbf{B}(G,\theta)^I/I,\tau_{\operatorname{\textsf{q}}}\right)$, and hence the natural homomorphism $\pi\colon \textbf{B}(G,\theta)^I\to \textbf{B}(G,\theta)^I/I$ is a hereditarily quotient map.

Since $I$ is a compact ideal of the semitopological semigroup $\left(\textbf{B}(G,\theta)^I,\tau\right)$, $\pi^{-1}(y)$ is a compact subset of $\left(\textbf{B}(G,\theta)^I,\tau\right)$ for every $y\in \textbf{B}(G,\theta)^I/I$. By Din' N'e T'ong's Theorem (see \cite{Din'-N'e-T'ong-1963} or \cite[3.7.E]{Engelking-1989}),
$\left(\textbf{B}(G,\theta)^I/I,\tau_{\operatorname{\textsf{q}}}\right)$ is a Hausdorff locally compact space. Since $I$ is not open by Theorem~\ref{theorem-3.13} the semitopological semigroup $\left(\textbf{B}(G,\theta)^I/I,\tau_{\operatorname{\textsf{q}}}\right)$ is topologically isomorphic to $\left(\textbf{B}(G,\theta)^0,\tau^{\operatorname{\textsf{Ac}}}_{\textbf{B}}\right)$ and hence it is compact.

At last we shall prove that the space $\left(\textbf{B}(G,\theta)^I,\tau\right)$ is compact. Let $\mathscr{U}=\left\{U_\alpha\colon\alpha\in\mathscr{I}\right\}$ be an arbitrary open cover of $\left(\textbf{B}(G,\theta)^I,\tau\right)$. Since $I$ is compact, it can be covered by a some finite subfamily $\mathscr{U}'=\left\{U_{\alpha_1},\ldots, U_{\alpha_n}\right\}$ of $\mathscr{U}$. Put $U=U_{\alpha_1}\cup\cdots\cup U_{\alpha_n}$. Then $\textbf{B}(G,\theta)^I\setminus U$ is a closed subset of $\left(\textbf{B}(G,\theta)^I,\tau\right)$. Also, since the restriction $\pi|_{\textbf{B}(G,\theta)}\colon \textbf{B}(G,\theta)\to \pi(\textbf{B}(G,\theta))$ of the natural homomorphism $\pi\colon \textbf{B}(G,\theta)^I\to \textbf{B}(G,\theta)^I/I$ is one-to-one, $\pi(\textbf{B}(G,\theta)^I\setminus U)$ is a closed subset of $(\textbf{B}(G,\theta)^I/I,\tau_{\operatorname{\textsf{q}}})$, and hence the image $\pi(\textbf{B}(G,\theta)^I\setminus U)$ is compact, because the semigroup $\left(\textbf{B}(G,\theta)^I/I,\tau_{\operatorname{\textsf{q}}}\right)$ is compact. Thus, the set $\textbf{B}(G,\theta)^I\setminus U$ is compact, so there exists a finite subfamily $\mathscr{U}''$ of $\mathscr{U}$, covering $\textbf{B}(G,\theta)^I\setminus U$. Then $\mathscr{U}'\cup \mathscr{U}''$ is a finite cover of $\left(\textbf{B}(G,\theta)^I,\tau\right)$. Hence the space $\left(\textbf{B}(G,\theta)^I,\tau\right)$ is compact as well.
\end{proof}

Since every Bruck--Reilly extension contains the bicyclic monoid ${\mathscr{C}}(p,q)$ and compact topological semigroup does not contain the semigroup ${\mathscr{C}}(p,q)$, Theorem~\ref{theorem-3.16} implies the following corollary.

\begin{corollary}\label{corollary-3.17}
Let \emph{$(\textbf{B}(G,\theta)^I,\tau)$} be a locally compact topological bisimple inverse $\omega$-semi\-group with a compact group of units $G$ and an adjoined compact ideal $I$, i.e., \emph{$\textbf{B}(G,\theta)^I=\textbf{B}(G,\theta)\sqcup I$}. Then the ideal $I$ is open in \emph{$(\textbf{B}(G,\theta)^I,\tau)$}.
\end{corollary}


\section{On locally compact semitopological $0$-bisimple inverse $\omega$-semigroups which contain the additive group of integers as a maximal subgroup}\label{section-4}

The structure of a locally compact topological group with adjoined non-isolated zero was described by Hofmann in~\cite{Hofmann-1988}:

\begin{theorem}[{\cite[Theorem~I]{Hofmann-1988}}]\label{Hofmann-Theorem}
Let $S$ be a locally compact group with non-isolated zero and $G$ its maximal subgroup $S\setminus\{0\}$. Then
\begin{itemize}
  \item[$(i)$] $G$ contains a unique characteristic maximal compact subgroup $C$.
  \item[$(ii)$] If $0$ and $1$ lie in some connected subspace, then $S$ contains in its center a locally compact group with zero $M_0=M\cup\{0\}$ which is isomorphic to the multiplicative semigroup of all non-negative real numbers and $S$ is isomorphic to the quotient semigroup of the direct product $M_0\times C$ modulo the congruence relation identifying all points of $\{0\}\times C$.

      If $0$ and $1$ do not lie in any connected subspace of $S$, then $S$ contains a locally compact group with zero $M_0=M\setminus\{0\}$ which is isomorphic  to the set of real numbers $\{0\}\cup\left\{2^n\colon n{=}0, \pm1, \pm2,\ldots\right\}$ under multiplication and S is isomorphic to the union of $0$ and the semidirect product $M\rtimes_\alpha C$, where $\alpha$ is the inner automorphism $c\longmapsto g^{-1}cg$ with the generator $g$ of $M$ whose powers converge to $0$.

  \item[$(iii)$] If $C$ is any compact group and a any automorphism of it, then there exists a locally compact group with zero whose maximal compact group is isomorphic to $C$ and whose maximal group can be made isomorphic to $M\times C$ where $M$ are positive reals under multiplication or to $M\rtimes_\alpha C$ where in this case $M$ is infinite cyclic.
\end{itemize}
\end{theorem}

When $G^0$ is a semitopological semigroup which is a locally compact group with adjoined non-isolated zero then its structure is more complicated than Hofmann's results. Thus in this section we consider a simple case  of the additive group of integers $\mathbb{Z}_{+}$.

We shall denote by  $\mathbb{Z}_{+}^{0}$ the additive group of integers $\mathbb{Z}_{+}$ with adjoined zero $0$ which is a semitopological semigroup and by $0_{\mathbb{Z}_{+}}$ the identity of $\mathbb{Z}_{+}$.

\begin{remark}\label{remark-4.1}
If the space $\mathbb{Z}_{+}^{0}$ is locally compact then by Theorem 3.3.8 from \cite{Engelking-1989} the subspace $\mathbb{Z}_{+}$ is locally compact. Then by the Baire Category Theorem (see \cite[Theorem~3.9.3]{Engelking-1989}) the space $\mathbb{Z}_{+}$ has an isolated point, and since translations in $\mathbb{Z}_{+}$ are homeomorphisms, the subspace  $\mathbb{Z}_{+}$ is discrete.
\end{remark}

We need two technical lemmas.

\begin{lemma}\label{lema-4.2}
Let $\mathbb{N}_{+}^{0}$ be a  locally compact semitopological additive semigroup of positive integers with adjoined zero $0$. Then $\mathbb{N}$ is a discrete subspace of $\mathbb{N}_{+}^{0}$ and the space $\mathbb{N}_{+}^{0}$ is either compact or discrete.
\end{lemma}

\begin{proof}
Fix an arbitrary $n\in\mathbb{N}$. Since $\mathbb{N}_{+}^{0}$ is locally compact Theorem 3.3.8 of \cite{Engelking-1989} implies $\mathbb{N}$ is locally compact too. Now Hausdorffness of the space $\mathbb{N}$ implies that the set $[n+1)=\mathbb{N}\setminus\{1,\ldots,n\}$ is open in $\mathbb{N}_{+}^{0}$, and hence by Theorem 3.3.8 from \cite{Engelking-1989} it is locally compact. By the Baire Category Theorem the space $[n+1)$ contains an isolated point $m$. It is obvious that $m$ is isolated in $\mathbb{N}$. Since $n<m$ the separate continuity of the semigroup operation in $\mathbb{N}_{+}^{0}$ implies that $n$ is an isolated point of $\mathbb{N}$, and hence by Hausdorffness of $\mathbb{N}_{+}^{0}$ the point $n$ is isolated in $\mathbb{N}_{+}^{0}$.

Suppose that the space $\mathbb{N}_{+}^{0}$ is not compact and its point $0$ is not isolated. Fix an arbitrary open neighbourhood $U(0)$ of $0$ in $\mathbb{N}_{+}^{0}$ with the compact closure $\operatorname{cl}_{\mathbb{N}_{+}^{0}}(U(0))$. Since $\mathbb{N}$ is a discrete subspace of $\mathbb{N}_{+}^{0}$,
$\operatorname{cl}_{\mathbb{N}_{+}^{0}}(U(0))=U(0)$. Since the space $\mathbb{N}_{+}^{0}$ is not compact and its subspace $\mathbb{N}$ is discrete,  $U(0)=\mathbb{N}_{+}^{0}\setminus\{n_i\colon i\in\mathbb{N}\}$, where $\left\{n_i\colon i\in\mathbb{N}\right\}$ is an infinite increasing sequence of positive integers. The separate continuity of the semigroup operation of $\mathbb{N}_{+}^{0}$ implies that there exists an open neighbourhood $V(0)\subseteq U(0)$ of the zero $0$ in $\mathbb{N}_{+}^{0}$ such that $1+V(0)\subseteq U(0)$.

Now, by the existence of the infinite sequence $\left\{n_i\colon i\in\mathbb{N}\right\}$ and since the set $U(0)\setminus V(0)$ is infinite, we have that the family $\mathscr{V}=\left\{V(0),\left\{\{i\}\colon i\in\mathbb{N}\right\}\right\}$ is an open cover of $U(0)$ which does not contain a finite subcover. This contradicts the compactness of $U(0)$. The obtained contradiction implies the second statement of the lemma.
\end{proof}

\begin{lemma}\label{lema-4.3}
Let $S$ be a semitopological semigroup and $M$ be a dense subsemigroup of $S$. If $1_M$ and $0_M$ are identity and zero in $M$, respectively, then $1_M$ and $0_M$ so are in $S$.
\end{lemma}

\begin{proof}
Let $\operatorname{id}$ be the identity map on $S$ and $\bar{O}$ be a constant map $S\to \{0_M\}$.
Clearly, these maps are continuous. Since the shifts (on $S$) by elements $1_M$ and $0_M$ are continuous too, and $1_Mx=x1_M=x=\operatorname{id}(x)$ and $0_Mx=x0_M=0_M=\bar{0}(x)$ for each element $x$ of a dense subset $M$ of a Hausdorff space $S$, by \cite[Theorem~1.5.4]{Engelking-1989} these equalities hold for each $x\in S$.
\end{proof}

Since the subsemigroup $\mathbb{N}_{-}^{0}=\left\{0\right\}\cup\left\{-n \colon n\in\mathbb{N}\right\}$ of $\mathbb{Z}_{+}^{0}$ is  isomorphic to $\mathbb{N}_{+}^{0}$, Lemmas~\ref{lema-4.2} and \ref{lema-4.3} and Remark~\ref{remark-4.1} imply the following proposition.

\begin{proposition}\label{proposition-4.4}
Let $\mathbb{Z}_{+}^{0}$ be a locally compact semitopological semigroup. Then all non-zero points of $\mathbb{Z}_{+}^{0}$ are isolated and exactly one of the following conditions holds:
\begin{itemize}
  \item[$(i)$] zero $0$ is isolated in $\mathbb{Z}_{+}^{0}$;
  \item[$(ii)$] the family $\mathscr{B}_{\mathbf{cf}}=\left\{U_F=\mathbb{Z}_{+}^{0}\setminus F\colon F \hbox{~is a finite subset of~} \mathbb{Z}\right\}$ is a base at zero;
  \item[$(iii)$] the family $\mathscr{B}^+=\left\{U_F^+=\mathbb{N}_{+}^{0}\setminus F\colon F \hbox{~is a finite subset of~} \mathbb{N}\right\}$ is a base at zero;
  \item[$(iv)$] the family $\mathscr{B}^-=\left\{U_F^-=\mathbb{Z}_{+}^{0}\setminus \left(\mathbb{N}\cup F\right)\colon F \hbox{~is a finite subset of~} \mathbb{Z}\right\}$ is a base at zero.
\end{itemize}
\end{proposition}

Next we define four topologies $\tau_{\mathfrak{d}}$,  $\tau_{\mathbf{cf}}$,  $\tau_{+}$ and $\tau_{-}$ on the Reilly semigroup $\textbf{B}(\mathbb{Z}_{+},\theta)$, where $\theta\colon \mathbb{Z}_{+}\to \mathbb{Z}_{+}\colon z\mapsto 0_{\mathbb{Z}_{+}}$ is the annihilating homomorphism.

\begin{example}\label{example-4.5}
Let $\tau_{\mathfrak{d}}$ be the discrete topology on $\textbf{B}(\mathbb{Z}_{+},\theta)$. It is obvious that $\left(\textbf{B}(\mathbb{Z}_{+},\theta),\tau_{\mathfrak{d}}\right)$ is a topological inverse semigroup.
\end{example}

\begin{example}\label{example-4.6}
We define a topology $\tau_{\mathbf{cf}}$ on $\textbf{B}(\mathbb{Z}_{+},\theta)$ in the following way. Let $(i,g,j)$ be an isolated point of $\left(\textbf{B}(\mathbb{Z}_{+},\theta),\tau_{\mathbf{cf}}\right)$ in the following cases:
\begin{itemize}
  \item[$(1)$] $g\neq 0_{\mathbb{Z}_{+}}$ and $i,j\in \mathbb{N}_0$;
  \item[$(2)$] $i=0$ or $j=0$.
\end{itemize}
The family
\begin{equation*}
  \mathscr{B}_{\mathbf{cf}}(i,0_{\mathbb{Z}_{+}},j)=\left\{\left(U_F\right)_{i-1,j-1}^0=\left(\mathbb{Z}_{+}\setminus F\right)_{i-1,j-1}^0\colon F \hbox{~is a finite subset of~} \mathbb{Z}\right\}
\end{equation*}
defines the base of the topology $\tau_{\mathbf{cf}}$ on $\textbf{B}(\mathbb{Z}_{+},\theta)$ at the point $(i,0_{\mathbb{Z}_{+}},j)$, for all $i,j\in\mathbb{N}$. Simple verifications show that $\tau_{\mathbf{cf}}$ is a Hausdorff locally compact topology on $\textbf{B}(\mathbb{Z}_{+},\theta)$.

Next we shall prove that $\left(\textbf{B}(\mathbb{Z}_{+},\theta),\tau_{\mathbf{cf}}\right)$ is a semitopological inverse semigroup with continuous inversion. The definition of the topology $\tau_{\mathbf{cf}}$ on $\textbf{B}(\mathbb{Z}_{+},\theta)$ implies that it is sufficient to show that the semigroup operation on $\left(\textbf{B}(\mathbb{Z}_{+},\theta),\tau_{\mathbf{cf}}\right)$ is separately continuous in the following two cases:
\begin{equation*}
  \left(U_F\right)_{i,j}^0\cdot (m,g,n) \qquad \hbox{and} \qquad (m,g,n)\cdot\left(U_F\right)_{i,j}^0,
\end{equation*}
where $i,j\in \mathbb{N}_0$, $m,n\in \mathbb{N}$, $F$ is an arbitrary finite subset of $\mathbb{Z}$ and $g$ is an arbitrary element of the group~$\mathbb{Z}_{+}$, because in cases $(1)$ and $(2)$ the point $(i,g,j)$ is isolated in $\left(\textbf{B}(\mathbb{Z}_{+},\theta),\tau_{\mathbf{cf}}\right)$.

Simple calculations show that
\begin{equation*}
  \left(U_F\right)_{i,j}^0\cdot (m,g,n)=
  \left\{
    \begin{array}{cl}
      \left\{(m+i-j,g,n)\right\},   & \hbox{if~} j<m;\\
      \left(U_{F+g}\right)_{i,n}^0, & \hbox{if~} j=m;\\
      \left(U_F\right)_{i,j+n-m}^0, & \hbox{if~} j> m,
    \end{array}
  \right.
\end{equation*}

\begin{equation*}
  (i,g,j)\cdot\left(U_F\right)_{m,n}^0=
  \left\{
    \begin{array}{cl}
      \left(U_F\right)_{i+m-j,n}^0, & \hbox{if~} j<m;\\
      \left(U_{g+F}\right)_{i,n}^0, & \hbox{if~} j=m;\\
      \left\{(i,g,j+n-m)\right\},   & \hbox{if~} j>m,
    \end{array}
  \right.
\end{equation*}
and
\begin{equation*}
  \operatorname{\textsf{inv}}\left(\left(U_F\right)_{i,j}^0\right)=\left(U_{-F}\right)_{j,i}^0,
\end{equation*}
where $g+F=F+g=\left\{f+g\colon f\in F\right\}$ and $-F=\left\{-f\colon f\in F\right\}$.
\end{example}

\begin{example}\label{example-4.7}
We define a topology $\tau_{+}$ on $\textbf{B}(\mathbb{Z}_{+},\theta)$ in the following way. Let $(i,g,j)$ be an isolated point of $\left(\textbf{B}(\mathbb{Z}_{+},\theta),\tau_{+}\right)$ in the following cases:
\begin{itemize}
  \item[$(1)$] $g\neq 0_{\mathbb{Z}_{+}}$ and $i,j\in \mathbb{N}_0$;
  \item[$(2)$] $i=0$ or $j=0$.
\end{itemize}
The family
\begin{equation}\label{eq-4.1}
  \mathscr{B}_{+}(i,0_{\mathbb{Z}_{+}},j)=\left\{\left(U_F^+\right)_{i-1,j-1}^0=\left(\mathbb{N}\setminus F\right)_{i-1,j-1}^0\colon F \hbox{~~is a finite subset of~~} \mathbb{N}\right\}
\end{equation}
defines the base of the topology $\tau_{+}$ on $\textbf{B}(\mathbb{Z}_{+},\theta)$ at the point $(i,0_{\mathbb{Z}_{+}},j)$, for all $i,j\in\mathbb{N}$. Simple verifications show that $\tau_{+}$ is a Hausdorff locally compact topology on $\textbf{B}(\mathbb{Z}_{+},\theta)$.

Next we shall prove that $\left(\textbf{B}(\mathbb{Z}_{+},\theta),\tau_{+}\right)$ is a topological semigroup. Now, the definition of the topology $\tau_{+}$ on $\textbf{B}(\mathbb{Z}_{+},\theta)$ implies that it is sufficient to show that the semigroup operation on $\left(\textbf{B}(\mathbb{Z}_{+},\theta),\tau_{+}\right)$ is continuous in the following three cases:
\begin{equation*}
  \left(U_F^+\right)_{i,j}^0\cdot (m,g,n), \qquad (m,g,n)\cdot\left(U_F^+\right)_{i,j}^0, \qquad \hbox{and} \qquad \left(U_F^+\right)_{k,l}^0\cdot\left(U_{F^{\prime}}^+\right)_{i,j}^0,
\end{equation*}
where $i,j,k.l\in \mathbb{N}_0$, $m,n\in \mathbb{N}$, $F$ and $F^{\prime}$ are arbitrary finite subsets of $\mathbb{Z}$ and $g$ is an arbitrary element of the group $\mathbb{Z}_{+}$, because in cases $(1)$ and $(2)$ the point $(i,g,j)$ is isolated in $\left(\textbf{B}(\mathbb{Z}_{+},\theta),\tau_{+}\right)$.

Without loss of generality we may assume that in formula (\ref{eq-4.1}) the set $F$ is an initial interval of the set of positive integers, because for every finite subset $F$ of $\mathbb{N}$ there exists an initial interval $\{1,\ldots,m\}$ of $\mathbb{N}$ such that $F\subseteq\{1,\ldots,m\}$.

Simple calculations show that
\begin{equation*}
  \left(U_{{F-g}}^+\right)_{i,j}^0\cdot (m,g,n)=
  \left\{
    \begin{array}{cl}
      \left\{(m+i-j,g,n)\right\},                    & \hbox{if~} j<m;\\
      \left(U_{F}^+\right)_{i,n}^0,                  & \hbox{if~} j=m;\\
      \left(U_{{F-g}}^+\right)_{i,j+n-m}^0, & \hbox{if~} j>m,
    \end{array}
  \right.
\end{equation*}

\begin{equation*}
  (i,g,j)\cdot\left(U_{{F-g}}^+\right)_{m,n}^0=
  \left\{
    \begin{array}{cl}
      \left(U_{{F-g}}^+\right)_{i+m-j,n}^0, & \hbox{if~} j<m;\\
      \left(U_{F}^+\right)_{i,n}^0,                  & \hbox{if~} j=m;\\
      \left\{(i,g,j+n-m)\right\},                    & \hbox{if~} j>m,
    \end{array}
  \right.
\end{equation*}
and
\begin{equation*}
  \left(U_F^+\right)_{i,j}^0\cdot\left(U_{F}^+\right)_{m,n}^0=
    \left\{
    \begin{array}{cl}
      \left(U_{F}^+\right)_{m,n}^0,  & \hbox{if~} j<m;\\
      \left(U_{F+F}^+\right)_{i,n}^0, & \hbox{if~} j=m;\\
      \left(U_F^+\right)_{i,j}^0,    & \hbox{if~} j>m,
    \end{array}
  \right.
\end{equation*}
where ${F-g}=\left\{f-g\colon f\in F\right\}\cap\mathbb{N}$  and ${F+F}=\left\{f+g\colon f,g\in F\right\}$.
\end{example}

\begin{example}\label{example-4.8}
We define a topology $\tau_{-}$ on $\textbf{B}(\mathbb{Z}_{+},\theta)$ in the following way. Let $(i,g,j)$ be an isolated point of $\left(\textbf{B}(\mathbb{Z}_{+},\theta),\tau_{-}\right)$ in the following cases:
\begin{itemize}
  \item[$(1)$] $g\neq 0_{\mathbb{Z}_{+}}$ and $i,j\in \mathbb{N}_0$;
  \item[$(2)$] $i=0$ or $j=0$.
\end{itemize}
The family
\begin{equation*}
  \mathscr{B}_{-}(i,0_{\mathbb{Z}_{+}},j)=\left\{\left(U_F^-\right)_{i,j}^0=\left(\mathbb{Z}\setminus \left(\mathbb{N}_0\cup -F\right)\right)_{i,j}^0\colon F \hbox{~~is a finite subset of~~} \mathbb{N}\right\},
\end{equation*}
where $-F=\{-m\colon m\in F\}$, defines the base of the topology $\tau_{-}$ on $\textbf{B}(\mathbb{Z}_{+},\theta)$ at the point $(i,0_{\mathbb{Z}_{+}},j)$, for all $i,j\in\mathbb{N}$. Simple verifications show that $\tau_{-}$ is a Hausdorff locally compact topology on $\textbf{B}(\mathbb{Z}_{+},\theta)$.

We observe that the map $\Phi\colon\left(\textbf{B}(\mathbb{Z}_{+},\theta),\tau_{+}\right)\to \left(\textbf{B}(\mathbb{Z}_{+},\theta),\tau_{-}\right)$ defined by the formula $\Phi(i,g,j)=(i,-g,j)$ is an isomorphism of the semigroup $\textbf{B}(\mathbb{Z}_{+},\theta)$ which is a homeomorphism of topological spaces $\left(\textbf{B}(\mathbb{Z}_{+},\theta),\tau_{+}\right)$ and $\left(\textbf{B}(\mathbb{Z}_{+},\theta),\tau_{-}\right)$, and hence $\left(\textbf{B}(\mathbb{Z}_{+},\theta),\tau_{-}\right)$ is a topological semigroup which is topologically isomorphic to the topological semigroup $\left(\textbf{B}(\mathbb{Z}_{+},\theta),\tau_{+}\right)$.
\end{example}

The following theorem describes all shift-continuous locally compact topologies on the Reilly semigroup $\textbf{B}(\mathbb{Z}_{+},\theta)$ in the case when $\theta\colon \mathbb{Z}_{+}\to \mathbb{Z}_{+}$, $z\mapsto 0_{\mathbb{Z}_{+}}$ is the annihilating homomorphism.

\begin{theorem}\label{theorem-4.9}
Let $\theta\colon \mathbb{Z}_{+}\to \mathbb{Z}_{+}$, $z\mapsto 0_{\mathbb{Z}_{+}}$ be the annihilating homomorphism and $\tau$ be a shift-continuous  locally compact topology on \emph{$\textbf{B}(\mathbb{Z}_{+},\theta)$}. Then only one of the following conditions holds:
\begin{itemize}
  \item[$(i)$] $\tau$ is discrete;
  \item[$(ii)$] $\tau=\tau_{\mathbf{cf}}$;
  \item[$(iii)$] $\tau=\tau_{+}$;
  \item[$(iv)$] $\tau=\tau_{-}$.
\end{itemize}
\end{theorem}

\begin{proof}
Since
\begin{equation*}
  (i,g,j)(1,0_{\mathbb{Z}_{+}},1)=
  \left\{
    \begin{array}{cl}
      (i+1,0_{\mathbb{Z}_{+}},1), & \hbox{if~~} j=0;\\
      (i,g,1),                    & \hbox{if~~} j=1;\\
      (i,g,j),                    & \hbox{if~~} j>1,
    \end{array}
  \right.
\end{equation*}
for arbitrary $i,j\in \mathbb{N}_0$ and the right shift $\rho_{(1,0_{\mathbb{Z}_{+}},1)}\colon \left(\textbf{B}(\mathbb{Z}_{+},\theta),\tau\right)\to \left(\textbf{B}(\mathbb{Z}_{+},\theta),\tau\right)$, $x\mapsto x\cdot(1,0_{\mathbb{Z}_{+}},1)$ is continuous, the subspace  $\left(\mathbb{Z}_{+}\right)_{0,0}^0=\rho_{(1,0_{\mathbb{Z}_{+}},1)}^{-1}\left((1,0_{\mathbb{Z}_{+}},1)\right)$ of $\left(\textbf{B}(\mathbb{Z}_{+},\theta),\tau\right)$ is closed, and hence the local compactness of $\left(\textbf{B}(\mathbb{Z}_{+},\theta),\tau\right)$ and Theorem~3.3.8 from \cite{Engelking-1989} imply that $\left(\mathbb{Z}_{+}\right)_{0,0}^0$ is locally compact. It is obvious that $\left(\mathbb{Z}_{+}\right)_{0,0}^0$ with the induced semigroup operation from $\textbf{B}(\mathbb{Z}_{+},\theta)$ is isomorphic to $\mathbb{Z}_{+}^{0}$, and hence the statement of Proposition~\ref{proposition-4.4} for the locally compact subsemigroup $\left(\mathbb{Z}_{+}\right)_{0,0}^0$ holds. Proposition~\ref{proposition-2.4}$(iv)$  completes the proof of our theorem.
\end{proof}

The following corollary follows from Theorem~\ref{theorem-4.9} and it describes all semigroup Hausdorff locally compact topologies on the Reilly semigroup $\textbf{B}(\mathbb{Z}_{+},\theta)$ in the case when $\theta$ is the annihilating homomorphism.

\begin{corollary}\label{corollary-4.10}
For a locally compact topological semigroup  \emph{$\left(\textbf{B}(\mathbb{Z}_{+},\theta),\tau\right)$} with the annihilating homomorphism $\theta$ exactly one of the following conditions holds:
\begin{itemize}
  \item[$(i)$] $\tau$ is discrete;
  \item[$(ii)$] $\tau=\tau_{+}$;
  \item[$(iii)$] $\tau=\tau_{-}$.
\end{itemize}
\end{corollary}

\begin{proposition}\label{proposition-4.11}
If $\theta$ is the annihilating homomorphism and $\tau$ is a shift-continuous locally compact topology on \emph{$\textbf{B}(\mathbb{Z}_{+},\theta)^0$} such that the induced topology of $\tau$ on \emph{$\textbf{B}(\mathbb{Z}_{+},\theta)$} is discrete, then $\tau$ is discrete.
\end{proposition}

\begin{proof}
Suppose to the contrary that the zero $\textsf{0}$ is not an isolated point of $\left(\textbf{B}(\mathbb{Z}_{+},\theta)^0,\tau\right)$.

Since all non-zero elements of the semigroup $\textbf{B}(\mathbb{Z}_{+},\theta)^0$ are isolated points in $\left(\textbf{B}(\mathbb{Z}_{+},\theta)^0,\tau\right)$, the binary relation $\eta^\natural$ on $\textbf{B}(\mathbb{Z}_{+},\theta)^0$ defined by formula \eqref{eq-3.1} is a closed congruence on $\left(\textbf{B}(\mathbb{Z}_{+},\theta)^0,\tau\right)$.   Next we shall show that if the natural homomorphism $\eta\colon \textbf{BR}(\mathbb{Z}_{+},\theta)^0\to \mathscr{C}^0$ is a quotient map then $\eta$ is open. Since all non-zero elements in $\left(\textbf{B}(\mathbb{Z}_{+},\theta)^0,\tau\right)$ are isolated points it suffices to prove that the image $\eta(U_\textbf{0})$ is open for every open neighbourhood $U_\textbf{0}$ of zero. Put $U_\textbf{0}^*=\eta^{-1}\left(\eta(U_\textbf{0})\right)$. Clearly, $U^*_\textbf{0}=\eta^{-1}\left(\eta(U^*_\textbf{0})\right)$. Since $\eta\colon \textbf{BR}(\mathbb{Z}_{+},\theta)^0\to \mathscr{C}^0$ is a natural homomorphism,
\begin{equation*}
  U^*_\textbf{0}=\bigcup\left\{(\mathbb{Z}_{+})_{i,j}\colon (\mathbb{Z}_{+})_{i,j}\cap U_\textbf{0}\neq\varnothing\right\}\cup \{\textbf{0}\}.
\end{equation*}
The last equality and Lemma~\ref{lemma-3.1} implies that $U^*_0$ is an open subset of the space $\left(\textbf{B}(\mathbb{Z}_{+},\theta)^0,\tau\right)$. Since $\eta$ is a quotient map and $U^*_0=\eta^{-1}\left(\eta(U^*_0)\right)$, $\eta(U_0)$ is an open subset of the space $\mathscr{C}^0$. This implies that the quotient semigroup $\textbf{B}(\mathbb{Z}_{+},\theta)^0/\eta^\natural$ with the quotient topology is a Hausdorff semitopological semigroup.

By Theorem~3.3.15 from \cite{Engelking-1989}, $\textbf{B}(\mathbb{Z}_{+},\theta)^0/\eta^\natural$ with the quotient topology is a locally compact space with non-isolated zero. Then by Theorem~1 from \cite{Gutik-2015} the quotient semigroup $\textbf{B}(\mathbb{Z}_{+},\theta)^0/\eta^\natural$ with the quotient topology is topologically isomorphic to the compact semitopological semigroup $(\mathscr{C}^0,\tau_{\operatorname{\textsf{Ac}}})$ (see Example~\ref{example-3.5}). Now, since the natural homomorphism $\eta\colon \left(\textbf{B}(\mathbb{Z}_{+},\theta)^0,\tau\right)\to (\mathscr{C}^0,\tau_{\operatorname{\textsf{Ac}}})$ is an open map, for every open neighbourhood $U_\textsf{0}$ of zero in $\left(\textbf{B}(\mathbb{Z}_{+},\theta)^0,\tau\right)$, there exist finitely many subsets  $(\mathbb{Z}_{+})_{i,j}$ in $\textbf{B}(\mathbb{Z}_{+},\theta)^0$ such that $(\mathbb{Z}_{+})_{i,j}\cap U_\textsf{0}=\varnothing$.

Next we shall show that for any open neighbourhood $U_\textsf{0}$ of zero in $\left(\textbf{B}(\mathbb{Z}_{+},\theta)^0,\tau\right)$ there exist finitely many subsets $(\mathbb{Z}_{+})_{0,j}$ in $\textbf{B}(\mathbb{Z}_{+},\theta)^0$, $j\in \mathbb{N}_0$, such that $(\mathbb{Z}_{+})_{0,j}\nsubseteq U_\textsf{0}$. Suppose to the contrary that there exist an open neighbourhood $U_\textsf{0}$ of zero in $\left(\textbf{B}(\mathbb{Z}_{+},\theta)^0,\tau\right)$ and infinitely many subsets $(\mathbb{Z}_{+})_{0,j}$ in $\textbf{B}(\mathbb{Z}_{+},\theta)^0$, $j\in \mathbb{N}_0$, such that $(\mathbb{Z}_{+})_{0,j}\nsubseteq U_\textsf{0}$. Since $\textbf{B}(\mathbb{Z}_{+},\theta)$ is a discrete subspace of the locally compact space $\left(\textbf{B}(\mathbb{Z}_{+},\theta)^0,\tau\right)$ without loss of generality we may assume that the neighbourhood $U_\textsf{0}$ is compact. Then the separate continuity of the semigroup operation in $\left(\textbf{B}(\mathbb{Z}_{+},\theta)^0,\tau\right)$ implies that there exists an open neighbourhood $V_\textsf{0}\subseteq U_\textsf{0}$ such that $(0,1,0)\cdot V_\textsf{0}\subseteq U_\textsf{0}$. Also, the assumption that there exist infinitely many subsets $(\mathbb{Z}_{+})_{0,j}$ in $\textbf{B}(\mathbb{Z}_{+},\theta)^0$, $j\in \mathbb{N}_0$, such that $(\mathbb{Z}_{+})_{0,j}\nsubseteq U_\textsf{0}$ implies that the following open cover $\mathscr{V}=\left\{\left\{V_\textsf{0}\right\},\left\{(\mathbb{Z}_{+})_{i,j}\colon i,j\in \mathbb{N}_0\right\}\right\}$ of the set $U_\textsf{0}$ does not contain a finite subcover, because our assumption implies that there does not exist finitely many subsets $(\mathbb{Z}_{+})_{0,j}$ in $\textbf{B}(\mathbb{Z}_{+},\theta)^0$, $j\in \mathbb{N}_0$, which cover the set $\left\{(\mathbb{Z}_{+})_{0,j}\colon j\in \mathbb{N}_0\right\}\cap \left((0,1,0)\cdot V_\textsf{0}\right)\setminus V_\textsf{0}$. This contradicts the compactness of $U_\textsf{0}$.

Fix an arbitrary compact open neighbourhood $U_\textsf{0}$ of zero in $\left(\textbf{B}(\mathbb{Z}_{+},\theta)^0,\tau\right)$ and a set $(\mathbb{Z}_{+})_{0,j_0}\subseteq U_\textsf{0}$. Then the following formula
\begin{equation*}
(1,0_{\mathbb{Z}_{+}},1)(i,g,j)=
  \left\{
    \begin{array}{cl}
      (1,0_{\mathbb{Z}_{+}},j+1), & \hbox{if~~} i=0;\\
      (1,g,j),                    & \hbox{if~~} i=1;\\
      (i,g,j),                    & \hbox{if~~} i>1,
    \end{array}
  \right.
\end{equation*}
$i,j\in \mathbb{N}_0$, implies that the set $(\mathbb{Z}_{+})_{0,j_0}^0$ is open-and-closed in $\left(\textbf{B}(\mathbb{Z}_{+},\theta)^0,\tau\right)$ because  $\textbf{B}(\mathbb{Z}_{+},\theta)$ is a discrete subspace of $\left(\textbf{B}(\mathbb{Z}_{+},\theta)^0,\tau\right)$ and $\tau$ is shift-continuous. Then the neighbourhood $U_\textsf{0}$ contains the open-and-closed discrete subspace $(\mathbb{Z}_{+})_{0,j_0}$ which contradicts the compactness of $U_\textsf{0}$.
\end{proof}

For all non-negative integers $k$ and $l$ we put  $(\widetilde{\mathbb{Z}}_{+})_{k,l}^0=(\mathbb{Z}_{+})_{k,l}^0\setminus \left\{(k,0_{\mathbb{Z}_{+}},l)\right\}$, $\widetilde{\mathfrak{Z}}=\left\{(\widetilde{\mathbb{Z}}_{+})_{i,j}^0\colon i,j\in \mathbb{N}_0\right\}$ and  $\widetilde{\mathfrak{O}}=\left\{\left\{(i,0_{\mathbb{Z}_{+}},j)\right\}\colon i,j\in \mathbb{N}_0 \hbox{~and~} ij=0\right\}$. It is obvious that $\textbf{B}(\mathbb{Z}_{+},\theta)=\bigcup\widetilde{\mathfrak{Z}}\cup \bigcup\widetilde{\mathfrak{O}}$.


\begin{proposition}\label{proposition-4.12}
Let $\theta$ be the annihilating homomorphism and $\tau$ be a locally compact topology on \emph{$\textbf{B}(\mathbb{Z}_{+},\theta)^0$} such that \emph{$\left(\textbf{B}(\mathbb{Z}_{+},\theta)^0,\tau\right)$} is a semitopological semigroup with non-isolated zero \emph{$\textsf{0}$} and the induced topology of $\tau$ on \emph{$\textbf{B}(\mathbb{Z}_{+},\theta)$} coincides with the topology $\tau_{\mathbf{cf}}$. Then the following assertions hold:
\begin{itemize}
  \item[$(i)$] for any non-negative integers $i$ and $j$ the set $(\widetilde{\mathbb{Z}}_{+})_{i,j}^0$ is open and compact in \emph{$\left(\textbf{B}(\mathbb{Z}_{+},\theta)^0,\tau\right)$} and hence it is closed;

  \item[$(ii)$] every open neighbourhood \emph{$U_\textsf{0}$} of zero in \emph{$\left(\textbf{B}(\mathbb{Z}_{+},\theta)^0,\tau\right)$} intersects infinitely many elements of the family $\widetilde{\mathfrak{Z}}$;

  \item[$(iii)$] for every open neighbourhood \emph{$U_\textsf{0}$} of zero  with the compact closure and any open neighbourhood \emph{$V_\textsf{0}$} of zero in \emph{$\left(\textbf{B}(\mathbb{Z}_{+},\theta)^0,\tau\right)$} each of  sets \emph{$\operatorname{cl}_{\left(\textbf{B}(\mathbb{Z}_{+},\theta)^0,\tau\right)}(U_\textsf{0})\setminus V_\textsf{0}$} and \emph{$U_\textsf{0}\setminus V_\textsf{0}$} intersects finitely many elements of the family $\widetilde{\mathfrak{Z}}$;

  \item[$(iv)$] for any non-negative integer $i_0$ every open neighbourhood \emph{$U_\textsf{0}$} of zero in \emph{$\left(\textbf{B}(\mathbb{Z}_{+},\theta)^0,\tau\right)$} intersects infinitely many elements of the family $\widetilde{\mathfrak{Z}}$ of the form $(\widetilde{\mathbb{Z}}_{+})_{i_0,j}^0$;

  \item[$(v)$] for any non-negative integer $i_0$ every open neighbourhood \emph{$U_\textsf{0}$} of zero in \emph{$\left(\textbf{B}(\mathbb{Z}_{+},\theta)^0,\tau\right)$} intersects infinitely many elements of the family $\widetilde{\mathfrak{Z}}$ of the form $(\widetilde{\mathbb{Z}}_{+})_{j,i_0}^0$;

  \item[$(vi)$] for any non-negative integer $i_0$ and every open neighbourhood \emph{$U_\textsf{0}$} of zero in \emph{$\left(\textbf{B}(\mathbb{Z}_{+},\theta)^0,\tau\right)$} the  family $\mathfrak{H}_{i_0}=\left\{(\widetilde{\mathbb{Z}}_{+})_{i_0,j}^0\in\widetilde{\mathfrak{Z}}\colon (\widetilde{\mathbb{Z}}_{+})_{i_0,j}^0\cap U_\emph{\textsf{0}}=\varnothing\right\}$ is finite;

  \item[$(vii)$] for any non-negative integer $i_0$ and every open neighbourhood $U_\emph{\textsf{0}}$ of zero in \emph{$\left(\textbf{B}(\mathbb{Z}_{+},\theta)^0,\tau\right)$} the  family $\mathfrak{V}_{i_0}=\left\{(\widetilde{\mathbb{Z}}_{+})_{j,i_0}^0\in\widetilde{\mathfrak{Z}}\colon (\widetilde{\mathbb{Z}}_{+})_{j,i_0}^0\cap U_\emph{\textsf{0}}=\varnothing\right\}$ is finite;

  \item[$(viii)$] for any non-negative integer $i_0$ and every open neighbourhood $U_\emph{\textsf{0}}$ of zero in \emph{$\left(\textbf{B}(\mathbb{Z}_{+},\theta)^0,\tau\right)$} the set $\left\{(i_0,0_{\mathbb{Z}_{+}},j)\colon j\in \mathbb{N}_0\right\}\setminus U_\emph{\textsf{0}}$ is finite;

  \item[$(ix)$] for any non-negative integer $j_0$ and every  open neighbourhood $U_\emph{\textsf{0}}$ of zero in \emph{$\left(\textbf{B}(\mathbb{Z}_{+},\theta)^0,\tau\right)$} the set $\left\{(i,0_{\mathbb{Z}_{+}},j_0)\colon i\in \mathbb{N}_0\right\}\setminus U_\emph{\textsf{0}}$ is finite;

  \item[$(x)$] for any non-negative integer $i_0$ and every open neighbourhood $U_\emph{\textsf{0}}$ of zero in \emph{$\left(\textbf{B}(\mathbb{Z}_{+},\theta)^0,\tau\right)$} there exist finitely many subsets of the form $(\mathbb{Z}_{+})_{i_0,j}^0$ such that $(\mathbb{Z}_{+})_{i_0,j}^0\nsubseteq U_\emph{\textsf{0}}$;

  \item[$(xi)$] for any non-negative integer $i_0$ and every open neighbourhood $U_\emph{\textsf{0}}$ of zero in \emph{$\left(\textbf{B}(\mathbb{Z}_{+},\theta)^0,\tau\right)$} there exist finitely many subsets of the form $(\mathbb{Z}_{+})_{j,i_0}^0$ such that $(\mathbb{Z}_{+})_{i_0,j}^0\nsubseteq U_\emph{\textsf{0}}$;

  \item[$(xii)$] for every open neighbourhood $U_\emph{\textsf{0}}$ of zero in \emph{$\left(\textbf{B}(\mathbb{Z}_{+},\theta)^0,\tau\right)$} there exist finitely many subsets $(\mathbb{Z}_{+})_{i_1,j_1}^0,\ldots,(\mathbb{Z}_{+})_{i_k,j_k}^0$ such that \emph{$(\mathbb{Z}_{+})_{i_1,j_1}^0\cup\ldots\cup(\mathbb{Z}_{+})_{i_k,j_k}^0\cup U_\textsf{0}=\textbf{B}(\mathbb{Z}_{+},\theta)^0$}.
\end{itemize}
\end{proposition}

\begin{proof}
$(i)$ Since $\textbf{B}(\mathbb{Z}_{+},\theta)$ is an open subset of $\left(\textbf{B}(\mathbb{Z}_{+},\theta)^0,\tau\right)$ the set $(\widetilde{\mathbb{Z}}_{+})_{i,j}^0$ is open for any non-negative integers $i$ and $j$. Moreover, the definition of the topology $\tau_{\mathbf{cf}}$ on  $\textbf{B}(\mathbb{Z}_{+},\theta)$ implies that $(\widetilde{\mathbb{Z}}_{+})_{i,j}^0$ is compact, and since $\tau$ is Hausdorff, the set $(\widetilde{\mathbb{Z}}_{+})_{i,j}^0$ is closed.

$(ii)$ Suppose to the contrary that there exists an open neighbourhood $U_\textsf{0}$ of zero in $\left(\textbf{B}(\mathbb{Z}_{+},\theta)^0,\tau\right)$ which intersects finitely many elements of the family $\widetilde{\mathfrak{Z}}$. The definition of the topology $\tau_{\mathbf{cf}}$ on $\textbf{B}(\mathbb{Z}_{+},\theta)$ implies that the set $(\widetilde{\mathbb{Z}}_{+})_{i,j}^0$ is compact for all non-negative integers $i$ and $j$, and hence our assumption implies that there exists finitely many $(\widetilde{\mathbb{Z}}_{+})_{i_1,j_1}^0,\ldots,(\widetilde{\mathbb{Z}}_{+})_{i_k,j_k}^0\in \widetilde{\mathfrak{Z}}$ such that $U_\textsf{0}\setminus\{\textsf{0}\}\subseteq \widetilde{\mathbb{Z}}_{+})_{i_1,j_1}^0\cup\ldots\cup(\widetilde{\mathbb{Z}}_{+})_{i_k,j_k}^0$. This implies that zero $\textsf{0}$ is an isolated point of $\left(\textbf{B}(\mathbb{Z}_{+},\theta)^0,\tau\right)$, a contradiction.

$(iii)$ Fix an arbitrary open neighbourhood $U_\textsf{0}$ of zero  with the compact closure and an open neighbourhood $V_\textsf{0}$ of zero in $\left(\textbf{B}(\mathbb{Z}_{+},\theta)^0,\tau\right)$. Then the family
\begin{equation*}
\mathscr{V}=\left\{V_\textsf{0}, \left\{(\widetilde{\mathbb{Z}}_{+})_{i,j}^0\colon i,j\in\mathbb{N}_0\right\}\right\}
\end{equation*}
is an open cover of $\operatorname{cl}_{\left(\textbf{B}(\mathbb{Z}_{+},\theta)^0,\tau\right)}(U_\textsf{0})$. Since $\operatorname{cl}_{\left(\textbf{B}(\mathbb{Z}_{+},\theta)^0,\tau\right)}(U_\textsf{0})$ is compact, $\operatorname{cl}_{\left(\textbf{B}(\mathbb{Z}_{+},\theta)^0,\tau\right)}(U_\textsf{0})\setminus V_\textsf{0}$ intersects finitely many elements of the family $\widetilde{\mathfrak{Z}}$, and hence $U_\textsf{0}\setminus V_\textsf{0}$ intersects finitely many elements of the family $\widetilde{\mathfrak{Z}}$  as well.

$(iv)$ We claim that for every open neighbourhood $U_\textsf{0}$ of zero in $\left(\textbf{B}(\mathbb{Z}_{+},\theta)^0,\tau\right)$ there exists a non-negative integer $i_0$ such that $U_\textsf{0}$ intersects infinitely many elements of the family $\widetilde{\mathfrak{Z}}$ of the form $(\widetilde{\mathbb{Z}}_{+})_{i_0,j}^0$. Indeed, suppose to the contrary and pick an open neighbourhood $U_\textsf{0}$ of zero which hasn't this property. Without loos of generality we may assume that the closure $\operatorname{cl}_{\left(\textbf{B}(\mathbb{Z}_{+},\theta)^0,\tau\right)}(U_\textsf{0})$ is a compact set. Then the separate continuity of the semigroup operation in $\left(\textbf{B}(\mathbb{Z}_{+},\theta)^0,\tau\right)$ implies that there exists an open neighbourhood $V_\textsf{0}\subseteq U_\textsf{0}$ of zero in $\left(\textbf{B}(\mathbb{Z}_{+},\theta)^0,\tau\right)$ such that  $V_\textsf{0}\cdot (1,0_{\mathbb{Z}_{+}},0)\subseteq U_\textsf{0}$. Since
\begin{equation*}
  (i,z,j)\cdot(1,0_{\mathbb{Z}_{+}},0)=
\left\{
  \begin{array}{cl}
    (i+1,0_{\mathbb{Z}_{+}},0), & \hbox{if~} j=0;\\
    (i,z,j-1),                  & \hbox{if~} j\geqslant1,
  \end{array}
\right. \qquad i,j\in\mathbb{N}_0,
\end{equation*}
our assumption implies that the set $U_\textsf{0}\setminus V_\textsf{0}$ intersects infinitely many elements of the family $\widetilde{\mathfrak{Z}}$, which contradicts assertion $(iii)$.

Next we suppose that that for an arbitrary open neighbourhood $W_\textsf{0}$ of zero in $\left(\textbf{B}(\mathbb{Z}_{+},\theta)^0,\tau\right)$ there exists a non-negative integer $i_0$ such that $W_\textsf{0}$ intersects infinitely many elements of the family $\widetilde{\mathfrak{Z}}$ of the form $(\widetilde{\mathbb{Z}}_{+})_{i_0,j}^0$. We shall prove that for an arbitrary non-negative integer $k$ the neighbourhood $W_\textsf{0}$ intersects infinitely many elements of the family $\widetilde{\mathfrak{Z}}$ of the form $(\widetilde{\mathbb{Z}}_{+})_{k,j}^0$. The separate continuity of the semigroup operation in $\left(\textbf{B}(\mathbb{Z}_{+},\theta)^0,\tau\right)$ implies that there exists an open neighbourhood $V_\textsf{0}\subseteq W_\textsf{0}$ of zero in $\left(\textbf{B}(\mathbb{Z}_{+},\theta)^0,\tau\right)$ such that $(k,0_{\mathbb{Z}_{+}},i_0)\cdot V_\textsf{0}\subseteq W_\textsf{0}$, and hence our assertion holds.

The proof of statement $(v)$ is similar to $(iv)$.

$(vi)$ Suppose to the contrary that there exist a non-negative integer $i_0$ and an open neighbourhood $U_\textsf{0}$ of zero with the compact closure in $\left(\textbf{B}(\mathbb{Z}_{+},\theta)^0,\tau\right)$ such that the family $\mathfrak{H}_{i_0}$ is infinite. Then the separate continuity of the semigroup operation in $\left(\textbf{B}(\mathbb{Z}_{+},\theta)^0,\tau\right)$ implies that there exists an open neighbourhood $V_\textsf{0}\subseteq U_\textsf{0}$ of zero in $\left(\textbf{B}(\mathbb{Z}_{+},\theta)^0,\tau\right)$ such that $V_\textsf{0}\cdot (1,0_{\mathbb{Z}_{+}},0)\subseteq U_\textsf{0}$. By $(iv)$ the neighbourhood $U_\textsf{0}$  intersects infinitely many elements of the family $\widetilde{\mathfrak{Z}}$ of the form $(\widetilde{\mathbb{Z}}_{+})_{i_0,j}^0$. Then our assumption and the following formula
\begin{equation*}
  (i,z,j)\cdot(1,0_{\mathbb{Z}_{+}},0)=
\left\{
  \begin{array}{cl}
    (i+1,0_{\mathbb{Z}_{+}},0), & \hbox{if~} j=0;\\
    (i,z,j-1),                  & \hbox{if~} j\geqslant1,
  \end{array}
\right. \qquad i,j\in\mathbb{N}_0,
\end{equation*}
imply that the set $U_\textsf{0}\setminus V_\textsf{0}$ intersects infinitely many elements of the family $\widetilde{\mathfrak{Z}}$, which contradicts assertion $(iii)$.

The proof of statement $(vii)$ is similar to $(vi)$.

$(viii)$ Fix an arbitrary non-negative integer $i_0$. Suppose to the contrary that there exists an open neighbourhood $U_\textsf{0}$ of zero with compact closure in $\left(\textbf{B}(\mathbb{Z}_{+},\theta)^0,\tau\right)$ such that the set $\left\{(i_0,0_{\mathbb{Z}_{+}},j)\colon j\in \mathbb{N}_0\right\}\setminus U_\textsf{0}$ is infinite. Since for any non-negative integer $j$ the subspace $(\mathbb{Z}_+)_{0,j}$ is open and discrete in $\left(\textbf{B}(\mathbb{Z}_{+},\theta)^0,\tau\right)$, our above assumption implies that the set $\left\{(i_0,0_{\mathbb{Z}_{+}},j)\colon j\in \mathbb{N}_0\right\}\setminus \operatorname{cl}_{\left(\textbf{B}(\mathbb{Z}_{+},\theta)^0,\tau\right)}(U_\textsf{0})$ is infinite too. By $(i)$ for every non-negative integer $j$ the set $(\widetilde{\mathbb{Z}}_{+})_{i_0,j}^0$ is compact and hence the discreetness of $(\mathbb{Z})_{i_0,j}$ implies that $(\widetilde{\mathbb{Z}}_{+})_{i_0,j}^0\cap U_\textsf{0}$ is compact, as well. Then the separate continuity of the semigroup operation in $\left(\textbf{B}(\mathbb{Z}_{+},\theta)^0,\tau\right)$ implies that there exists an open neighbourhood $V_\textsf{0}\subseteq U_\textsf{0}$ of zero in $\left(\textbf{B}(\mathbb{Z}_{+},\theta)^0,\tau\right)$ such that $(i_0,1_{\mathbb{Z}_{+}},i_0)\cdot V_\textsf{0}\subseteq U_\textsf{0}$, where $1_{\mathbb{Z}_{+}}$ is a generator of the additive group of integers $\mathbb{Z}_{+}$. By $(vi)$ there exist finitely many $(\widetilde{\mathbb{Z}}_{+})_{i_0,j}^0\in\widetilde{\mathfrak{Z}}$ such that $(\widetilde{\mathbb{Z}}_{+})_{i_0,j}^0\cap U_\emph{\textsf{0}}=\varnothing$. Then our assumption and the equality
\begin{equation*}
  (i_0,1_{\mathbb{Z}_{+}},i_0)\cdot(i_0,z,j)=(i_0,1_{\mathbb{Z}_{+}}+z,j), \qquad j\in\mathbb{N}_0,
\end{equation*}
imply that the set $U_\textsf{0}\setminus V_\textsf{0}$ intersects infinitely many elements of the family $\widetilde{\mathfrak{Z}}$, which contradicts assertion $(iii)$.

The proof of statement $(ix)$ is similar to $(viii)$.

$(x)$ Suppose to the contrary that there exist a non-negative integer $i_0$, an open neighbourhood $U_\textsf{0}$ of zero with the compact closure in $\left(\textbf{B}(\mathbb{Z}_{+},\theta)^0,\tau\right)$ and  infinitely many subsets of the form $(\mathbb{Z}_{+})_{i_0,j}^0$ such that $(\mathbb{Z}_{+})_{i_0,j}^0\nsubseteq U_\textsf{0}$. Then the separate continuity of the semigroup operation in $\left(\textbf{B}(\mathbb{Z}_{+},\theta)^0,\tau\right)$ implies that there exists an open neighbourhood $V_\textsf{0}\subseteq U_\textsf{0}$ of zero in $\left(\textbf{B}(\mathbb{Z}_{+},\theta)^0,\tau\right)$ such that $(i_0,1_{\mathbb{Z}_{+}},i_0)\cdot V_\textsf{0}\subseteq U_\textsf{0}$, where $1_{\mathbb{Z}_{+}}$ is a generator of the additive group of integers $\mathbb{Z}_{+}$. By statement $(viii)$ there exist finitely many subsets of the form $(\mathbb{Z}_{+})_{i_0,j}^0$ such that $(\mathbb{Z}_{+})_{i_0,j}^0\cap U_\textsf{0}=\varnothing$. Then our assumption and the equality
\begin{equation*}
  (i_0,1_{\mathbb{Z}_{+}},i_0)\cdot(i_0,z,j)=(i_0,1_{\mathbb{Z}_{+}}+z,j), \qquad j\in\mathbb{N}_0,
\end{equation*}
imply that the set $U_\textsf{0}\setminus V_\textsf{0}$ intersects infinitely many elements of the family $\widetilde{\mathfrak{Z}}$, which contradicts statement $(iii)$.

The proof of statement $(xi)$ is similar to $(x)$.

$(xii)$ Suppose to the contrary that there exists an open neighbourhood $U_\textsf{0}$ of zero with the compact closure in $\left(\textbf{B}(\mathbb{Z}_{+},\theta)^0,\tau\right)$ such that
\begin{equation*}
(\mathbb{Z}_{+})_{i_1,j_1}^0\cup\ldots\cup(\mathbb{Z}_{+})_{i_k,j_k}^0\cup U_\textsf{0}\neq\textbf{B}(\mathbb{Z}_{+},\theta)^0
\end{equation*}
for any finite family $\left\{(\mathbb{Z}_{+})_{i_1,j_1}^0,\ldots,(\mathbb{Z}_{+})_{i_k,j_k}^0\right\}$. This assumption, statements $(x)$ and $(xi)$ imply that there exists an infinite sequence $\left\{(i_n,g_n,j_n)\colon n\in\mathbb{N}\right\}$ in $\textbf{B}(\mathbb{Z}_{+},\theta)^0\setminus U_\textsf{0}$ such that $g_n\in \mathbb{Z}_{+}$ for $n\in\mathbb{N}$, $\left\{i_n\colon n\in\mathbb{N}\right\}$ and  $\left\{j_n\colon n\in\mathbb{N}\right\}$ are increasing sequences of the positive integers.

Also, statements $(x)$ and $(xi)$ imply that without loss of generality we may assume that for elements of the sequence $\left\{(i_n,g_n,j_n)\colon n\in\mathbb{N}\right\}$ the following condition holds:
\begin{equation*}
  (i_n+1,g_n,j_n), (i_n,g_n,j_n+1)\in U_\textsf{0}.
\end{equation*}
By the separate continuity of the semigroup operation in $\left(\textbf{B}(\mathbb{Z}_{+},\theta)^0,\tau\right)$ there exists an open neighbourhood $V_\textsf{0}\subseteq U_\textsf{0}$ of zero in $\left(\textbf{B}(\mathbb{Z}_{+},\theta)^0,\tau\right)$ such that
\begin{equation*}
(0,0_{\mathbb{Z}_{+}},1)\cdot V_\textsf{0}\subseteq U_\textsf{0} \qquad \hbox{and} \qquad V_\textsf{0}\cdot(1,0_{\mathbb{Z}_{+}},0)\subseteq U_\textsf{0}.
\end{equation*}
Then statements $(x)$ and $(xi)$, our assumption and the equalities
\begin{equation*}
  (0,0_{\mathbb{Z}_{+}},1)\cdot(i,z,j)=(i-1,z,j) \qquad \hbox{and} \qquad (i,z,j)\cdot(1,0_{\mathbb{Z}_{+}},0)=(i,z,j-1), \qquad i,j\in\mathbb{N}_0,
\end{equation*}
imply that the set $U_\textsf{0}\setminus V_\textsf{0}$ intersects infinitely many elements of the family $\widetilde{\mathfrak{Z}}$, which contradicts $(iii)$.
\end{proof}

\begin{theorem}\label{theorem-4.13}
Let $\theta$ be the annihilating homomorphism and $\tau$ be a shift-continuous locally compact topology on \emph{$\textbf{B}(\mathbb{Z}_{+},\theta)^0$} such that the induced topology of $\tau$ on \emph{$\textbf{B}(\mathbb{Z}_{+},\theta)$} coincides with the topology $\tau_{\mathbf{cf}}$. Then either \emph{$\left(\textbf{B}(\mathbb{Z}_{+},\theta)^0,\tau\right)$} is compact or zero is an isolated point in \emph{$\left(\textbf{B}(\mathbb{Z}_{+},\theta)^0,\tau\right)$}. Moreover, \emph{$\left(\textbf{B}(\mathbb{Z}_{+},\theta)^0,\tau\right)$} is compact, then the family
\begin{equation*}
  \mathscr{B}_\textbf{0}=\left\{O_\emph{\textbf{0}}=\textbf{B}(\mathbb{Z}_{+},\theta)^0\setminus\left(M_{p_1}\cup\cdots\cup M_{p_k}\right)\colon p_1,\ldots,p_k\in\mathbb{N}\right\},
\end{equation*}
where for any positive integer $p$ the set $M_p$ is one of the following sets: $\left\{\left(0,0_{\mathbb{Z}_{+}},j\right)\right\}$, $\left\{\left(i,0_{\mathbb{Z}_{+}},0\right)\right\}$ or $(\widetilde{\mathbb{Z}}_{+})_{m,n}^0$, $i,j\in\mathbb{N}_0$, $m,n\in\mathbb{N}$, defines a base of the topology $\tau$ at zero $\emph{\textsf{0}}$ of \emph{$\textbf{B}(\mathbb{Z}_{+},\theta)^0$}.
\end{theorem}

\begin{proof}
It is obvious that $\left(\textbf{B}(\mathbb{Z}_{+},\theta),\tau_{\mathbf{cf}}\right)$ with the adjoined isolated zero is a locally compact semitopological semigroup.

Later we assume that zero $\textsf{0}$ is non-isolated point of $\left(\textbf{B}(\mathbb{Z}_{+},\theta)^0,\tau\right)$.

By Proposition~\ref{proposition-4.12}$(i)$ for any non-negative integers $i$ and $j$ the set $(\widetilde{\mathbb{Z}}_{+})_{i,j}^0$ is open and compact in $\left(\textbf{B}(\mathbb{Z}_{+},\theta)^0,\tau\right)$, and hence by Proposition~\ref{proposition-2.4}$(i)$ and the definition of the topology $\tau_{\mathbf{cf}}$ on $\textbf{B}(\mathbb{Z}_{+},\theta)$ we have that $(0,0_{\mathbb{Z}_{+}},i)$ and $(j,0_{\mathbb{Z}_{+}},0)$ are isolated point in $\left(\textbf{B}(\mathbb{Z}_{+},\theta)^0,\tau\right)$  for all non-negative integers $i$ and~$j$.  For any non-negative integers $i$ and $j$ put
\begin{equation*}
  A_{i,j}=
\left\{
  \begin{array}{ll}
    (\widetilde{\mathbb{Z}}_{+})_{0,j}^0\cup\left\{\left(0,0_{\mathbb{Z}_{+}},j\right)\right\}, & \hbox{if~} i=0;\\
    (\widetilde{\mathbb{Z}}_{+})_{i,0}^0\cup\left\{\left(i,0_{\mathbb{Z}_{+}},0\right)\right\}, & \hbox{if~} j=0;\\
    (\widetilde{\mathbb{Z}}_{+})_{i,j}^0\cup(\widetilde{\mathbb{Z}}_{+})_{i-1,j-1}^0, & \hbox{if~} i\neq 0 \hbox{~and~} j\neq 0.
  \end{array}
\right.
\end{equation*}
Then for any non-negative integers $i$ and $j$ we have that $(\mathbb{Z}_{+})_{i,j}^0\subseteq A_{i,j}$ and $A_{i,j}$ is a compact open subset of $\left(\textbf{B}(\mathbb{Z}_{+},\theta)^0,\tau\right)$. By Proposition~\ref{proposition-4.12}$(xii)$ for an arbitrary open neighbourhood $U_\textbf{0}$ of zero in $\left(\textbf{B}(\mathbb{Z}_{+},\theta)^0,\tau\right)$ there exist finitely many subsets $(\mathbb{Z}_{+})_{i_1,j_1}^0,\ldots,(\mathbb{Z}_{+})_{i_k,j_k}^0$ such that $(\mathbb{Z}_{+})_{i_1,j_1}^0\cup\ldots\cup(\mathbb{Z}_{+})_{i_k,j_k}^0\cup U_\textsf{0}=\textbf{B}(\mathbb{Z}_{+},\theta)^0$. Then $O(U_\textbf{0})=\textbf{B}(\mathbb{Z}_{+},\theta)^0\setminus\left(A_{i_1,j_1}\cup\cdots\cup A_{i_k,j_k}\right)$ is an open-and-closed neighbourhood of zero in $\left(\textbf{B}(\mathbb{Z}_{+},\theta)^0,\tau\right)$ such that $O(U_\textbf{0})\subseteq U_\textbf{0}$. This implies that if the family $\left\{U_\textbf{0}^\alpha\right\}_{\alpha\in\Omega}$ is a base of the topology $\tau$ at zero of $\textbf{B}(\mathbb{Z}_{+},\theta)^0$, then so constructed family $\left\{O(U_\textbf{0}^\alpha)\right\}_{\alpha\in\Omega}$ is a base of the topology $\tau$ at zero of $\textbf{B}(\mathbb{Z}_{+},\theta)^0$ too. This completes the last statement of the theorem.
\end{proof}

For all non-negative integers $k$ and $l$ we put  $\left(-\mathbb{N}\right)_{k,l}=\left\{(k,-n,l)\colon n\in\mathbb{N}\right\}$, $\mathfrak{N}^+=\left\{\mathbb{N}_{i,j}^0\colon i,j\in \mathbb{N}_0\right\}$ and $\mathfrak{N}^-=\left\{\left(-\mathbb{N}\right)_{i,j}\colon i,j\in \mathbb{N}_0\right\}$.

\begin{proposition}\label{proposition-4.14}
If $\theta$ is the annihilating homomorphism, $\tau$ is a locally compact shift-continuous topology on \emph{$\textbf{B}(\mathbb{Z}_{+},\theta)^0$} such that zero $\emph{\textsf{0}}$ is a non-isolated point in \emph{$\left(\textbf{B}(\mathbb{Z}_{+},\theta)^0,\tau\right)$} and the induced topology of $\tau$ on \emph{$\textbf{B}(\mathbb{Z}_{+},\theta)$} coincides with $\tau_{+}$, then the following assertions hold:

\begin{itemize}
  \item[$(i)$] for any non-negative integers $i$ and $j$ the set $\mathbb{N}^0_{i,j}$ is open and compact in \emph{$\left(\textbf{B}(\mathbb{Z}_{+},\theta)^0,\tau\right)$} and hence it is closed;

  \item[$(ii)$] every open neighbourhood $U_\emph{\textsf{0}}$ of zero in \emph{$\left(\textbf{B}(\mathbb{Z}_{+},\theta)^0,\tau\right)$} intersects infinitely many elements of the family $\mathfrak{N}^+$;

  \item[$(iii)$] for any non-negative integer $i_0$ every open neighbourhood $U_\emph{\textsf{0}}$ of zero in \emph{$\left(\textbf{B}(\mathbb{Z}_{+},\theta)^0,\tau\right)$} intersects almost all elements of the family $\mathfrak{N}^+$ of the form $\mathbb{N}_{i_0,j}^0$;

  \item[$(iv)$] every open neighbourhood $U_\emph{\textsf{0}}$ of zero in \emph{$\left(\textbf{B}(\mathbb{Z}_{+},\theta)^0,\tau\right)$} contains almost all elements of the form $(0,0_{\mathbb{Z}_{+}},j)$, $j\in \mathbb{N}_0$;

  \item[$(v)$] every open neighbourhood $U_\emph{\textsf{0}}$ of zero in \emph{$\left(\textbf{B}(\mathbb{Z}_{+},\theta)^0,\tau\right)$} contains almost all subsets of the form $(\mathbb{Z}_{+})_{0,j}^0$, $j\in \mathbb{N}_0$.
\end{itemize}
\end{proposition}

\begin{proof}
$(i)$ Since $\textbf{B}(\mathbb{Z}_{+},\theta)$ is an open subset of $\left(\textbf{B}(\mathbb{Z}_{+},\theta)^0,\tau\right)$, by the definition of $\tau_{+}$ the set $\mathbb{N}_{i,j}^0$ is open for any non-negative integers $i$ and $j$. Moreover, the definition of the topology $\tau_{+}$ on  $\textbf{B}(\mathbb{Z}_{+},\theta)$ implies that $\mathbb{N}_{i,j}^0$ is compact, and since $\tau$ is Hausdorff, $\mathbb{N}_{i,j}^0$ is closed.

$(ii)$ Suppose to the contrary that there exists an open neighbourhood $U_\textsf{0}$ of zero with the compact closure in $\left(\textbf{B}(\mathbb{Z}_{+},\theta)^0,\tau\right)$ which intersects finitely many elements of the family $\mathfrak{N}^+$. By the definition of the topology $\tau_{+}$ on $\textbf{B}(\mathbb{Z}_{+},\theta)$ the set $\mathbb{N}_{i,j}^0$ is compact for all non-negative integers $i$ and $j$, and hence without loss of generality we may assume that $U_\textsf{0}\cap\bigcup\mathfrak{N}^+=\varnothing$. By the equality
\begin{equation*}
  (k,0_{\mathbb{Z}_{+}},k)\cdot(i,z,j)=
\left\{
  \begin{array}{ll}
    (k,0_{\mathbb{Z}_{+}},k+j-i), & \hbox{if~} k>i;\\
    (i,z,j),                      & \hbox{if~}  k\leqslant i,
  \end{array}
\right. \qquad i,j,k\in\mathbb{N}_0,
\end{equation*}
the separate continuity of the semigroup operation in $\left(\textbf{B}(\mathbb{Z}_{+},\theta)^0,\tau\right)$ and assertion $(i)$  the set
\begin{equation*}
  \rho_{(k+1,0_{\mathbb{Z}_{+}},k+1)}^{-1}\left(\mathbb{N}^0_{k,l}\right)=
\left\{
  \begin{array}{ll}
    (\widetilde{\mathbb{Z}}_{+})_{k,l}^0\cup(\widetilde{\mathbb{Z}}_{+})_{k-1,l-1}^0\cup\cdots\cup(\widetilde{\mathbb{Z}}_{+})_{k-l,0}^0\cup \left\{(k-l,0_\mathbb{Z},0)\right\}, & \hbox{if~} k\geqslant l;\\
    (\widetilde{\mathbb{Z}}_{+})_{k,l}^0\cup(\widetilde{\mathbb{Z}}_{+})_{k-1,l-1}^0\cup\cdots\cup(\widetilde{\mathbb{Z}}_{+})_{0,l-k}^0\cup \left\{(0,0_\mathbb{Z},l-k)\right\}, & \hbox{if~} k\leqslant l
  \end{array}
\right.
\end{equation*}
is open-and-closed in $\left(\textbf{B}(\mathbb{Z}_{+},\theta)^0,\tau\right)$. This and the definition of the topology $\tau_+$ imply that the set $(\widetilde{\mathbb{Z}}_{+})_{k,l}^0$ is opens-and-closed and the points $(0,0_{\mathbb{Z}_{+}},k)$ and $(k,0_{\mathbb{Z}_{+}},0)$ are isolated in $\left(\textbf{B}(\mathbb{Z}_{+},\theta)^0,\tau\right)$ for all $k,l\in\mathbb{N}_0$. By the definition of the topology $\tau_+$ and statement $(i)$ we have that $\left(-\mathbb{N}\right)_{k,l}=(\widetilde{\mathbb{Z}}_{+})_{k,l}^0\setminus\mathbb{N}_{k,l}^0$ is an open-and-closed discrete subspace of $\left(\textbf{B}(\mathbb{Z}_{+},\theta)^0,\tau\right)$ for all $k,l\in \mathbb{N}_0$. This implies that the set $U_\textsf{0}\cap\left(-\mathbb{N}\right)_{k,l}$ is finite for all $k,l\in \mathbb{N}_0$. By the separate continuity of the semigroup operation of $\left(\textbf{B}(\mathbb{Z}_{+},\theta)^0,\tau\right)$ there exists an open neighbourhood $V_\textsf{0}\subseteq U_\textsf{0}$  of zero in $\left(\textbf{B}(\mathbb{Z}_{+},\theta)^0,\tau\right)$ such that $V_\textsf{0}\cdot (1,0_{\mathbb{Z}_{+}},0)\subseteq U_\textsf{0}$. Then by the equality
\begin{equation*}
  (i,z,j)\cdot(1,0_{\mathbb{Z}_{+}},0)=
\left\{
  \begin{array}{cl}
    (i+1,0_{\mathbb{Z}_{+}},0), & \hbox{if~} j=0;\\
    (i,z,j-1),                  & \hbox{if~} j\geqslant1,
  \end{array}
\right. \qquad i,j\in\mathbb{N}_0,
\end{equation*}
we have that $U_\textsf{0}\setminus V_\textsf{0}$ is an infinite subset of the compactum $\operatorname{cl}_{\left(\textbf{B}(\mathbb{Z}_{+},\theta)^0,\tau\right)}(U(0))$, and the definition of the topology $\tau_+$ on $\textbf{B}(\mathbb{Z}_{+},\theta)$ implies that the set $U_\textsf{0}\setminus V_\textsf{0}$ does not have an accumulation point in $\operatorname{cl}_{\left(\textbf{B}(\mathbb{Z}_{+},\theta)^0,\tau\right)}(U_\textsf{0})$, which contradicts the compactness of $\operatorname{cl}_{\left(\textbf{B}(\mathbb{Z}_{+},\theta)^0,\tau\right)}(U_\textsf{0})$.

$(iii)$ We claim that for every open neighbourhood $U_\textsf{0}$ of zero with the compact closure in $\left(\textbf{B}(\mathbb{Z}_{+},\theta)^0,\tau\right)$ there exists a non-negative integer $i_0$ such that $U_\textsf{0}$ intersects infinitely many elements  of the form $\mathbb{N}_{i_0,j}^0$. Indeed, assume the contrary and pick $U_\textsf{0}$ which has not this property. Without loss of generality we may assume that the neighbourhood $U_\textsf{0}$ has the compact closure $\operatorname{cl}_{\left(\textbf{B}(\mathbb{Z}_{+},\theta)^0,\tau\right)}(U_\textsf{0})$. By item $(ii)$ there exists an increasing sequence $\left\{i_n\right\}_{n\in\mathbb{N}}$ of positive integers such that $\mathbb{N}_{i_n,j_n}^0\cap U_\textsf{0}\neq\varnothing$ for some sequence of non-negative integers $\left\{j_n\right\}_{n\in\mathbb{N}}$. Then for every element $i_p$ of the sequence $\left\{i_n\right\}_{n\in\mathbb{N}}$ there exits a maximum non-negative integer $j_{i_p}$ such that $\mathbb{N}_{i_p,j_{i_p}}^0\cap U_\textsf{0}\neq\varnothing$ and $\mathbb{N}_{i_p,j_{i_p}+k}^0\cap U_\textsf{0}=\varnothing$ for any positive integer $k$. By the separate continuity of the semigroup operation of $\left(\textbf{B}(\mathbb{Z}_{+},\theta)^0,\tau\right)$ there exists an open neighbourhood $V_\textsf{0}\subseteq U_\textsf{0}$ of zero in $\left(\textbf{B}(\mathbb{Z}_{+},\theta)^0,\tau\right)$ such that $V_\textsf{0}\cdot (0,0_{\mathbb{Z}_{+}},1)\subseteq U_\textsf{0}$. The above arguments imply that there exists a sequence of distinct points $\left\{\left(i_n,z_n,j_{i_n}\right)\right\}_{n\in\mathbb{N}}\subseteq U_\textsf{0}\setminus V_\textsf{0} \subset\operatorname{cl}_{\left(\textbf{B}(\mathbb{Z}_{+},\theta)^0,\tau\right)}(U_\textsf{0})$. Then the definition of the topology $\tau_+$ on $\textbf{B}(\mathbb{Z}_{+},\theta)$ implies that this sequence has not an accumulation point in the set  $\operatorname{cl}_{\left(\textbf{B}(\mathbb{Z}_{+},\theta)^0,\tau\right)}(U_\textsf{0})$ which contradicts the compactness of $\operatorname{cl}_{\left(\textbf{B}(\mathbb{Z}_{+},\theta)^0,\tau\right)}(U_\textsf{0})$.

By the above for every open neighbourhood $U_\textsf{0}$ of zero in $\left(\textbf{B}(\mathbb{Z}_{+},\theta)^0,\tau\right)$ there exists a non-negative integer $i_0$ such that $U_\textsf{0}\cap\mathbb{N}_{i_0,j}^0\neq\varnothing$ for infinitely many elements of the family $\mathfrak{N}^+$ of the form $\mathbb{N}_{i_0,j}^0$. Since the semigroup operation in $\left(\textbf{B}(\mathbb{Z}_{+},\theta)^0,\tau\right)$ is separately continuous, for any non-negative integer $i$ there exists an open neighbourhood $V_\textsf{0}\subseteq U_\textsf{0}$ of zero in $\left(\textbf{B}(\mathbb{Z}_{+},\theta)^0,\tau\right)$ such that $(i,0_{\mathbb{Z}_{+}},i_0)\cdot V_\textsf{0}\subseteq U_\textsf{0}$. Thus $U_\textsf{0}\cap\mathbb{N}_{i,j}^0\neq\varnothing$ for infinitely many elements of the family $\mathfrak{N}^+$ of the form $\mathbb{N}_{i,j}^0$. Hence we get that $U_\textsf{0}\cap\mathbb{N}_{0,j}^0\neq\varnothing$ for infinitely many elements of the family $\mathfrak{N}^+$ of the form $\mathbb{N}_{0,j}^0$.

Suppose that there exists an infinite increasing sequence $\left\{j_n\right\}_{n\in\mathbb{N}}$ such that $U_\textsf{0}\cap\mathbb{N}_{0,j_n}^0=\varnothing$ for any elements $j_n$ of $\left\{j_n\right\}_{n\in\mathbb{N}}$. Without loss of generality we may assume that the sequence $\left\{j_n\right\}_{n\in\mathbb{N}}$ is maximal, i.e., $U_\textsf{0}\cap\mathbb{N}_{0,j}^0\neq\varnothing$ for any non-negative integer $j\notin\left\{j_n\right\}_{n\in\mathbb{N}}$. Then there exists a subsequence $\left\{j_{n_k}\right\}_{k\in\mathbb{N}}$ of $\left\{j_n\right\}_{n\in\mathbb{N}}$ such that $U_\textsf{0}\cap\mathbb{N}_{0,j_{n_k}}^0=\varnothing$ and $U_\textsf{0}\cap\mathbb{N}_{0,j_{n_k}+1}^0\neq\varnothing$. The separate continuity of the semigroup operation in $\left(\textbf{B}(\mathbb{Z}_{+},\theta)^0,\tau\right)$ implies that there exists an open neighbourhood $V_\textsf{0}\subseteq U_\textsf{0}$ of zero in $\left(\textbf{B}(\mathbb{Z}_{+},\theta)^0,\tau\right)$ such that $(1,0_{\mathbb{Z}_{+}},0)\cdot V_\textsf{0}\subseteq U_\textsf{0}$. It is obvious that our above arguments imply that there exists a sequence of distinct points $\left\{\left(0,z_{n_k},j_{n_k}\right)\right\}_{n\in\mathbb{N}}\subseteq U_\textsf{0}\setminus V_\textsf{0}$ which is a subset of the compactum $\operatorname{cl}_{\left(\textbf{B}(\mathbb{Z}_{+},\theta)^0,\tau\right)}(U_\textsf{0})$. Then the definition of the topology $\tau_+$ on $\textbf{B}(\mathbb{Z}_{+},\theta)$ implies that this sequence has not an accumulation point in $\operatorname{cl}_{\left(\textbf{B}(\mathbb{Z}_{+},\theta)^0,\tau\right)}(U_\textsf{0})$ which contradicts the compactness of $\operatorname{cl}_{\left(\textbf{B}(\mathbb{Z}_{+},\theta)^0,\tau\right)}(U_\textsf{0})$. Hence we have that every open neighbourhood $U_\textsf{0}$ of zero in $\left(\textbf{B}(\mathbb{Z}_{+},\theta)^0,\tau\right)$ intersects almost all elements of the form $\mathbb{N}_{0,j}^0$. Again, since the semigroup operation in $\left(\textbf{B}(\mathbb{Z}_{+},\theta)^0,\tau\right)$ is separately continuous, for any non-negative integer $i$ there exists an open neighbourhood $W_\textsf{0}\subseteq U_\textsf{0}$ of zero in $\left(\textbf{B}(\mathbb{Z}_{+},\theta)^0,\tau\right)$ such that $(i,0_{\mathbb{Z}_{+}},i_0)\cdot W_\textsf{0}\subseteq U_\textsf{0}$. The last inclusion implies assertion $(iii)$.

$(iv)$ Fix an arbitrary open neighbourhood $U_\textsf{0}$ of zero in $\left(\textbf{B}(\mathbb{Z}_{+},\theta)^0,\tau\right)$. Without loss of generality we may assume that the neighbourhood $U_\textsf{0}$ has the compact closure $\operatorname{cl}_{\left(\textbf{B}(\mathbb{Z}_{+},\theta)^0,\tau\right)}(U_\textsf{0})$. By statement $(iii)$ the neighbourhood $U_\textsf{0}$ intersects almost all elements of the family $\mathfrak{N}^+$ of the form $\mathbb{N}_{0,j}^0$ and by the separate continuity of the semigroup operation in $\left(\textbf{B}(\mathbb{Z}_{+},\theta)^0,\tau\right)$ there exists an open neighbourhood $V_\textsf{0}\subseteq U_\textsf{0}$ of zero in $\left(\textbf{B}(\mathbb{Z}_{+},\theta)^0,\tau\right)$ such that  $(1,0_{\mathbb{Z}_{+}},1)\cdot V_\textsf{0}\subseteq U_\textsf{0}$. Since $(1,0_{\mathbb{Z}_{+}},1)\cdot(0,z,k)=(0,0_{\mathbb{Z}_{+}},k)$ the inclusion $(1,0_{\mathbb{Z}_{+}},1)\cdot V_\textsf{0}\subseteq U_\textsf{0}$ implies our assertion.

$(v)$ Fix an arbitrary open neighbourhood $U_\textsf{0}$ of zero with the compact closure in $\left(\textbf{B}(\mathbb{Z}_{+},\theta)^0,\tau\right)$.  The separate continuity of the semigroup operation of $\left(\textbf{B}(\mathbb{Z}_{+},\theta)^0,\tau\right)$ implies that for every element $k$ of the additive group of integers $\mathbb{Z}_{+}$ there exists an open neighbourhood $V_\textsf{0}\subseteq U_\textsf{0}$ of zero in $\left(\textbf{B}(\mathbb{Z}_{+},\theta)^0,\tau\right)$ such that $(0,k,0)\cdot V_\textsf{0}\subseteq U_\textsf{0}$. By statement $(iv)$ the neighbourhood $U_\textsf{0}$ contains almost all elements of the form $(0,0_{\mathbb{Z}_{+}},j)$ and hence the above arguments imply  statement $(v)$.
\end{proof}

\begin{theorem}\label{theorem-4.15}
Let $\theta$ be the annihilating homomorphism and $\tau$ be a shift-continuous locally compact topology on \emph{$\textbf{B}(\mathbb{Z}_{+},\theta)^0$} such that the induced topology of $\tau$ onto \emph{$\textbf{B}(\mathbb{Z}_{+},\theta)$} coincides with the topology $\tau_{+}$. Then zero $\emph{\textsf{0}}$ of \emph{$\textbf{B}(\mathbb{Z}_{+},\theta)^0$} is an isolated point in \emph{$\left(\textbf{B}(\mathbb{Z}_{+},\theta)^0,\tau\right)$}.
\end{theorem}

\begin{proof}
First we observe that Example~\ref{example-4.7} and Proposition~\ref{proposition-4.14} imply that for all non-negative integers $i$ and $j$ and any $n\in\mathbb{Z}_{+}$ the point $(i,n,j)$ is isolated in $(\mathbb{Z}_{+})_{i,j}^0$. Then the proof of assertion $(ii)$ of Proposition~\ref{proposition-4.14} implies that the set $(\mathbb{Z}_{+})_{0,j}^0$ is open-and-closed in $\left(\textbf{B}(\mathbb{Z}_{+},\theta)^0,\tau\right)$ for any non-negative integer $j$.

Fix an arbitrary open neighbourhood $U_\textsf{0}$ of zero with the compact closure in $\left(\textbf{B}(\mathbb{Z}_{+},\theta)^0,\tau\right)$. By Proposition~\ref{proposition-4.14}$(v)$ there exists a positive integer $j_0$ such that $(\mathbb{Z}_{+})_{0,j_0}^0\subseteq U_\textsf{0}$ and hence $(\mathbb{Z}_{+})_{0,j}^0$ is an open-and-closed subset of $\operatorname{cl}_{\left(\textbf{B}(\mathbb{Z}_{+},\theta)^0,\tau\right)}(U(0))$. But $(\mathbb{Z}_{+})_{0,j}^0$ is not compact, which contradicts the compactness of $\operatorname{cl}_{\left(\textbf{B}(\mathbb{Z}_{+},\theta)^0,\tau\right)}(U(0))$.
\end{proof}

The proof of the following theorem is similar to Theorem~\ref{theorem-4.15}.

\begin{theorem}\label{theorem-4.17}
Let $\theta$ be the annihilating homomorphism and $\tau$ be a shift-continuous locally compact topology on \emph{$\textbf{B}(\mathbb{Z}_{+},\theta)^0$} such that the induced topology of $\tau$ onto \emph{$\textbf{B}(\mathbb{Z}_{+},\theta)$} coincides with the topology $\tau_{-}$. Then zero $\emph{\textsf{0}}$ of \emph{$\textbf{B}(\mathbb{Z}_{+},\theta)^0$} is an isolated point in \emph{$\left(\textbf{B}(\mathbb{Z}_{+},\theta)^0,\tau\right)$}.
\end{theorem}

Later we need the following two folklore lemmas.

\begin{lemma}\label{lemma-4.19}
Let $\theta\colon \mathbb{Z}_{+}\to \mathbb{Z}_{+}$ be a homomorphism. Then the image $\theta(\mathbb{Z}_{+})$ is isomorphic to the additive group of integers $\mathbb{Z}_{+}$ if and only if $\theta$ is non-annihilating.
\end{lemma}

\begin{proof}
The implication $(\Rightarrow)$ is trivial.

$(\Leftarrow)$ Suppose that a homomorphism $\theta$ is non-annihilating. Since the group $\mathbb{Z}_{+}$ is generated by the element $1$ we have that $\theta(1)=n\neq 0_{\mathbb{Z}_{+}}\in\theta(\mathbb{Z}_{+})\subseteq \mathbb{Z}_{+}$ for some integer $n$. This implies that the image $\theta(\mathbb{Z}_{+})$ generated by the element $n$ as a subgroup of $\mathbb{Z}_{+}$, and since $\mathbb{Z}_{+}$ is isomorphic to the free group over a singleton set, $\theta(\mathbb{Z}_{+})$ is isomorphic to the subgroup $n\mathbb{Z}_{+}=\{nk\colon k\in \mathbb{Z}_{+}\}$ of $\mathbb{Z}_{+}$. It is obvious that $n\mathbb{Z}_{+}$ is isomorphic to $\mathbb{Z}_{+}$.
\end{proof}

\begin{lemma}\label{lemma-4.20}
If $\theta\colon \mathbb{Z}_{+}\to \mathbb{Z}_{+}$ is an arbitrary non-annihilating homomorphism then for arbitrary positive integer $i$ and an arbitrary $a\in\mathbb{Z}_{+}$ the equation $\theta^{i}(x)=a$ has at most one solution.
\end{lemma}

\begin{proof}
Suppose that $\theta(1)=n$ for some $n\in\mathbb{Z}_{+}\setminus \{0_{\mathbb{Z}_{+}}\}$. Since the homomorphism $\theta\colon \mathbb{Z}_{+}\to \mathbb{Z}_{+}$ is non-annihilating, it is obvious that the equation $\theta^{i}(x)=a$ has a unique solution if and only if $a=n^ib$ for some integer $b$, and in the other case the equation $(x)\theta^{i}=a$ hasn't a solution.
\end{proof}

\begin{proposition}\label{proposition-4.21}
Let $\theta\colon \mathbb{Z}_{+}\to \mathbb{Z}_{+}$ be an arbitrary non-annihilating homomorphism. Then both equations $\alpha\cdot \chi=\beta$ and $\chi\cdot\gamma=\delta$ have finitely many solutions in \emph{$\textbf{B}(\mathbb{Z}_{+},\theta)$}.
\end{proposition}

\begin{proof}
We consider only the case of $\alpha\cdot \chi=\beta$. The proof in the other case is similar.

Put
\begin{equation*}
  \alpha=\left(n_1,z_1,m_1\right), \quad \beta=\left(n_2,z_2,m_2\right) \quad \hbox{and} \quad \chi=\left(n,z,m\right).
\end{equation*}
Then the semigroup operation of $\textbf{B}(\mathbb{Z}_{+},\theta)$ implies that
\begin{equation*}
  \left(n_2,z_2,m_2\right)=\left(n_1,z_1,m_1\right)\cdot\left(n,z,m\right)=
\left\{
  \begin{array}{cl}
    \left(n_1,z_1+\theta^{m_1-n}(z),m_1-n+m\right), & \hbox{if~} m_1>n;\\
    \left(n_1,z_1+z,m\right),                       & \hbox{if~} m_1=n;\\
    \left(n_1-m_1+n,\theta^{n-m_1}(z_1)+z,m\right), & \hbox{if~} m_1<n.
  \end{array}
\right.
\end{equation*}
Then
\begin{itemize}
  \item[$(1)$] in the case when $m_1<n$ we have that $n=n_2-n_1+m_1$, $m=m_2$ and $z=z_2-\theta^{n-m_1}(z_1)$;
  \item[$(2)$] in the case when $m_1=n$ we have that $m=m_2$ and $z=z_2-z_1$;
  \item[$(3)$] in the case when $m_1>n$ we have that $n_1=n_2$, $n-m=m_2-m_1$ and $\theta^{m_1-n}(z)=z_2-z_1$.
\end{itemize}
Now, the above three cases and Lemma~\ref{lemma-4.20} imply the statement of the proposition.
\end{proof}

We recall that a topological space $X$ is said to be \emph{Baire}, if for each sequence $A_1, A_2,\ldots, A_i,\ldots$ of open dense subsets of $X$ the intersection $\bigcap_{i=1}^\infty A_i$ is dense in $X$. \cite{Haworth-McCoy-1977}. It is well known that every \v{C}ech-complete (and hence every locally compact) space is Baire (see \cite[Section~3.9]{Engelking-1989}).

The following theorem describes Baire $T_1$-semitopological semigroups $\left(\textbf{B}(\mathbb{Z}_{+},\theta),\tau\right)$ with a non-an\-ni\-hi\-la\-ting homomorphism $\theta$.

\begin{theorem}\label{theorem-4.22}
If $\theta$ is an arbitrary non-annihilating homomorphism then every shift-continuous Baire $T_1$-topology on \emph{$\textbf{B}(\mathbb{Z}_{+},\theta)$} is discrete.
\end{theorem}

\begin{proof}
Since the space $\left(\textbf{B}(\mathbb{Z}_{+},\theta),\tau\right)$ is Baire and countable, $\left(\textbf{B}(\mathbb{Z}_{+},\theta),\tau\right)$ contains an isolated point $(i_0,z_0,j_0)$, where $i_0$ and $j_0$ are non-negative integers and $z_0\in \mathbb{Z}_{+}$. Fix an arbitrary element $(i_1,z_1,j_1)\in \textbf{B}(\mathbb{Z}_{+},\theta)$. Then
\begin{equation*}
  (i_0,z_0-z_1,i_1)(i_1,z_1,j_1)(j_1, 0_{\mathbb{Z}_{+}},j_0)=(i_0,z_0,j_0).
\end{equation*}
Proposition~\ref{proposition-4.21} implies that the equation
\begin{equation*}
  (i_0,z_0-z_1,i_1)\cdot\chi\cdot(j_1, 0_{\mathbb{Z}_{+}},j_0)=(i_0,z_0,j_0)
\end{equation*}
has a finite non-empty set of solutions. Since $\tau$ is a $T_1$-topology on $\textbf{B}(\mathbb{Z}_{+},\theta)$ the separate continuity of the semigroup operation in $\left(\textbf{B}(\mathbb{Z}_{+},\theta),\tau\right)$ implies that the map $f\colon \textbf{B}(\mathbb{Z}_{+},\theta)\to \textbf{B}(\mathbb{Z}_{+},\theta)$, $f(x)=(i_0,z_0-z_1,i_1)\cdot x\cdot(j_1, 0_{\mathbb{Z}_{+}},j_0)$ is continuous and hence $(i_1,z_1,j_1)$ is an isolated point of $\left(\textbf{B}(\mathbb{Z}_{+},\theta),\tau\right)$. Therefore, $\left(\textbf{B}(\mathbb{Z}_{+},\theta),\tau\right)$ is the discrete space.
\end{proof}

Theorem~\ref{theorem-4.22} implies the following corollary.

\begin{corollary}\label{corollary-4.23}
If $\theta$ is an arbitrary non-annihilating homomorphism then every shift-continuous locally compact $T_1$-topology on \emph{$\textbf{B}(\mathbb{Z}_{+},\theta)$} is discrete.
\end{corollary}

Now we obtain the description of Baire $T_1$-semitopological semigroups $\left(\textbf{B}(\mathbb{Z}_{+},\theta)^0,\tau\right)$ with a non-an\-ni\-hi\-la\-ting homomorphism $\theta$.

\begin{theorem}\label{theorem-4.24}
Let $\theta$ be an arbitrary non-annihilating homomorphism and $\tau$ be a shift-continuous Baire $T_1$-topology on \emph{$\textbf{B}(\mathbb{Z}_{+},\theta)^0$}. Then every non-zero element of the semigroup \emph{$\textbf{B}(\mathbb{Z}_{+}^0,\theta)^0$} is an isolated point in \emph{$\left(\textbf{B}(\mathbb{Z}_{+}^0,\theta)^0,\tau\right)$}.
\end{theorem}

\begin{proof}
Since $\left(\textbf{B}(\mathbb{Z}_{+},\theta)^0,\tau\right)$ is a Baire $T_1$-space, $\textbf{B}(\mathbb{Z}_{+},\theta)$ is its open subspace, and hence by Proposition~1.14 from \cite{Haworth-McCoy-1977}, $\textbf{B}(\mathbb{Z}_{+},\theta)$ is Baire, too. Next we apply Theorem~\ref{theorem-4.22}.
\end{proof}

Theorem~\ref{theorem-4.24} implies the following corollary.

\begin{corollary}\label{corollary-4.25}
Let $\theta$ be an arbitrary non-annihilating homomorphism and $\tau$ be a shift-continuous locally compact $T_1$-topology on \emph{$\textbf{B}(\mathbb{Z}_{+},\theta)^0$}. Then every non-zero element of the semigroup \emph{$\textbf{B}(\mathbb{Z}_{+}^0,\theta)^0$} is an isolated point in \emph{$\left(\textbf{B}(\mathbb{Z}_{+}^0,\theta)^0,\tau\right)$}.
\end{corollary}

\begin{proposition}\label{proposition-4.26}
Let $\theta$ be an arbitrary non-annihilating homomorphism and  \emph{$\left(\textbf{B}(\mathbb{Z}_{+},\theta)^0,\tau\right)$} be a locally compact semitopological semigroup with non-isolated zero $\emph{\textsf{0}}$. Then the following assertions hold:
\begin{itemize}
  \item[$(i)$] every open neighbourhood $U_\emph{\textsf{0}}$ of zero in \emph{$\left(\textbf{B}(\mathbb{Z}_{+},\theta)^0,\tau\right)$} intersects infinitely many sets of the form $\mathbb{Z}_{i,j}$, $i,j\in \mathbb{N}_0$;

  \item[$(ii)$] for any non-negative integer $i_0$ every open neighbourhood $U_\emph{\textsf{0}}$ of zero in \emph{$\left(\textbf{B}(\mathbb{Z}_{+},\theta)^0,\tau\right)$} intersects almost all sets of the form $\mathbb{Z}_{i_0,j}$, $j\in \mathbb{N}_0$;

  \item[$(iii)$]  every open neighbourhood $U_\emph{\textsf{0}}$ of zero in \emph{$\left(\textbf{B}(\mathbb{Z}_{+},\theta)^0,\tau\right)$} contains almost all elements of the form $(0,0_{\mathbb{Z}_{+}},j)$, $j\in \mathbb{N}_0$;

  \item[$(iv)$] every open neighbourhood $U_\emph{\textsf{0}}$ of zero in \emph{$\left(\textbf{B}(\mathbb{Z}_{+},\theta)^0,\tau\right)$} contains almost all subsets of the form $(\mathbb{Z}_{+})_{0,j}$, $j\in \mathbb{N}_0$;

  \item[$(v)$] for any non-negative integer $i_0$ every open neighbourhood $U_\emph{\textsf{0}}$ of zero in \emph{$\left(\textbf{B}(\mathbb{Z}_{+},\theta)^0,\tau\right)$} contains almost all subsets of the form $(\mathbb{Z}_{+})_{i_0,j}$, $j\in \mathbb{N}_0$;

  \item[$(vi)$] for any non-negative integer $j_0$ every open neighbourhood $U_\emph{\textsf{0}}$ of zero in \emph{$\left(\textbf{B}(\mathbb{Z}_{+},\theta)^0,\tau\right)$} contains almost all subsets of the form $(\mathbb{Z}_{+})_{i,j_0}$, $i\in \mathbb{N}_0$;

  \item[$(vii)$] every open neighbourhood $U_\emph{\textsf{0}}$ of zero in \emph{$\left(\textbf{B}(\mathbb{Z}_{+},\theta)^0,\tau\right)$} contains almost all subsets of the form $(\mathbb{Z}_{+})_{i,j}$, $i,j\in \mathbb{N}_0$;

  \item[$(viii)$] for every open neighbourhood $U_\emph{\textsf{0}}$ of zero in \emph{$\left(\textbf{B}(\mathbb{Z}_{+},\theta)^0,\tau\right)$} the set \emph{$\textbf{B}(\mathbb{Z}_{+},\theta)^0\setminus U_\textsf{0}$} is finite.
\end{itemize}
\end{proposition}

\begin{proof}
$(i)$ Suppose to the contrary that there exists an open neighbourhood $U_\textsf{0}$ of zero the compact closure in $\left(\textbf{B}(\mathbb{Z}_{+},\theta)^0,\tau\right)$ which intersects finitely many sets of the form $\mathbb{Z}_{i,j}$.  By the separate continuity of the semigroup operation of $\left(\textbf{B}(\mathbb{Z}_{+},\theta)^0,\tau\right)$ there exists an open neighbourhood $V_\textsf{0}\subseteq U_\textsf{0}$ of zero in $\left(\textbf{B}(\mathbb{Z}_{+},\theta)^0,\tau\right)$ such that  $V_\textsf{0}\cdot (1,0_{\mathbb{Z}_{+}},0)\subseteq U_\textsf{0}$. Then the equality
\begin{equation*}
  (i,z,j)\cdot(1,0_{\mathbb{Z}_{+}},0)=
\left\{
  \begin{array}{cl}
    (i+1,0_{\mathbb{Z}_{+}},0), & \hbox{if~} j=0;\\
    (i,z,j-1),                  & \hbox{if~} j\geqslant1,
  \end{array}
\right. \qquad i,j\in\mathbb{N}_0,
\end{equation*}
implies that $U_\textsf{0}\setminus V_\textsf{0}$ is an infinite subset of the compactum $\operatorname{cl}_{\left(\textbf{B}(\mathbb{Z}_{+},\theta)^0,\tau\right)}(U_\textsf{0})$ and by Theorem~\ref{theorem-4.22}, $\textbf{B}(\mathbb{Z}_{+},\theta)$ is a discrete subspace of $\left(\textbf{B}(\mathbb{Z}_{+},\theta)^0,\tau\right)$. This implies that the set $U_\textsf{0}\setminus V_\textsf{0}$ does not have an accumulation point in $\operatorname{cl}_{\left(\textbf{B}(\mathbb{Z}_{+},\theta)^0,\tau\right)}(U_\textsf{0})$, which contradicts the compactness of $\operatorname{cl}_{\left(\textbf{B}(\mathbb{Z}_{+},\theta)^0,\tau\right)}(U_\textsf{0})$.

$(ii)$ We claim that for every open neighbourhood $U_\textsf{0}$ of zero with the compact closure in $\left(\textbf{B}(\mathbb{Z}_{+},\theta)^0,\tau\right)$ there exists a non-negative integer $i_0$ such that $U_\textsf{0}$ intersects infinitely many sets of the form $\mathbb{Z}_{i_0,j}$. If we assume to the contrary then by $(i)$ there exists an increasing sequence $\left\{i_n\right\}_{n\in\mathbb{N}}$ of positive integers such that $\mathbb{Z}_{i_n,j_n}\cap U_\textsf{0}\neq\varnothing$ for some sequence of non-negative integers $\left\{j_n\right\}_{n\in\mathbb{N}}$. Then for every element $i_p$ of the sequence $\left\{i_n\right\}_{n\in\mathbb{N}}$ there exits a maximum non-negative integer $j_{i_p}$ such that $\mathbb{Z}_{i_p,j_{i_p}}\cap U_\textsf{0}\neq\varnothing$ and $\mathbb{Z}_{i_p,j_{i_p}+k}\cap U_\textsf{0}=\varnothing$ for any positive integer $k$. By the separate continuity of the semigroup operation of $\left(\textbf{B}(\mathbb{Z}_{+},\theta)^0,\tau\right)$ there exists an open neighbourhood $V_\textsf{0}\subseteq U_\textsf{0}$ of zero in $\left(\textbf{B}(\mathbb{Z}_{+},\theta)^0,\tau\right)$ such that $V_\textsf{0}\cdot (0,0_{\mathbb{Z}_{+}},1)\subseteq U_\textsf{0}$. Then the equality
\begin{equation*}
  (i,z,j)\cdot(0,0_{\mathbb{Z}_{+}},1)=(i,z,j+1), \qquad j\in\mathbb{N}_0, \; z\in \mathbb{Z}_{+},
\end{equation*}
implies that there exists a sequence of distinct points $\left\{\left(i_n,z_n,j_{i_n}\right)\right\}_{n\in\mathbb{N}}\subseteq U_\textsf{0}\setminus V_\textsf{0}$ which is a subset of the compactum $\operatorname{cl}_{\left(\textbf{B}(\mathbb{Z}_{+},\theta)^0,\tau\right)}(U_\textsf{0})$. By Corollary~\ref{corollary-4.23}, $\textbf{B}(\mathbb{Z}_{+},\theta)$ is a discrete subspace of $\left(\textbf{B}(\mathbb{Z}_{+},\theta)^0,\tau\right)$, and hence this sequence has not an accumulation point in $\operatorname{cl}_{\left(\textbf{B}(\mathbb{Z}_{+},\theta)^0,\tau\right)}(U_\textsf{0})$ which contradicts the compactness of $\operatorname{cl}_{\left(\textbf{B}(\mathbb{Z}_{+},\theta)^0,\tau\right)}(U_\textsf{0})$.

By the previous part of the proof for every open neighbourhood $U_\textsf{0}$ of zero in $\left(\textbf{B}(\mathbb{Z}_{+},\theta)^0,\tau\right)$ there exists a non-negative integer $i_0$ such that $U_\textsf{0}\cap\mathbb{Z}_{i_0,j}\neq\varnothing$ for infinitely many sets of the form $\mathbb{Z}_{i_0,j}$, $j\in\mathbb{N}\cup\{0\}$. Since the semigroup operation in $\left(\textbf{B}(\mathbb{Z}_{+},\theta)^0,\tau\right)$ is separately continuous, for any non-negative integer $i$ there exists an open neighbourhood $V_\textsf{0}\subseteq U_\textsf{0}$ of zero in $\left(\textbf{B}(\mathbb{Z}_{+},\theta)^0,\tau\right)$ such that $(i_1,0_{\mathbb{Z}_{+}},i_0)\cdot V_\textsf{0}\subseteq U_\textsf{0}$. The last inclusion implies that for every open neighbourhood $U_\textsf{0}$ of zero in $\left(\textbf{B}(\mathbb{Z}_{+},\theta)^0,\tau\right)$ and any non-negative integer $i_1$ we have that $U_\textsf{0}\cap\mathbb{Z}_{i,j}\neq\varnothing$ for infinitely many sets of the form $\mathbb{Z}_{i_1,j}$, $j\in \mathbb{N}_0$. Hence we get that $U_\textsf{0}\cap\mathbb{Z}_{0,j}\neq\varnothing$ for infinitely many sets of the form $\mathbb{Z}_{0,j}$, $j\in \mathbb{N}_0$.

Suppose that there exists an infinite increasing sequence $\left\{j_n\right\}_{n\in\mathbb{N}}$ of non-negative integers such that $U_\textsf{0}\cap\mathbb{Z}_{0,j_n}=\varnothing$ for any element $j_n$ of $\left\{j_n\right\}_{n\in\mathbb{N}}$. Without loss of generality we may assume that the sequence $\left\{j_n\right\}_{n\in\mathbb{N}}$ is maximal, i.e., $U_\textsf{0}\cap\mathbb{Z}_{0,j}\neq\varnothing$ for any non-negative integer $j\notin\left\{j_n\right\}_{n\in\mathbb{N}}$. Then there exists a subsequence $\left\{j_{n_k}\right\}_{k\in\mathbb{N}}$ in $\left\{j_n\right\}_{n\in\mathbb{N}}$ such that $U_\textsf{0}\cap\mathbb{Z}_{0,j_{n_k}}=\varnothing$ and $U_\textsf{0}\cap\mathbb{Z}_{0,j_{n_k}+1}\neq\varnothing$. The separate continuity of the semigroup operation in $\left(\textbf{B}(\mathbb{Z}_{+},\theta)^0,\tau\right)$ implies that there exists an open neighbourhood $V_\textsf{0}\subseteq U_\textsf{0}$ of zero in $\left(\textbf{B}(\mathbb{Z}_{+},\theta)^0,\tau\right)$ such that $(1,0_{\mathbb{Z}_{+}},0)\cdot V_\textsf{0}\subseteq U_\textsf{0}$. Then the equality
\begin{equation*}
(1,0_{\mathbb{Z}_{+}},0)\cdot(i,z,j)=(i+1,z,j), \qquad j\in\mathbb{N}_0, \; z\in \mathbb{Z}_{+},
\end{equation*}
implies that there exists a sequence of distinct points $\left\{\left(0,z_{n_k},j_{n_k}\right)\right\}_{n\in\mathbb{N}}\subseteq U_\textsf{0}\setminus V_\textsf{0}$ which is a subset of the compactum $\operatorname{cl}_{\left(\textbf{B}(\mathbb{Z}_{+},\theta)^0,\tau\right)}(U_\textsf{0})$. By Corollary~\ref{corollary-4.23}, $\textbf{B}(\mathbb{Z}_{+},\theta)$ is a discrete subspace of $\left(\textbf{B}(\mathbb{Z}_{+},\theta)^0,\tau\right)$ and hence this sequence has not an accumulation point in $\operatorname{cl}_{\left(\textbf{B}(\mathbb{Z}_{+},\theta)^0,\tau\right)}(U_\textsf{0})$ which contradicts the compactness of $\operatorname{cl}_{\left(\textbf{B}(\mathbb{Z}_{+},\theta)^0,\tau\right)}(U_\textsf{0})$. This implies that every open neighbourhood $U_\textsf{0}$ of zero in \emph{$\left(\textbf{B}(\mathbb{Z}_{+},\theta)^0,\tau\right)$} intersects almost all sets of the form $\mathbb{Z}_{i_0,j}$, $j\in \mathbb{N}_0$. Again, since the semigroup operation in $\left(\textbf{B}(\mathbb{Z}_{+},\theta)^0,\tau\right)$ is separately continuous, for any non-negative integer $i$ there exists an open neighbourhood $W_\textsf{0}\subseteq U_\textsf{0}$ of zero in $\left(\textbf{B}(\mathbb{Z}_{+},\theta)^0,\tau\right)$ such that $(i,0_{\mathbb{Z}_{+}},i_0)\cdot W_\textsf{0}\subseteq U_\textsf{0}$. The last inclusion implies  assertion $(ii)$.

$(iii)$ Fix an arbitrary open neighbourhood $U_\textsf{0}$ of zero  in $\left(\textbf{B}(\mathbb{Z}_{+},\theta)^0,\tau\right)$. Then by $(ii)$ the set $U_\textsf{0}$ intersects almost all sets of the form $\mathbb{Z}_{i_0,j}$, $j\in \mathbb{N}_0$. By the separate continuity of the semigroup operation of $\left(\textbf{B}(\mathbb{Z}_{+},\theta)^0,\tau\right)$ exists an open neighbourhood $V_\textsf{0}\subseteq U_\textsf{0}$ of zero in $\left(\textbf{B}(\mathbb{Z}_{+},\theta)^0,\tau\right)$ such that $(1,0_{\mathbb{Z}_{+}},1)\cdot V_\textsf{0}\subseteq U_\textsf{0}$. Since $(1,0_{\mathbb{Z}_{+}},1)\cdot(0,z,k)=(0,0_{\mathbb{Z}_{+}},k)$ the inclusion $(1,0_{\mathbb{Z}_{+}},1)\cdot V(\textsf{0})\subseteq U_\textsf{0}$ implies our assertion.

$(iv)$ Fix an arbitrary open neighbourhood $U_\textsf{0}$ of zero in $\left(\textbf{B}(\mathbb{Z}_{+},\theta)^0,\tau\right)$. The separate continuity of the semigroup operation of $\left(\textbf{B}(\mathbb{Z}_{+},\theta)^0,\tau\right)$ implies that for every element $k$ of the additive group of integers $\mathbb{Z}_{+}$ there exists an open neighbourhood $V_\textsf{0}\subseteq U_\textsf{0}$ of zero in $\left(\textbf{B}(\mathbb{Z}_{+},\theta)^0,\tau\right)$ such that  $(0,k,0)\cdot V_\textsf{0}\subseteq U_\textsf{0}$. By statement $(iii)$ the neighbourhood $U_\textsf{0}$ contains almost all elements of the form $(0,0_{\mathbb{Z}_{+}},j)$ and hence the inclusion $(0,k,0)\cdot V_\textsf{0}\subseteq U_\textsf{0}$ implies the statement.

$(v)$ Fix an arbitrary open neighbourhood $U_\textsf{0}$ of zero in $\left(\textbf{B}(\mathbb{Z}_{+},\theta)^0,\tau\right)$. By $(iv)$ the neighbourhood $U_\textsf{0}$  contains almost all subsets of the form $(\mathbb{Z}_{+})_{0,j}$, $j\in \mathbb{N}_0$. Since the semigroup operation in $\left(\textbf{B}(\mathbb{Z}_{+},\theta)^0,\tau\right)$ is separately continuous, for any non-negative integer $i$ there exists an open neighbourhood $V_\textsf{0}\subseteq U_\textsf{0}$ of zero in $\left(\textbf{B}(\mathbb{Z}_{+},\theta)^0,\tau\right)$ such that $(i_0,0_{\mathbb{Z}_{+}},0)\cdot V_\textsf{0}\subseteq U_\textsf{0}$, which implies the assertion.

The proof of assertion $(vi)$ is similar to $(v)$.

$(vii)$ If we assume to the contrary then by items $(v)$ and $(vi)$ there exist an open neighbourhood $U_\textsf{0}$ of zero with the compact closure in $\left(\textbf{B}(\mathbb{Z}_{+},\theta)^0,\tau\right)$ and an increasing sequence $\left\{(i_n,i_n)\right\}_{n\in\mathbb{N}}$ of ordered pairs of positive integers such that $\mathbb{Z}_{i_n,j_n}\cap U_\textsf{0}=\varnothing$ and $\mathbb{Z}_{i_n,j_n+k}\cap U_\textsf{0}=\varnothing$ for any positive integer $k$. The separate continuity of the semigroup operation in $\left(\textbf{B}(\mathbb{Z}_{+},\theta)^0,\tau\right)$ implies that there exists an open neighbourhood $V_\textsf{0}\subseteq U_\textsf{0}$ of zero in $\left(\textbf{B}(\mathbb{Z}_{+},\theta)^0,\tau\right)$ such that $(1,0_{\mathbb{Z}_{+}},0)\cdot V_\textsf{0}\subseteq U_\textsf{0}$. Then the equality
\begin{equation*}
(1,0_{\mathbb{Z}_{+}},0)\cdot(i,z,j)=(i+1,z,j), \qquad j\in\mathbb{N}_0, \; z\in \mathbb{Z}_{+},
\end{equation*}
implies that there exists a sequence of distinct points $\left\{\left(0,z_{n_k},j_{n_k}\right)\right\}_{n\in\mathbb{N}}\subseteq U_\textsf{0}\setminus V_\textsf{0}$ which is a subset of the compactum $\operatorname{cl}_{\left(\textbf{B}(\mathbb{Z}_{+},\theta)^0,\tau\right)}(U_\textsf{0})$. By Corollary~\ref{corollary-4.23}, $\textbf{B}(\mathbb{Z}_{+},\theta)$ is a discrete subspace of $\left(\textbf{B}(\mathbb{Z}_{+},\theta)^0,\tau\right)$ and hence this sequence has not an accumulation point in $\operatorname{cl}_{\left(\textbf{B}(\mathbb{Z}_{+},\theta)^0,\tau\right)}(U_\textsf{0})$ which contradicts the compactness of $\operatorname{cl}_{\left(\textbf{B}(\mathbb{Z}_{+},\theta)^0,\tau\right)}(U_\textsf{0})$.

$(viii)$ Suppose to the contrary that there exists an open neighbourhood $U_\textsf{0}$ of zero with the compact closure in $\left(\textbf{B}(\mathbb{Z}_{+},\theta)^0,\tau\right)$ such that the set $\textbf{B}(\mathbb{Z}_{+},\theta)^0\setminus U_\textsf{0}$ is infinite. By assertion $(vii)$ there exist non-negative integers $i_0$ and $j_0$ such that the set $\mathbb{Z}_{i_0,j_0}\setminus U_\textsf{0}$ is infinite and the set $\mathbb{Z}_{i_0,j_0+1}\setminus U_\textsf{0}$ is finite. The separate continuity of the semigroup operation in $\left(\textbf{B}(\mathbb{Z}_{+},\theta)^0,\tau\right)$ implies that there exists an open neighbourhood $V_\textsf{0}\subseteq U_\textsf{0}$ of zero in $\left(\textbf{B}(\mathbb{Z}_{+},\theta)^0,\tau\right)$ such that $(1,0_{\mathbb{Z}_{+}},0)\cdot V_\textsf{0}\subseteq U_\textsf{0}$. Then the equality
\begin{equation*}
(1,0_{\mathbb{Z}_{+}},0)\cdot(i,z,j)=(i+1,z,j), \qquad j\in\mathbb{N}_0, \; z\in \mathbb{Z}_{+},
\end{equation*}
implies that there exists a sequence of distinct points $\left\{\left(i_0,z_{n},j_0+1\right)\right\}_{n\in\mathbb{N}}\subseteq \mathbb{Z}_{i_0,j_0+1}\setminus V_\textsf{0}\subseteq U_\textsf{0}\setminus V_\textsf{0}$ which is a subset of the compactum $\operatorname{cl}_{\left(\textbf{B}(\mathbb{Z}_{+},\theta)^0,\tau\right)}(U_\textsf{0})$. By Corollary~\ref{corollary-4.23}, $\textbf{B}(\mathbb{Z}_{+},\theta)$ is a discrete subspace of $\left(\textbf{B}(\mathbb{Z}_{+},\theta)^0,\tau\right)$ and hence this sequence has not an accumulation point in $\operatorname{cl}_{\left(\textbf{B}(\mathbb{Z}_{+},\theta)^0,\tau\right)}(U_\textsf{0})$ which contradicts the compactness of $\operatorname{cl}_{\left(\textbf{B}(\mathbb{Z}_{+},\theta)^0,\tau\right)}(U_\textsf{0})$.
\end{proof}

The following example shows that the Reilly semigroup $\textbf{B}(\mathbb{Z}_{+},\theta)^0$ with a non-annihilating homomorphism $\theta$ admits the structure of a Hausdorff compact semitopological inverse semigroup with continuous inversion.

\begin{example}\label{example-4.27}
Let $\theta\colon \mathbb{Z}_{+}\to \mathbb{Z}_{+}$ be an arbitrary non-annihilating homomorphism. We define a topology $\tau_{\textsf{AC}}$ on the semigroup $\textbf{B}(\mathbb{Z}_{+},\theta)^0$ in the following way:
\begin{itemize}
  \item[$(i)$] all non-zero elements of the semigroup $\textbf{B}(\mathbb{Z}_{+},\theta)^0$ are isolated points in $\left(\textbf{B}(\mathbb{Z}_{+},\theta)^0,\tau_{\textsf{AC}}\right)$;
  \item[$(ii)$] the family $\mathscr{B}_{\textsf{AC}}(\textsf{0})=\left\{U\colon \textsf{0}\in U \hbox{~and~} \textbf{B}(\mathbb{Z}_{+},\theta)^0\setminus U \hbox{~is finite}\right\}$ is a base of the topology $\tau_{\textsf{AC}}$ at zero $\textsf{0}\in\textbf{B}(\mathbb{Z}_{+},\theta)^0$.
\end{itemize}
It is obvious that the space $\left(\textbf{B}(\mathbb{Z}_{+},\theta)^0,\tau_{\textsf{AC}}\right)$ is the one-point Alexandroff compactification of the discrete space $\textbf{B}(\mathbb{Z}_{+},\theta)$ with the remainder $\{\textsf{0}\}$. Proposition~\ref{proposition-4.21} implies that both equations $\alpha\cdot \chi=\beta$ and $\chi\cdot\gamma=\delta$ have finitely many solutions in $\textbf{B}(\mathbb{Z}_{+},\theta)$. Then for every open neighbourhood $U_\textsf{0}$ of zero in $\left(\textbf{B}(\mathbb{Z}_{+},\theta)^0,\tau_{\textsf{AC}}\right)$ and any non-zero element $(i,z,j)\in\textbf{B}(\mathbb{Z}_{+},\theta)^0$ the sets
\begin{equation*}
  \left\{x\in\textbf{B}(\mathbb{Z}_{+},\theta)^0\colon x\cdot(i,z,j)\in\textbf{B}(\mathbb{Z}_{+},\theta)^0\setminus U_\textsf{0} \right\} \quad \hbox{and}\quad \left\{x\in\textbf{B}(\mathbb{Z}_{+},\theta)^0\colon(i,z,j)\cdot x\in\textbf{B}(\mathbb{Z}_{+},\theta)^0\setminus U_\textsf{0} \right\}
\end{equation*}
are finite. Hence for any open neighbourhood $U_\textsf{0}$ of zero in $\left(\textbf{B}(\mathbb{Z}_{+},\theta)^0,\tau_{\textsf{AC}}\right)$ and any non-zero element $(i,z,j)\in\textbf{B}(\mathbb{Z}_{+},\theta)^0$ there exists a neighbourhood $V_\textsf{0}$  of zero such that
\begin{equation*}
  (i,z,j)\cdot V_\textsf{0}\subseteq U_\textsf{0} \qquad \hbox{and} \qquad V_\textsf{0}\cdot(i,z,j)\subseteq U_\textsf{0},
\end{equation*}
which implies that the semigroup operation in $\left(\textbf{B}(\mathbb{Z}_{+},\theta)^0,\tau_{\textsf{AC}}\right)$ is separately continuous. It is easy to see that the space $\left(\textbf{B}(\mathbb{Z}_{+},\theta)^0,\tau_{\textsf{AC}}\right)$ is Hausdorff and inversion in $\left(\textbf{B}(\mathbb{Z}_{+},\theta)^0,\tau_{\textsf{AC}}\right)$ is continuous.
\end{example}

Corollary~\ref{corollary-4.25}, Proposition~\ref{proposition-4.26} and Example~\ref{example-4.27} imply the following dichotomy for a locally compact $T_1$-semitopological semigroup $\left(\textbf{B}(\mathbb{Z}_{+},\theta)^0,\tau\right)$ with a non-an\-ni\-hi\-la\-ting homomorphism $\theta$.

\begin{theorem}\label{theorem-4.28}
Let $\theta$ be an arbitrary non-annihilating homomorphism and  \emph{$\left(\textbf{B}(\mathbb{Z}_{+},\theta)^0,\tau\right)$} be a locally compact $T_1$-semitopological semigroup. Then either \emph{$\left(\textbf{B}(\mathbb{Z}_{+},\theta)^0,\tau\right)$} is topologically isomorphic to $\left(\textbf{B}(\mathbb{Z}_{+},\theta)^0,\tau_{\textsf{AC}}\right)$ or $\tau$ is discrete.
\end{theorem}

Theorem~\ref{theorem-4.28} implies the following corollary.

\begin{corollary}\label{corollary-4.29}
Every locally compact $T_1$-topological semigroup \emph{$\left(\textbf{B}(\mathbb{Z}_{+},\theta)^0,\tau\right)$} with a non-annihilating homomorphism $\theta$ is the  discrete space.
\end{corollary}


\section{On the closure of the discrete semigroup $\textbf{B}(\mathbb{Z}_{+},\theta)$ with a non-an\-ni\-hi\-la\-ting homomorphism $\theta$}\label{section-5}

The following proposition extends Theorem~I.3 from \cite{Eberhart-Selden-1969}.

\begin{proposition}\label{proposition-5.1}
Let $\theta$ be an arbitrary non-annihilating homomorphism and \emph{$\textbf{B}(\mathbb{Z}_{+},\theta)$} is a discrete dense subsemigroup of a semitopological monoid $S$ such that  \emph{$I=S\setminus\textbf{B}(\mathbb{Z}_{+},\theta)\neq\varnothing$}. Then $I$ is a two-sided ideal of $S$.
\end{proposition}

\begin{proof}
Fix an arbitrary element $y\in I$. If $xy=z\notin I$ for some $x\in\textbf{B}(\mathbb{Z}_{+},\theta)$, then there exists an open neighbourhood $U(y)$ of the point $y$ in the space $S$ such that $\{x\}\cdot U(y)=\{z\}\subset\textbf{B}(\mathbb{Z}_{+},\theta)$. The neighbourhood $U(y)$ contains infinitely many elements of the semigroup $\textbf{B}(\mathbb{Z}_{+},\theta)$. This contradicts Proposition~\ref{proposition-4.21}. The obtained contradiction implies that $xy\in I$ for all $x\in \textbf{B}(\mathbb{Z}_{+},\theta)$ and $y\in I$. The proof of the statement that $yx\in I$ for all $x\in \textbf{B}(\mathbb{Z}_{+},\theta)$ and $y\in I$ is similar.

Suppose to the contrary that $xy=w\notin I$ for some $x,y\in I$. Then $w\in\textbf{B}(\mathbb{Z}_{+},\theta)$ and the separate continuity of the semigroup operation in $S$ implies that there exist open neighbourhoods $U(x)$ and $U(y)$ of the points $x$ and $y$ in $S$, respectively, such that $\{x\}\cdot U(y)=\{w\}$ and $U(x)\cdot \{y\}=\{w\}$. Since both neighbourhoods $U(x)$ and $U(y)$ contain infinitely many elements of the semigroup $\textbf{B}(\mathbb{Z}_{+},\theta)$, both equalities $\{x\}\cdot U(y)=\{w\}$ and $U(x)\cdot \{y\}=\{w\}$ contradict mentioned above Proposition~\ref{proposition-4.21}. The obtained contradiction implies that $xy\in I$.
\end{proof}

Later we need the following trivial lemma, which follows from separate continuity of the semigroup operation in semitopological semigroups.

\begin{lemma}\label{lemma-5.2}
Let $S$ be a semitopological semigroup and $I$ be a compact ideal in $S$. Then the Rees-quotient semigroup $S/I$ with the quotient topology is a Hausdorff semitopological semigroup.
\end{lemma}

\begin{theorem}\label{theorem-5.3}
Let $\theta$ be an arbitrary non-annihilating homomorphism and \emph{$(\textbf{B}(\mathbb{Z}_{+},\theta)_I,\tau)$} be a locally compact semitopological semigroup, where \emph{$\textbf{B}(\mathbb{Z}_{+},\theta)_I=\textbf{B}(\mathbb{Z}_{+},\theta)\sqcup I$} and $I$ is a compact ideal of \emph{$\textbf{B}(\mathbb{Z}_{+},\theta)_I$}. Then either \emph{$(\textbf{B}(\mathbb{Z}_{+})_I,\tau)$} is compact or $I$ is an open subset of \emph{$(\textbf{B}(\mathbb{Z}_{+})_I,\tau)$}.
\end{theorem}

\begin{proof}
Since $I$ is a compact ideal of $\textbf{B}(\mathbb{Z}_{+},\theta)_I$, Corollary~3.3.10 of \cite{Engelking-1989} implies that $\textbf{B}(\mathbb{Z}_{+},\theta)$ is a locally compact subspace of $(\textbf{B}(\mathbb{Z}_{+},\theta)_I,\tau)$, and hence by Corollary~\ref{corollary-4.23}, $\textbf{B}(\mathbb{Z}_{+},\theta)$ is the discrete space.

Suppose that $I$ is not open. By Lemma~\ref{lemma-5.2} the Rees-quotient semigroup $\textbf{B}(\mathbb{Z}_{+},\theta)_I/I$ with the quotient topology $\tau_{\operatorname{\textsf{q}}}$ is a semitopological semigroup. Let $\pi\colon \textbf{B}(\mathbb{Z}_{+},\theta)_I\to \textbf{B}(\mathbb{Z}_{+},\theta)_I/I$ be the natural homomorphism which is a quotient map. It is obvious that the Rees-quotient semigroup $\textbf{B}(\mathbb{Z}_{+},\theta)_I/I$ is isomorphic to the semigroup $\textbf{B}(\mathbb{Z}_{+},\theta)^0$ and the image $\pi(I)$ corresponds zero $\textsf{0}$ of $\textbf{B}(\mathbb{Z}_{+},\theta)^0$. Now we shall show that the natural homomorphism $\pi\colon \textbf{B}(\mathbb{Z}_{+},\theta)_I\to \textbf{B}(\mathbb{Z}_{+},\theta)_I/I$ is a hereditarily quotient map. Since $\pi(\textbf{B}(\mathbb{Z}_{+},\theta))$ is a discrete subspace of $(\textbf{B}(\mathbb{Z}_{+},\theta)_I/I,\tau_{\operatorname{\textsf{q}}})$, it is sufficient to show that for every open neighbourhood $U(I)$ of the ideal $I$ in the space $(\textbf{B}(\mathbb{Z}_{+},\theta)_I,\tau)$ we have that the image  $\pi(U(I))$ is an open neighbourhood of the zero $\textsf{0}$ in the space $(\textbf{B}(\mathbb{Z}_{+},\theta)_I/I,\tau_{\operatorname{\textsf{q}}})$. Indeed, $\textbf{B}(\mathbb{Z}_{+},\theta)_I\setminus U(I)$ is an open-and-closed subset of $(\textbf{B}(\mathbb{Z}_{+},\theta)_I,\tau)$, because by Corollary~\ref{corollary-4.23} the elements of the semigroup $\textbf{B}(\mathbb{Z}_{+},\theta)$ are isolated point of $(\textbf{B}(\mathbb{Z}_{+},\theta)_I,\tau)$. Also, since the restriction $\pi|_{\textbf{B}(\mathbb{Z}_{+},\theta)}\colon \textbf{B}(\mathbb{Z}_{+},\theta)\to \pi(\textbf{B}(\mathbb{Z}_{+},\theta))$ of the natural homomorphism $\pi\colon \textbf{B}(\mathbb{Z}_{+},\theta)_I\to \textbf{B}(\mathbb{Z}_{+},\theta)_I/I$ is one-to-one,  $\pi(\textbf{B}(\mathbb{Z}_{+},\theta)_I\setminus U(I))$ is an open-and-closed subset of $(\textbf{B}(\mathbb{Z}_{+},\theta)_I/I,\tau_{\operatorname{\textsf{q}}})$. So $\pi(U(I))$ is an open neighbourhood of the zero $\textsf{0}$ of the semigroup $(\textbf{B}(\mathbb{Z}_{+},\theta)_I/I,\tau_{\operatorname{\textsf{q}}})$, and hence the natural homomorphism $\pi\colon \textbf{B}(\mathbb{Z}_{+},\theta)_I\to \textbf{B}(\mathbb{Z}_{+},\theta)_I/I$ is a hereditarily quotient map. Since $I$ is a compact ideal of the semitopological semigroup $(\textbf{B}(\mathbb{Z}_{+},\theta)_I,\tau)$, $\pi^{-1}(y)$ is a compact subset of $(\textbf{B}(\mathbb{Z}_{+},\theta)_I,\tau)$ for every $y\in \textbf{B}(\mathbb{Z}_{+},\theta)_I/I$. By Din' N'e T'ong's Theorem (see \cite{Din'-N'e-T'ong-1963} or \cite[3.7.E]{Engelking-1989}), $(\textbf{B}(\mathbb{Z}_{+},\theta)_I/I,\tau_{\operatorname{\textsf{q}}})$ is a Hausdorff locally compact space. If $I$ is not open then by Theorem~\ref{theorem-4.28} the semitopological semigroup $(\textbf{B}(\mathbb{Z}_{+},\theta)_I/I,\tau_{\operatorname{\textsf{q}}})$ is topologically isomorphic to $(\textbf{B}(\mathbb{Z}_{+},\theta)^0,\tau_{\operatorname{\textsf{Ac}}})$ and hence it is compact.
Next we shall prove that the space $(\textbf{B}(\mathbb{Z}_{+},\theta)_I,\tau)$ is compact. Let $\mathscr{U}=\left\{U_\alpha\colon\alpha\in\mathscr{I}\right\}$ be an arbitrary open cover of $(\textbf{B}(\mathbb{Z}_{+},\theta)_I,\tau)$. Since $I$ is compact, there exist $U_{\alpha_1},\ldots,U_{\alpha_n}\in\mathscr{U}$ such that $I\subseteq U_{\alpha_1}\cup\cdots\cup U_{\alpha_n}$. Put $U=U_{\alpha_1}\cup\cdots\cup U_{\alpha_n}$. Then $\textbf{B}(\mathbb{Z}_{+},\theta)_I\setminus U$ is an open-and-closed subset of $(\textbf{B}(\mathbb{Z}_{+},\theta)_I,\tau)$. Also, since the restriction $\pi|_{\textbf{B}(\mathbb{Z}_{+},\theta)}\colon \textbf{B}(\mathbb{Z}_{+},\theta)\to \pi(\textbf{B}(\mathbb{Z}_{+},\theta))$ of the natural homomorphism $\pi\colon \textbf{B}(\mathbb{Z}_{+},\theta)_I\to \textbf{B}(\mathbb{Z}_{+},\theta)_I/I$ is one-to-one, $\pi(\textbf{B}(\mathbb{Z}_{+},\theta)_I\setminus U(I))$ is an open-and-closed subset of $(\textbf{B}(\mathbb{Z}_{+},\theta)_I/I,\tau_{\operatorname{\textsf{q}}})$, and hence the image $\pi(\textbf{B}(\mathbb{Z}_{+},\theta)_I\setminus U(I))$ is finite, because the semigroup $(\textbf{B}(\mathbb{Z}_{+},\theta)_I/I,\tau_{\operatorname{\textsf{q}}})$ is compact. Thus, the set $\textbf{B}(\mathbb{Z}_{+},\theta)_I\setminus U$ is finite and hence the space $(\textbf{B}(\mathbb{Z}_{+},\theta)_I,\tau)$ is compact as well.
\end{proof}

Theorem~\ref{theorem-5.3} implies the following corollary.

\begin{corollary}\label{corollary-5.4}
Let $\theta$ be an arbitrary non-annihilating homomorphism and \emph{$(\textbf{B}(\mathbb{Z}_{+},\theta)_I,\tau)$} be a locally compact topological semigroup, where \emph{$\textbf{B}(\mathbb{Z}_{+},\theta)_I=\textbf{B}(\mathbb{Z}_{+},\theta)\sqcup I$} and $I$ is a compact ideal of $\textbf{B}(\mathbb{Z}_{+},\theta)_I$. Then $I$ is an open subset of \emph{$(\textbf{B}(\mathbb{Z}_{+})_I,\tau)$}.
\end{corollary}

\section*{Acknowledgements}

The author acknowledges Alex Ravsky and the referee for their important comments and suggestions.


\begin{thebibliography}{00}

\bibitem{Andersen-1952}
O.~Andersen,
{\em Ein Bericht \"{u}ber die Struktur abstrakter Halbgruppen},
PhD Thesis, Hamburg, 1952.

\bibitem{Anderson-Hunter-Koch-1965}
L.~W.~Anderson, R.~P.~Hunter, and R.~J.~Koch,
\emph{Some results on stability in semigroups}.
Trans. Amer. Math. Soc. {\bf 117} (1965), 521--529.

\bibitem{Arkhangelskii-1963}
A.~V.~Arkhangel'ski\v{\i},
\emph{Bicompact sets and the topology of spaces},
Dokl. Akad. Nauk SSSR \textbf{150} (1963), 9--12 (in Russian); English version in: Soviet Math. Dokl. \textbf{4} (1963), 561--564.

\bibitem{Banakh-Dimitrova-Gutik-2009}
T.~O. Banakh, S.~Dimitrova, and O.~V. Gutik,
\emph{The Rees-Suschkewitsch Theorem for simple topological semigroups},
Mat. Stud. \textbf{31}:2 (2009), 211--218.

\bibitem{Banakh-Dimitrova-Gutik-2010}
T.~Banakh, S.~Dimitrova, and O.~Gutik,
\emph{Embedding the bicyclic semigroup into countably compact topological semigroups},
Topology Appl. \textbf{157}:18 (2010), 2803--2814.

\bibitem{Bardyla-2016}
S. Bardyla,
\emph{Classifying locally compact semitopological polycyclic monoids},
Math. Bull. Shevchenko Sci. Soc. \textbf{13} (2016), 31--38.

\bibitem{Bardyla-Gutik-2016}
S. Bardyla and O. Gutik,
\emph{On a semitopological polycyclic monoid},
Algebra Discrete Math. \textbf{21}:2 (2016) 163--183.

\bibitem{Bertman-West-1976}
M.~O.~Bertman and T.~T.~West,
{\it Conditionally compact bicyclic
semitopological semigroups},
Proc. Roy. Irish Acad. {\bf A76}:21--23 (1976), 219--226.

\bibitem{Bruck-1958}
R. H. Bruck,
\emph{A survey of binary systems}, Springer, Berlin-G\"{o}ttingen-Heidelberg, VII,
Ergebn. Math. \textbf{20} (1958), 185S.

\bibitem{Carruth-Hildebrant-Koch-1983-1986}
J.~H.~Carruth, J.~A.~Hildebrant, and R.~J.~Koch,
\emph{The Theory of Topological Semigroups}, Vols I and
II, Marcell Dekker, Inc., New York and Basel, 1983 and 1986.

\bibitem{Chuchman-Gutik-2010}
I.~Chuchman and O.~Gutik,
\emph{Topological monoids of almost monotone injective co-finite partial selfmaps of the set of positive integers},
Carpathian Math. Publ. \textbf{2}:1 (2010), 119--132.

\bibitem{Chuchman-Gutik-2011}
I.~Chuchman and O.~Gutik,
\emph{On monoids of injective partial selfmaps almost everywhere the identity},
Demonstr. Math. \textbf{44}:4 (2011), 699--722.


\bibitem{Clifford-Preston-1961-1967}
A.~H.~Clifford and G.~B.~Preston,
\emph{The Algebraic Theory of Semigroups}, Vols. I and II,
Amer. Math. Soc. Surveys {\bf 7}, Providence, R.I.,  1961 and  1967.

\bibitem{Din'-N'e-T'ong-1963}
Din' N'e T'ong,
\emph{Preclosed mappings and A.~D.~Ta{\u{\i}}manov's theorem},
Dokl. Akad. Nauk SSSR \textbf{152} (1963), 525--528 (in Russian); English version in: Soviet Math. Dokl. \textbf{4} (1963), 1335--1338.


\bibitem{Eberhart-Selden-1969}
C.~Eberhart and J.~Selden,
{\em On the closure of the bicyclic semigroup},
Trans. Amer. Math. Soc. {\bf 144} (1969), 115--126.


\bibitem{Ellis-1957}
R. Ellis,
\emph{Locally compact transformation groups},
Duke Math. J. \textbf{24}:2 (1957), 119--125.


\bibitem{Engelking-1989}
R.~Engelking,
\emph{General Topology},
2nd ed., Heldermann, Berlin, 1989.

\bibitem{Fihel-Gutik-2011}
I. Fihel and O. Gutik,
\emph{On the closure of the extended bicyclic semigroup},
Carpathian Math. Publ. \textbf{3}:2 (2011) 131--157.

\bibitem{Gutik-1994}
O. V. Gutik,
\emph{Embedding of topological semigroups into simple semigroups},
Mat. Stud. \textbf{3} (1994), 10--14  (in Russian).

\bibitem{Gutik-1996}
O. V. Gutik,
\emph{Any topological semigroup topologically isomorphically embeds into a simple path-connected topological semigroup},
Algebra and Topology, Lviv Univ. Press (1996), 65--73 (in Ukrainian).

\bibitem{Gutik-1997}
O. V. Gutik,
\emph{On a coarsening of a direct sum topology on the Bruck semigroup},
Visn. L'viv. Univ., Ser. Mekh.-Mat. \textbf{47} (1997), 17--21 (in Ukrainian).

\bibitem{Gutik-2015}
O. Gutik,
\emph{On the dichotomy of a locally compact semitopological bicyclic monoid with adjoined zero},
Visn. L'viv. Univ., Ser. Mekh.-Mat. \textbf{80} (2015), 33--41.

\bibitem{Gutik-Maksymyk-2016}
O. Gutik and K. Maksymyk,
\emph{On semitopological interassociates of the bicyclic monoid},
Visn. L'viv. Univ., Ser. Mekh.-Mat. \textbf{82 }(2016), 98--108.

\bibitem{Gutik-Maksymyk-2016??}
O. Gutik and K. Maksymyk,
\emph{On semitopological bicyclic extensions of linearly ordered groups},
Mat. Metody Fiz.-Mekh. Polya 59:4 (2016) (to appear) (arXiv:1608.00959).





\bibitem{Gutik-Pavlyk-2009}
O. V. Gutik and K. P. Pavlyk,
\emph{Bruck-Reilly extension of a semitopological semigroups},
Applied Problems of Mech. and Math. \textbf{7} (2009), 66--72.

\bibitem{Gutik-Pozdnyakova-2014}
O.~Gutik and I.~Pozdnyakova,
\emph{On monoids of monotone injective partial selfmaps of $L_n\times_{\operatorname{lex}}\mathbb{Z}$ with co-finite domains and images},
Algebra Discrete Math. \textbf{17}:2 (2014) 256--279.


\bibitem{Gutik-Repovs-2007}
O.~Gutik and D.~Repov\v{s},
\emph{On countably compact $0$-simple topological inverse semigroups},
Semigroup Forum \textbf{75}:2 (2007), 464--469.

\bibitem{Gutik-Repovs-2011}
O.~Gutik, D.~Repov\v{s},
\emph{Topological monoids of monotone, injective partial selfmaps of $\mathbb{N}$ having cofinite domain and image},
Stud. Sci. Math. Hungar. \textbf{48}:3 (2011), 342--353.

\bibitem{Gutik-Repovs-2012}
O.~Gutik, D.~Repov\v{s},
\emph{On monoids of injective partial selfmaps of integers with cofinite domains and images},
Georgian Math. J. \textbf{19}:3 (2012), 511--532.


\bibitem{Haworth-McCoy-1977}
R.~C.~Haworth and R.~A.~McCoy,
\emph{Baire spaces},
Diss. Math.  \textbf{141}  (1977), 73 pp.


\bibitem{Hildebrant-Koch-1988}
J.~A.~Hildebrant and R.~J.~Koch,
{\it Swelling actions of $\Gamma$-compact semigroups}, Semigroup
Forum {\bf 33} (1986), 65--85.

\bibitem{Hofmann-1988}
K.~H.~Hofmann,
\emph{Locally compact semigroups in which a subgroup with compact complement is dense},
Trans. Amer. Math. Soc. \textbf{106}:1 (1963), 19--51.

\bibitem{Lallement-Petrich-1969}
G. Lallement and M. Petrich,
\emph{A generalization of the Rees theorem in semigroups},
Acta Sci. Math. (Szeged)  \textbf{30}:1--2 (1969),  113--132.


\bibitem{McAlister-1974}
D. B. McAlister,
\emph{$0$-bisimple inverse semigroups},
Proc. London Math. Soc. (3) \textbf{28} (1974), 193--221.

\bibitem{McDougle-1958}
P.~McDougle,
\emph{A theorem on quasi-compact mappings},
Proc. Amer. Math. Soc. \textbf{9}:3 (1958), 474--477.

\bibitem{McDougle-1959}
P.~McDougle,
\emph{Mapping and space relations},
Proc. Amer. Math. Soc. \textbf{10}:2 (1959), 320--323.

\bibitem{Mesyan-Mitchell-Morayne-Peresse-2016}
Z. Mesyan, J. D. Mitchell, M. Morayne, and Y. H. P\'{e}resse,
\emph{Topological graph inverse semigroups},
Topology Appl. \textbf{208} (2016), 106--126.


\bibitem{Moore-1925}
R.~L.~Moore,
\emph{Concerning upper semi-continuous collections of continua},
Trans. Amer. Math. Soc. \textbf{27} (1925), 416--428.

\bibitem{Munn-1970}
W. D. Munn,
\emph{$0$-bisimple inverse semigroups},
J. Algebra \textbf{15}:4 (1970), 570--588.

\bibitem{Munn-Reilly-1966}
W. D. Munn and N. R. Reilly,
\emph{Congruences on a bisimple $\omega$-semigroup},
Proc. Glasg. Math. Assoc. \textbf{7} (1966), 184--192.

\bibitem{Petrich-1984}
M.~Petrich,
\emph{Inverse Semigroups},
John Wiley $\&$ Sons, New York, 1984.

\bibitem{Reilly-1966}
N. R. Reilly,
\emph{Bisimple  $\omega$-semigroups},
Proc. Glasgow Math. Assoc. \textbf{7} (1966), 160--169.

\bibitem{Ruppert-1984}
W.~Ruppert,
\emph{Compact Semitopological Semigroups: An Intrinsic Theory},
Lect. Notes Math., \textbf{1079}, Springer, Berlin, 1984.

\bibitem{Selden-1975}
A. A. Selden,
\emph{Bisimple $\omega$-semigroups in the locally compact setting},
Bogazici Univ. J. Sci. Math. \textbf{3} (1975),  15--77.


\bibitem{Selden-1976}
A. A. Selden,
\emph{On the closure of bisimple $\omega$-semigroup},
Semigroup Forum \textbf{12} (1976), 373--379.

\bibitem{Selden-1977}
A. A. Selden,
\emph{The kernel of the determining endomorphism of a bisimple $\omega$-semigroup},
Semigroup Forum \textbf{14}:3 (1977), 265--271.

\bibitem{Vainstein-1947}
I.~A.~Va{\u{\i}}n\v{s}te{\u{\i}}n,
\emph{On closed mappings of metrc spaces},
Dokl. Akad. Nauk SSSR \textbf{57} (1947), 319--321 (in Russian).

\bibitem{Warne-1966}
R. J. Warne,
\emph{A class of bisimple inverse semigroups},
Pacif. J. Math. \textbf{18} (1966), 563--577.

\end{thebibliography}
\end{document}